\documentclass[]{amsart}

    \usepackage{amsmath, amssymb,amsthm}
    \usepackage{tikz}
    \usepackage{mathrsfs} 
    \usepackage{xcolor}
    \usepackage{enumitem}
    \usepackage{subcaption}
    \usepackage{hyperref}  
    \usepackage{changepage}
    \usepackage{soul} 
    \usepackage{xspace} 
    \usepackage[utf8]{inputenc} 

    \usepackage[backend=bibtex, style=alphabetic,sortcites]{biblatex}

    \usepackage{mathtools}
    \mathtoolsset{showonlyrefs}

    \setenumerate{label=(\arabic*)}
    \numberwithin{equation}{section}
    \definecolor{robin}{HTML}{158c89}
    \definecolor{pinkrobin}{HTML}{8c1554}
    \definecolor{redrobin}{HTML}{8c1518}
    \definecolor{greenrobin}{HTML}{158c4e}
    \definecolor{bluerobin}{HTML}{15548c}
    \definecolor{hotpink}{HTML}{e6194B}
    \definecolor{orange}{HTML}{f58231}
    \definecolor{green}{HTML}{3cb44b}
    \definecolor{cyan}{HTML}{42d4f4}
    \definecolor{indigo}{HTML}{4363d8}
    \definecolor{purple}{HTML}{911eb4}
    \definecolor{magenta}{HTML}{f032e6}
    \colorlet{emphcolor}{blue}
    \hypersetup{colorlinks,citecolor=pinkrobin,linkcolor=.}
  
    \newcommand{\eref}[1]{\textcolor{       black}{\hyperref[#1]{\eqref{#1}}}}
    \newcommand{\secref}[1]{{\textcolor{    black}{\hyperref[#1]{Section~\ref{#1}}}}}
    \newcommand{\subsecref}[1]{{\textcolor{ black}{\hyperref[#1]{Subsection~\ref{#1}}}}}
    \newcommand{\subsubsecref}[1]{{\textcolor{ black}{\hyperref[#1]{Subsubsection~\ref{#1}}}}}
    \newcommand{\thmref}[1]{{\textcolor{    black}{\hyperref[#1]{Theorem~\ref{#1}}}}}
    \newcommand{\propref}[1]{{\textcolor{   black}{\hyperref[#1]{Proposition~\ref{#1}}}}}
    \newcommand{\lemref}[1]{{\textcolor{    black}{\hyperref[#1]{Lemma~\ref{#1}}}}}
    \newcommand{\corref}[1]{{\textcolor{    black}{\hyperref[#1]{Corollary~\ref{#1}}}}}
    \newcommand{\figref}[1]{{\textcolor{    black}{\hyperref[#1]{Figure~\ref{#1}}}}}
    \newcommand{\defref}[1]{{\textcolor{    black}{\hyperref[#1]{Definition~\ref{#1}}}}}
    \newcommand{\itemref}[1]{{\textcolor{   black}{\hyperref[#1]{Item~\ref{#1}}}}}
    \newcommand{\remref}[1]{{\textcolor{    black}{\hyperref[#1]{Remark~\ref{#1}}}}}
    \newcommand{\exampleref}[1]{{\textcolor{black}{\hyperref[#1]{Example~\ref{#1}}}}}
    \newcommand{\chapref}[1]{{\textcolor{    black}{\hyperref[#1]{Chapter~\ref{#1}}}}}

    \newcommand{\emphnew}[1]{\textcolor{    emphcolor}{\emph{#1}}}
    \newcommand{\emphnewmath}[1]{\textcolor{    emphcolor}{#1}}

    \newcommand{\todonote}[1]{}

    \newcommand{\lt}{<}
    \newcommand{\gt}{>}
    \newcommand{\braceme}[1]{[#1]}

    \newcommand{\bracketbox}[2]{\left\{ \parbox[c][][c]{#1}{\centering #2}\right\}}

    \newcommand{\includegraphicswide}[1]{\makebox[\textwidth][c]{\includegraphics[width=.80\paperwidth]{#1}}}
    \newcommand{\nomargins}[1]{\makebox[\textwidth][c]{#1}}

    \newcommand{\lpadtext}[1]{\; \text{#1}}
    
    \newcommand{\padtext}[1]{\; \text{#1} \;}

    \DeclareMathOperator{\given}{\; \vert \;}
    
    \DeclareMathOperator{\isoto}{\xrightarrow{\sim}}

    \DeclareMathOperator{\image}{Im}
    \DeclareMathOperator{\nextnode}{nextnode}
    \DeclareMathOperator{\nextlive}{nextlive}

    \newlength{\offsetpage}
    {\end{adjustwidth}}

    \newcommand{\caseif}               [2]{\textbf{\textcolor{black}{Case #1:}} \ul{#2}}

    \newcommand{\specialword} [1]{\textbf{#1}}
    \newcommand{\emphspecial} [1]{\emphnew{\specialword{#1}}}

    \newcommand{\lesslex}   {\lt_{\text{lex}}}
    \newcommand{\lesseqlex} {\leq_{\text{lex}}}

    \newcommand{\inefunc}               {\delta}
    \newcommand{\pairsilessthanj}       {\binom{[n]}{2}}
    \newcommand{\epsvector}             {\boldsymbol{\epsilon}}
    \newcommand{\edgeweight}            {\omega_{\epsvector}}

      \newcommand{\catconditionone}       {(1)\xspace}
      \newcommand{\catconditiontwo}       {(2)\xspace}
      \newcommand{\catconditionthree}     {(3)\xspace}
      \newcommand{\shiconditionone}       {(1)\xspace}
      \newcommand{\shiconditiontwo}       {(2)\xspace}
      \newcommand{\shiconditionthree}     {(3)\xspace}

      \newcommand{\compl}                   {c}
      \newcommand{\live}                    {live\xspace}
      \newcommand{\specialf}                {\boldsymbol{f}}
      \newcommand{\mcatmap}                 {\Phi_{\textnormal{Cat}_n}}
      \newcommand{\mcatmapinv}              {\mcatmap^{-1}}
      \newcommand{\mshimap}                 {\Phi_{\textnormal{Shi}_n}}
      \newcommand{\mcatconditionone}        {(C1)\xspace}
      \newcommand{\mcatconditiontwo}        {(C2)\xspace}
      \newcommand{\mcatconditionthree}      {(C3)\xspace}
      \newcommand{\mcatconditionthreeprime} {(C3$^{\prime}$)\xspace}
      \newcommand{\mcatconditions}          {\mcatconditionone, \mcatconditiontwo, and \mcatconditionthree}
      \newcommand{\canonicalmcatface}       {\mathfrak{c}}
      
      \newcommand{\mtree}                   {$[n]$-decorated $(m+1)$-ary tree\xspace}
      \newcommand{\mtrees}                  {$[n]$-decorated $(m+1)$-ary trees\xspace}
      \newcommand{\rhodot}                  {\overline{\rho}}
      \newcommand{\mshiconditionone}        {(S1)\xspace}
      \newcommand{\mshiconditiontwo}        {(S2)\xspace}
      \newcommand{\mshiconditionthree}      {(S3)\xspace}
      \newcommand{\mshiconditionfour}       {(S2M)\xspace}
      \newcommand{\mshiconditionfive}       {(S3M)\xspace}
      \newcommand{\mshiconditions}          {\mshiconditionone, \mshiconditiontwo, \mshiconditionfour, \mshiconditionthree, and \mshiconditionfive}
      \newcommand{\canonicalmshiface}       {\mathfrak{s}}
      \newcommand{\precconditionone}        {(PR0)\xspace}
      \newcommand{\precconditiontwo}        {(PR1)\xspace}
      \newcommand{\precconditionthree}      {(PR2)\xspace}
      \newcommand{\precconditions}          {\precconditionone, \precconditiontwo, and \precconditionthree}
      \newcommand{\mcattreeconditionone}    {(T1)\xspace}
      \newcommand{\mcattreeconditiontwo}    {(T2)\xspace}
      \newcommand{\mcattreeconditionthree}  {(T3)\xspace}
      \newcommand{\mnntriple}               {\boldsymbol{\eta}}
      \newcommand{\mnnsub}                  [1]{\eta_{#1}}
      \newcommand{\mnnzero}                 {\mnnsub{0}}
      \newcommand{\mnnone}                  {\mnnsub{1}}
      \newcommand{\mnntriplefull}           {\mnnzero, \dots, \mnnsub{m}} 
      \newcommand{\mnntriplesub}            [1]{\mnntriple_#1} 
      \newcommand{\mnntriplep}              {\mnntriplesub{p}} 
      \newcommand{\mnnzeroinv}              {\mnnzero^{-1}}
      
      \newcommand{\mnnsubinv}               [1]{\mnnsub{#1}^{-1}}
      \newcommand{\mnndashedsub}            [1]{E_{#1}}
      \newcommand{\mnndashedfull}           { \mnndashedsub{1},\dots,\mnndashedsub{m} }
      
      \newcommand{\mcatface}                [1]{\mathcal{C}_m(#1)}
      \newcommand{\mcattreeconstruct}       {\Psi_{\textnormal{Cat}_n}}
      \newcommand{\mnndashedcup}            {\cup_{s=1}^{m} \mnndashedsub{s}}
      \newcommand{\mshitreeconditionone}    {(ST1)\xspace}
      \newcommand{\mshitreeconditiontwo}    {(ST2)\xspace}
      \newcommand{\mshitreeconditionthree}  {(ST3)\xspace}
      \newcommand{\mshitreeconditionfour}   {(ST2M)\xspace}
      \newcommand{\mshitreeconditionfive}   {(ST3M)\xspace}
      
      \newcommand{\shirank}                 {h}
      \newcommand{\kthChildOfColor}         [2]{v^{({#1})}_{#2}}

    \theoremstyle{definition}
    \newtheorem{definition}{Definition}[section]
    \newtheorem{example}[definition]{Example}

    \theoremstyle{plain}
    \newtheorem{lemma}[definition]{Lemma}
    \newtheorem{theorem}[definition]{Theorem}
    \newtheorem{corollary}[definition]{Corollary}
    \newtheorem{proposition}[definition]{Proposition}

    \theoremstyle{remark}
    \newtheorem{remark}[definition]{Remark}

    \title{Bijections for faces of the Shi and Catalan arrangements}
    \author{Duncan Levear}

\begin{document}

\maketitle

\subsection*{Abstract}
  In 1986, Shi derived the famous formula $(n+1)^{n-1}$ for the number of regions of the Shi arrangement, a hyperplane arrangement in $\mathbb{R}^n$. There are at least two different bijective explanations of this formula, one by Pak and Stanley, another by Athanasiadis and Linusson. In 1996, Athanasiadis used the finite field method to derive a formula for the number of $k$-dimensional faces of the Shi arrangement for any $k$. Until now, the formula of Athanasiadis did not have a bijective explanation. In this paper, we extend a bijection for regions defined by Bernardi to obtain a bijection between the $k$-dimensional faces of the Shi arrangement for any $k$ and a set of decorated binary trees. Furthermore, we show how these trees can be converted to a simple set of functions of the form $f: [n-1] \to [n+1]$ together with a marked subset of $\text{Im}(f)$. This correspondence gives the first bijective proof of the formula of Athanasiadis. In the process, we also obtain a bijection and counting formula for the faces of the Catalan arrangement. All of our results generalize to both extended arrangements.

\subsection*{Acknowledgments}
  This paper is a condensed version of the author's Ph.D.\ dissertation. The author thanks his advisor Olivier Bernardi for suggesting the project and for sharing numerous illuminating suggestions. Additional thanks to Christos Athanasiadis for very helpful background references.

\setcounter{tocdepth}{4}
\tableofcontents

\section{Introduction} \label{sec_introduction}
  A (real) \emphnew{hyperplane arrangement} is a finite collection of affine hyperplanes in $\mathbb{R}^n$ for $n \geq 1$. There are several interesting examples, including the braid, Catalan, and Shi arrangements (defined below). These arrangements have been studied in numerous research papers, such as \cite{cite_armstrong2012shi, cite_athanasiadis2010acombinatorial, cite_athanasiadis1999asimple, cite_postnikov2000deformationsof, cite_gill1998thenumber, cite_seo2012shithreshold, cite_hopkins2015bigraphicalarrangements, cite_hopkins2011orientationssemiorders, cite_corteel2015bijectionsbetween, cite_bernardi2018deformationsof, cite_gessel2019labeledbinary, cite_athanasiadis2004generalizedcatalan}. However, one aspect of these arrangements that has received less attention is their faces, which are the focus of this paper. Our main result is an explicit bijection between the faces of the Catalan/Shi arrangements and certain decorated binary trees. In turn, we obtain a bijective explanation for counting formulae that were previously only obtained by a finite field method. We start with some basic definitions.

  \subsection{Hyperplane arrangements} \label{subsec_arrangements}
    \begin{figure}[t]
      \centering
      \nomargins
      {
        \begin{tikzpicture}
        \begin{scope}[scale=.5]
          \node at (0,4.3) {\bf Braid};
          \draw[-,ultra thick] (-3.0,0) -- (3.0,0) node[above]{\tiny $x_1-x_3=0$};
          \draw[-,rotate=60,ultra thick] (-3.3,0) -- (3.3,0) node[below,rotate=60]{\tiny $x_1-x_2=0$};
          \draw[-,rotate=-60, ultra thick] (-3.3,0) -- (3.3,0) node[above,rotate=-60]{\tiny $x_2-x_3=0$};
        \end{scope}
        
        \begin{scope}[scale=.5,shift={(10,0)}]
          \node at (0,4.3) {\bf Shi};
          \draw[-,ultra thick] (-3.0,0) -- (3.0,0) node[above]{};
          \draw[-,black,thick] (-3.0,1) -- (3.0,1) node[above]{\tiny $x_1-x_3=1$};
          \draw[-,rotate=60,ultra thick] (-3.3,0) -- (3.3,0) node[below,rotate=60]{};
          \draw[-,black,rotate around={(60:(0,0))},thick] (-3.3,1) -- (3.3,1) node[below,rotate=60]{\tiny $x_1-x_2=1$};
          \draw[-,rotate=-60, ultra thick] (-3.3,0) -- (3.3,0) node[above,rotate=-60]{};
          \draw[-,black,rotate around={(-60:(0,0))}, thick] (-3.3,1) -- (3.3,1) node[above,rotate=-60]{\tiny $x_2-x_3=1$};
        \end{scope}
        
        \begin{scope}[scale=.5,shift={(20,0)}]
          \node at (0,4.3) {\bf Catalan};
          \draw[-,ultra thick] (-3.0,0) -- (3.0,0) node[above]{};
          \draw[-,black,thick] (-3.0,1) -- (3.0,1) node[above]{};
          \draw[-,black,thick] (-3.0,-1) -- (3.0,-1) node[above,xshift=.1cm]{\tiny $x_1-x_3=-1$};
          \draw[-,rotate=60,ultra thick] (-3.3,0) -- (3.3,0) node[below,rotate=60]{};
          \draw[-,black,rotate around={(60:(0,0))},thick] (-3.3,1) -- (3.3,1) node[right,rotate=0]{};
          \draw[-,black,rotate around={(60:(0,0))},thick] (-3.3,-1) -- (3.3,-1) node[below,xshift=-.05cm,yshift=.05cm,rotate=60]{\tiny \; \; $x_1-x_2=-1$};
          \draw[-,rotate=-60, ultra thick] (-3.3,0) -- (3.3,0) node[above,rotate=-60]{};
          \draw[-,black,rotate around={(-60:(0,0))}, thick] (-3.3,1) -- (3.3,1) node[above,rotate=-60]{};
          \draw[-,black,rotate around={(-60:(0,0))}, thick] (-3.3,-1) -- (3.3,-1) node[above,rotate=-60]{\tiny $x_2-x_3=-1$};
        \end{scope}
        \end{tikzpicture}
      }
      \caption[Braid, Shi, and Catalan arrangements]{The braid, Shi, and Catalan hyperplane arrangements\footnotemark{} in $\mathbb{R}^3$.}
      \label{fig_arrangements}
    \end{figure}
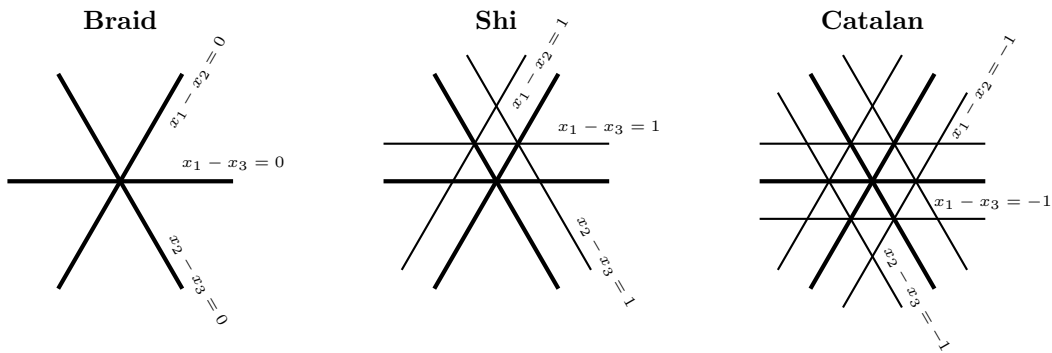
    The \emphnew{regions} of a hyperplane arrangement are the connected components of the complement of its hyperplanes. The \emphnew{braid arrangement} in $\mathbb{R}^n$ is the collection of hyperplanes 
    \begin{align}
    x_i - x_j = 0, \; \text{for} \; 1 \leq i < j \leq n. 
    \end{align} 
    It is not hard to see that there are $n!$ regions, one for each possible relative ordering of the coordinates $(p_i) \in \mathbb{R}^n$. The \emphnew{Shi arrangement} in $\mathbb{R}^n$ is the collection of hyperplanes 
    \begin{align}
    x_i - x_j = 0,1, \; \text{for} \; 1 \leq i < j \leq n. 
    \end{align} 
    It is known that there are $(n+1)^{n-1}$ regions. This formula was first derived algebraically by Shi \cite{cite_shi1986thekazhdan} in 1986, and later re-proved bijectively \cite{cite_athanasiadis1999asimple, cite_bernardi2018deformationsof, cite_stanley1996hyperplanearrangements, cite_duarte2015thebraid}. The \emphnew{Catalan arrangement} in $\mathbb{R}^n$ is the collection of hyperplanes
    \begin{align}
    x_i - x_j = 0,1, \; \text{for} \; 1 \leq i,j \leq n. 
    \end{align}
    There are $n! \text{Cat}_n$ regions, where $\text{Cat}_n = \frac{(2n)!}{n!(n+1)!}$ is the $n$th Catalan number. This result first appeared explicitly in \cite{cite_stanley1996hyperplanearrangements}, where it is obtained by describing a connection to semiorders. All three arrangements are shown in \figref{fig_arrangements} for $n=3$. 

    \footnotetext{Technically speaking, \figref{fig_arrangements} displays the induced hyperplane arrangements in the vector space $\{(x_1,x_2,x_3) \in \mathbb{R}^3 \given \sum x_i =0\} \cong \mathbb{R}^2$. Such induced arrangements are referred to in the literature as ``essentializations''. In this paper, we will not use any essentializations.} 

  \subsection{Faces} \label{subsec_faces}
    A \emphnew{face} of a hyperplane arrangement is the solution set to a non-void system of equalities and (strict\footnote{We use strict inequalities so that the faces are disjoint as subsets of $\mathbb{R}^n$. Taking the topological closure of our faces gives the conventional notion of a face.}) inequalities, one for each hyperplane. The \emphnew{dimension} of a face is the dimension of its affine span. The regions are the highest-dimensional faces. Viewing the regions as polytopes, the faces of the arrangement are the (open) faces of these polytopes. We also refer to the faces of the braid (resp. Shi, Catalan) arrangement as \emphnew{braid} (resp. \emphnew{Shi, Catalan}) \emphnew{faces}. 
    
    The braid faces have an easy combinatorial description, which we recall now. A \emphnew{set partition} of a set $X$ is a set of pairwise-disjoint nonempty subsets, called \emphnew{blocks}, whose union equals $X$. The number of set partitions of an $n$-element set with $k$ blocks is denoted by $\emphnewmath{S(n,k)}$, and called the \emphnew{Stirling number of the second kind}. An \emphnew{ordered set partition} of a set $X$ is a set partition of $X$ with a chosen ordering of the blocks. The number of ordered set partitions of an $n$-element set with $k$ blocks is therefore $k! S(n,k)$. 
    
    Here and throughout we use the notation $\emphnewmath{[n]}$ for the set \{1,2,\dots,n\}. It is not hard to prove that the braid faces are in bijection with ordered set partitions of $[n]$ (see \cite[Exercise 2.10]{cite_stanley2004anintroduction}). From this bijection we obtain the counting formula $k! S(n,k)$ for the number of Braid faces of dimension $k = 1,2,\dots,n$.
    
    Counting the Catalan and Shi faces is not so simple. In 1996, Athanasiadis devised the innovative ``finite field method,'' and thereby obtained counting formulae for the number of faces of these arrangements. His result for the Shi arrangement is the following:
    \begin{theorem} \cite[Cor. 8.2.2]{cite_athanasiadis1996algebraiccombinatorics} \label{thm_ath_shi}
      The number of $k$-dimensional faces of the Shi arrangement in $\mathbb{R}^n$ is 
      \begin{align} \label{eq_athanasiadis_shi_counting}
        \binom{n}{k} \sum_{i=0}^{n-k} \binom{n-k}{i} (-1)^i (n-i+1)^{n-1}, \; \; 1 \leq k \leq n.
      \end{align}
    \end{theorem}
    Athanasiadis remarked that, by inclusion-exclusion, the formula \eref{eq_athanasiadis_shi_counting} enumerates the set of \emphnew{marked functions}, 
    \begin{align} \label{eq_athanasiadis_shi_faces}
      \left\{(f,S) \given \begin{array}{l} f : [n-1] \to [n+1] \\ S \subset \text{Im}(f) \cap [n] \\ |S| = n-k \\ \end{array} \right \},
    \end{align}
    where $\emphnewmath{\text{Im}(f)}$ denotes the image of $f$. He raised the question of a bijective explanation. One of our main results is an answer to this question, that is, an explicit bijection between the set of Shi faces and the set of marked functions in \eref{eq_athanasiadis_shi_faces}.

  \subsection{Main results} \label{subsec_mainresults}
    We have three main results. First, we obtain a bijection between the set of Catalan faces and a set of decorated binary trees that we call $[n]$-decorated binary trees. Second, using the map for Catalan faces as an ingredient, we obtain a bijection between the set of Shi faces and a simple subset of these trees (namely, those that ``decrease to the right''). The key fact is that among all Catalan faces belonging to a given Shi face, there is a unique face whose corresponding tree is ``decreasing to the right.'' Our bijection (and overall method) is an extension of an approach for regions due to Bernardi \cite{cite_bernardi2018deformationsof}. Our third result is a bijection between the trees corresponding to Shi faces and the set of marked functions \eref{eq_athanasiadis_shi_faces}. An example of the resulting bijections for the Shi arrangement with $n=3$ is shown in \figref{fig_shi_example}. 

    Furthermore, our results extend to the ``extended arrangements,'' indexed by an integer parameter $m \geq 1$. The \emphnew{$m$-Catalan arrangement} in $\mathbb{R}^n$ is the collection of hyperplanes 
    \begin{align}
        x_i - x_j = 0,1,\dots,m-1,m, \; \text{for} \; 1 \leq i , j \leq n.
    \end{align}
    The \emphnew{$m$-Shi arrangement} in $\mathbb{R}^n$ is the collection of hyperplanes 
    \begin{align}
        x_i - x_j = -m+1,\dots,m-1,m, \; \text{for} \; 1 \leq i < j \leq n.
    \end{align}
    The classical arrangements are obtained by specializing to $m=1$. For any $m \geq 1$, we describe an explicit bijection between the faces of these arrangements and certain decorated $(m+1)$-ary trees. In the $m$-Shi case, we also obtain a set of functions analogous to those in \eref{eq_athanasiadis_shi_faces}. The counting formulae were already found by Athanasiadis, but the bijective correspondences are new.
    
    \begin{figure}[tp]
      \centering
      \includegraphicswide{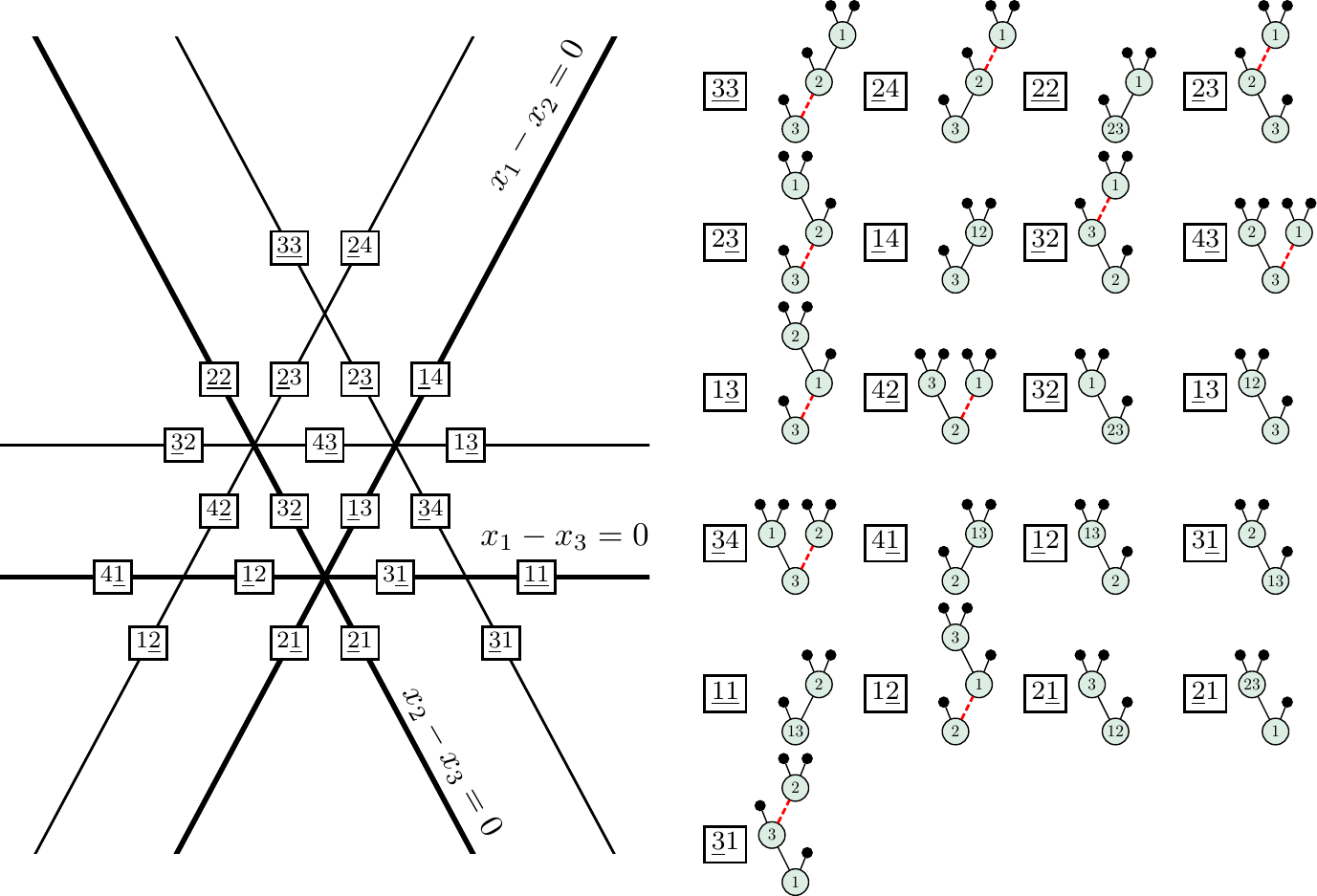}
      \caption[Example of Shi bijection in $\mathbb{R}^3$]{The two-dimensional faces of the Shi arrangement in $\mathbb{R}^3$, each labeled by its image under our bijection to ``marked functions,'' that is, pairs $(f,\{s\})$ where $f : [2] \to [4]$, displayed as a number sequence, and $s \in \image(f) \cap [3]$, shown by underlining the number $s$. In addition, we display each of the corresponding $[3]$-decorated binary trees with exactly two free nodes and such that all right internal edges are descents. The bijection between faces and trees is described in \secref{subsec_mshi_bijection}, and the bijection between trees and ``marked functions'' is described in \secref{sec_additional_bijections}.}
    \label{fig_shi_example}
    \end{figure}

\section{Bijections from faces to trees} \label{sec_bijections}
  In this section, we describe our two main bijections. First, we describe our bijection from a set of decorated trees to the faces of the Catalan arrangement. We also give a generalization for the $m$-Catalan arrangement. Next, we describe our bijection from a sub-family of the decorated trees to the faces of the Shi arrangement, and we also give the analogous result for the $m$-Shi arrangement. We postpone the proofs until \secref{sec_proof_mcatalan} and \secref{sec_proof_mshi}.

  \subsection{Result for the Catalan arrangement} \label{subsec_cat_bijection}
    We first clarify the appropriate set of trees, and then describe our bijection.
    \subsubsection{Decorated binary trees} \label{subsubsec_decorated_binary_trees}
      A \emphnew{tree} is a finite connected acyclic graph. A \emphnew{rooted tree} is a tree with a distinguished vertex, which is called its \emphnew{root}. We adopt the usual vocabulary of \emphnew{parent}, \emphnew{children}, \emphnew{siblings}, \emphnew{leaves} (vertices with no children), and \emphnew{nodes} (vertices with some children). When drawing rooted trees, we conventionally draw the root at the base, meaning e.g.\ that a child appears above its parent. 
      A \emphnew{rooted plane tree} is a rooted tree with a chosen ordering of the children of each node.
      A \emphnew{binary tree} is a rooted plane tree such that each vertex has either $2$ or $0$ children.
      For binary trees, the first child of a node is called its \emphnew{left-child}, and the second is called its \emphnew{right-child}. 
      The edge connecting a node to its left (resp. right) child is called a \emphnew{left} (resp. \emphnew{right}) \emphnew{edge}.
      An \emphnew{internal edge} is an edge between two nodes.
      \begin{definition} \label{def_binary_trees}
        An \emphnew{$[n]$-decorated binary tree} is a binary tree together with the following decorations:
        \begin{itemize}
          \item Each node is labeled with a non-empty subset of $[n]$. Together, the set of labels forms a partition of $[n]$.
          \item All right internal edges are of two types: solid and dashed.
        \end{itemize}
      \end{definition}
      We consider any edge that is not dashed to be solid, in particular, all left-edges must always be solid.
      \figref{fig_seven_decorated_tree} shows a $[7]$-decorated binary tree.
      \begin{figure}[t]
        \begin{center}
          \includegraphics[scale=2.3]{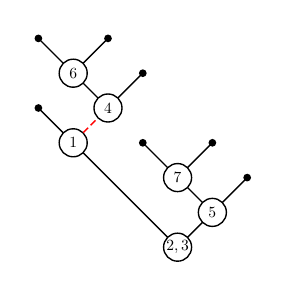}
        \end{center}
        \caption[A $\braceme{7}$-decorated binary tree.]{A $[7]$-decorated binary tree.}
        \label{fig_seven_decorated_tree}
      \end{figure}
      To define our bijection between Catalan faces and $[n]$-decorated binary trees, we require a certain total ordering on the set of vertices. For $v$ any vertex in a binary tree, let $\emphnewmath{\rho(v)}$ denote the word in the alphabet $\{E_0,E_1\}$ encoding the path from the root to $v$, where $E_0$ refers to taking a left-edge, and $E_1$ refers to taking a right-edge.  Let $\emphnewmath{\rhodot(v)}$ be the number of right-edges in this path. 
      \begin{definition}
        Let $T$ be a binary tree and let $v \neq w$ be two vertices. We say \emphnew{$v \prec_T w$} if either:
        \begin{itemize}
          \item $\rhodot(v) < \rhodot(w)$, or
          \item $\rhodot(v) = \rhodot(w)$ and either $\rho(v)$ is a proper prefix of $\rho(w)$, or the first index at which they differ has $E_{1}$ in $\rho(v)$ and $E_{0}$ in $\rho(w)$.
        \end{itemize}
      \end{definition}
      This order is illustrated by example in \figref{fig_binary_prect}.
      \begin{figure}[t]
        \centering
        \includegraphics[width=.5\textwidth]{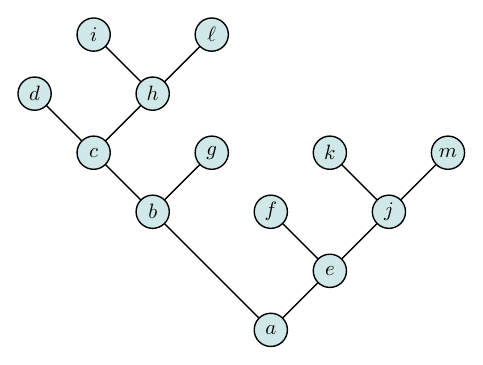}
        \caption[A binary tree example of the $\prec_T$ order]{In this binary tree, $a \prec_T b \prec_T \dots \prec_T m$.}
        \label{fig_binary_prect}
      \end{figure}
    \subsubsection{Bijection from \texorpdfstring{$[n]$-decorated}{[n]-decorated} binary trees to Catalan faces}
      In this subsection, we will define the bijection between Catalan faces in $\mathbb{R}^n$ and $[n]$-decorated binary trees. To define the map, we require some further vocabulary and notation. 
      \begin{definition}
        Let $T$ be an $[n]$-decorated binary tree. A \emphnew{captive node} of $T$ is a node that is connected to its parent by a dashed edge. A \emphnew{free node} is a node that is not a captive node. 
      \end{definition}
      \begin{definition}
        Let $T$ be an $[n]$-decorated binary tree. For $i \in [n]$,
        \begin{enumerate}
          \item Let $\emphnewmath{v_T(i)}$ denote the unique node of $T$ whose label contains $i$.
          \item Let $\emphnewmath{v^1_T(i)}$ denote the right-child of $v_T(i)$.
        \end{enumerate}
      \end{definition}
      Our main result about the faces of the Catalan arrangement is the following:
      \begin{theorem} \label{thm_cat_bijection}
        For any $[n]$-decorated binary tree $T$, there exists a unique Catalan face $\canonicalmcatface$ in $\mathbb{R}^n$ consisting of all points $p = (p_1,\dots,p_n)$ for which the following three conditions hold:
        \begin{enumerate}
          \item[\catconditionone] For all $i,j \in [n]$, $p_i \leq p_j$ if and only if $v_T(i) \preceq_{T} v_T(j)$.
          \item[\catconditiontwo] For all $i,j \in [n]$, $p_i < p_j+1$ if and only if $v_T(i) \prec_{T} v^{1}_T(j)$.
          \item[\catconditionthree] For all $i,j \in [n]$, $p_i = p_j+1$ if and only if $v_T(i) = v^1_T(j)$ and $v^1_T(j)$ is a captive node.
        \end{enumerate}
        Let $\mcatmap$ denote the associated mapping $T \mapsto \canonicalmcatface$. Then $\mcatmap$ is a bijection from $[n]$-decorated binary trees to Catalan faces in $\mathbb{R}^n$. Furthermore, an $[n]$-decorated binary tree $T$ has $k$ free nodes if and only if $\mcatmap(T)$ is a face of dimension $k$. 
      \end{theorem}
      \thmref{thm_cat_bijection} follows from setting $m=1$ in \thmref{thm_mcat_bijection}, our result for the $m$-Catalan arrangement. The map $\mcatmap$ can be made very explicit at the level of individual hyperplanes as follows. A Catalan face may be defined by a choice function 
      \begin{align} 
        \inefunc : \{(i,j,s) \given i,j \in [n], s \in \{0,1\}\} \to \{-1,0,1\},
      \end{align}
      assigning to each hyperplane a choice of $-1$ ($x_i-x_j<s$), $0$ ($x_i-x_j=s$), or $+1$ ($x_i-x_j>s$), such that the resulting system of equalities and inequalities is non-void. Given $T, i, j$, we define $\inefunc(i,j,0)$ and $\inefunc(i,j,1)$ as:
      \begin{align} \label{eq_inefunc_zero}
        \inefunc(i,j,0) &:= 
        \begin{cases} 
          -1 &\padtext{if} v_T(i) \prec_T v_T(j), \\ 
          0 &\padtext{if} v_T(i) = v_T(j), \\ 
          +1 &\lpadtext{otherwise}. 
        \end{cases}  \\
        \inefunc(i,j,s) &:= \label{eq_inefunc_one}
        \begin{cases} 
          -1 &\padtext{if} v_T(i) \prec_T v^{1}_T(j), \\ 
          0 & \padtext{if} v_T(i) = v^{1}_T(j) \lpadtext{and $v^{1}_T(j)$ is a captive node}, \\
          +1 &\lpadtext{otherwise}.
        \end{cases} 
      \end{align}
      The face $\mcatmap(T)$ is the face arising from this choice function $\inefunc$. 
      It is easy to see that this is an equivalent definition of $\mcatmap$ to the one given in \thmref{thm_cat_bijection}. 
      \begin{example}
        Suppose $T$ is the tree in \figref{fig_seven_decorated_tree}. Since $T$ is a $[7]$-decorated binary tree with 5 free nodes, its corresponding Catalan face lives in $\mathbb{R}^7$ and has dimension $5$.
        We list a selection of the inequalities that define the face corresponding to $T$.
        \begin{itemize}
          \item Using the definition of $\inefunc(i,j,0)$, we obtain the relative order of all $x_i$: 
          \begin{align} \label{eq_T_manyinequalities}
            x_2 = x_3 < x_1 < x_5 < x_7 < x_4 < x_6.
          \end{align}
          \item Since $v_T(2) \prec_T v_T(1) \prec_T v^{1}_T(2)$, we have $x_2 < x_1 < x_2 + 1$.
          \item Since $v_T^{1}(1) = v_T(4)$, and $v_T^{1}(1)$ is a captive node, we have $x_1 + 1 = x_4$.
          \item On the other hand, $v_T^{1}(2) = v_T(5)$ but $v_T^{1}(2)$ is a free node, so $x_2+1 < x_5$.
        \end{itemize}
        One point in the face corresponding to $T$ is 
        \begin{align}
          (1.3, \; 1.0, \; 1.0, \; 2.3, \; 2.1, \; 2.4, \; 2.2).
        \end{align}
        For this example, $\inefunc$ is given by the following two tables:
        \begin{center}
          \includegraphics[width=.45\textwidth]{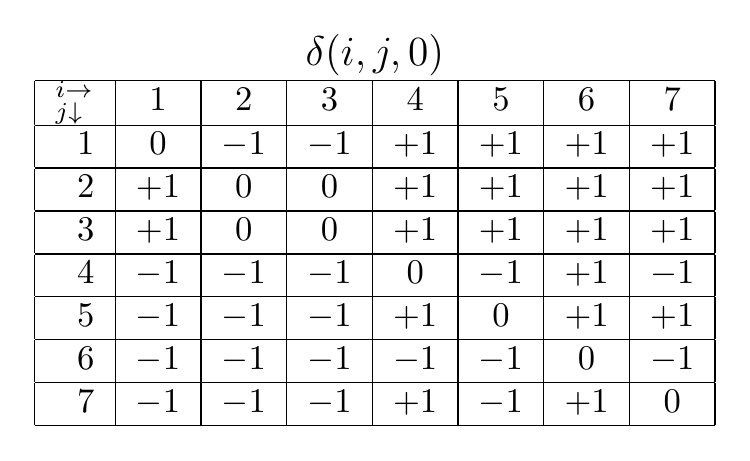}
          \hfill
          \includegraphics[width=.45\textwidth]{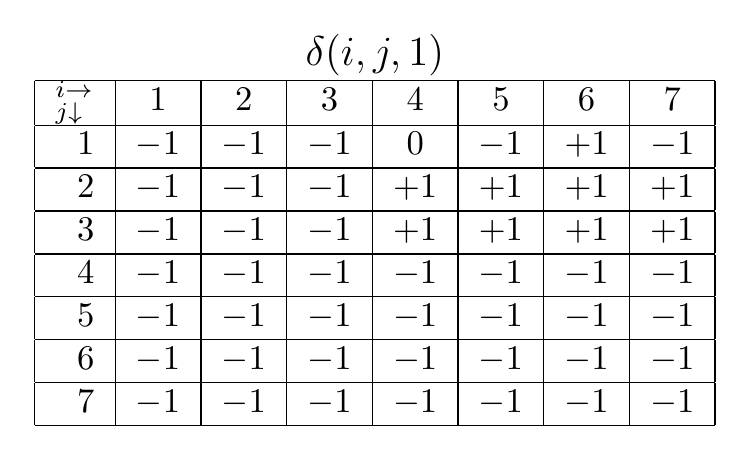}
        \end{center}
        The left table describes whether $x_i$ is less than, equal to, or greater than $x_j$. This simply encapsulates \eref{eq_T_manyinequalities}. The right table describes whether $x_i$ is less than, equal to, or greater than $x_j+1$.
      \end{example}
      \begin{remark} \label{rem_cat_onedim}
        We highlight the lowest-dimensional case of \thmref{thm_cat_bijection} ($k=1$). 
        The theorem establishes a bijection between the one-dimensional Catalan faces in $\mathbb{R}^n$ and $[n]$-decorated binary trees with exactly one free node. 
        It is easy to see that these trees must be right-paths with only dashed edges (with all left-children being leaves). 
        Each of these paths is uniquely defined by the labels of its nodes, which form any ordered set partition. 
        Thus, the number of one-dimensional faces of the Catalan arrangement is $\sum_{i=1}^{n} S(n,i) i!$. 
        This result also appears in \cite[Thm. 1]{cite_gill1998thenumber}, where it is obtained by solving a recurrence relation, and in \cite[Cor. 8.3.2]{cite_athanasiadis1996algebraiccombinatorics} as a result of the finite field method.
      \end{remark}
  \subsection{Result for the \texorpdfstring{$m$-Catalan}{m-Catalan} arrangement} \label{subsec_mcat_bijection}
    A few more definitions are required to state our result for the $m$-Catalan arrangement. First we clarify the appropriate set of trees. 
    \subsubsection{Decorated \texorpdfstring{$m$-ary}{m-ary} trees} \label{subsubsec_decorated_mary_trees}
      An \emphnew{$m$-ary tree} is a rooted plane tree such that each vertex has either $m$ or $0$ children. 
      For a non-root vertex $v$, let $\emphnewmath{\text{lsib}(v)}$ denote the siblings of $v$ to its left (strictly earlier in the child ordering). 
      The \emphnew{rank} of a non-root vertex $v$ is $|\text{lsib}(v)|$. 
      For a node $v$ such that at least one child of rank $>0$ is a node, the \emphnew{cadet} of $v$ is its rightmost such child\footnote{In \cite{cite_bernardi2018deformationsof} ``cadet'' is used to mean the rightmost child that is a node. Our definition differs in that the cadet must have rank $>0$. The word comes from genealogy, in which it refers to a junior heir.}. 
      A \emphnew{cadet edge} is an internal edge whose child is the cadet of the parent. 
      \begin{definition} \label{def_mtrees}
        An \emphnew{$[n]$-decorated $m$-ary tree} is an $m$-ary tree together with the following decorations:
        \begin{itemize}
          \item Each node is labeled with a non-empty subset of $[n]$. Together, the set of labels forms a partition of $[n]$.
          \item All cadet edges are of two types: solid and dashed.
        \end{itemize}
      \end{definition}
      We consider any edge that is not dashed to be solid, in particular, all leftmost edges must always be solid. Note that, each node can have at most one child that is joined by a dashed edge, which must be of rank $>0$. 
      \figref{fig_nine_decorated_tree} shows a $[9]$-decorated binary tree.

      To define our bijection between $m$-Catalan faces and \mtrees, we require a certain total ordering on the set of vertices. For $v$ any vertex in an $(m+1)$-ary tree, let $\emphnewmath{\rho(v)}$ denote the word in the alphabet $\{E_0,E_1,\dots,E_{m}\}$ encoding the path from the root to $v$, where $E_i$ refers to taking an edge of rank $i$. Let $\emphnewmath{\rhodot(v)} := \sum_{i \geq 0} i \cdot |\rho_{i}(v)|$ where $|\rho_{i}(v)|$ is the number of occurrences of letter $E_i$ in $\rho(v)$.
      \begin{definition}
        Let $T$ be an $m$-ary tree and let $v \neq w$ be two vertices. We say \emphnew{$v \prec_T w$} if either:
        \begin{itemize}
          \item $\rhodot(v) < \rhodot(w)$, or
          \item $\rhodot(v) = \rhodot(w)$ and either $\rho(v)$ is a proper prefix of $\rho(w)$, or the first index at which they differ has $E_{b}$ in $\rho(v)$ and $E_{a}$ in $\rho(w)$ with $a < b$.
        \end{itemize}
      \end{definition}
      The important features of this ordering are the following:
      \begin{figure}[t]
        \begin{center}
          \includegraphics[scale=2.3]{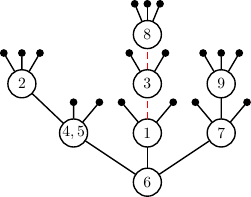}
        \end{center}
        \caption[A $\braceme{9}$-decorated 3-ary tree]{A $[9]$-decorated 3-ary tree.}
        \label{fig_nine_decorated_tree}
      \end{figure}
      \begin{enumerate}
        \item[\precconditionone] The root of $T$ is the $\prec_T$ minimum vertex.
        \item[\precconditiontwo] A node is immediately followed in the $\prec_T$ order by its child of rank $0$.
        \item[\precconditionthree] For vertices $v,w$ each of rank $>0$, $v \prec_T w$ if and only if $\text{left}(v) \prec_T \text{left}(w)$, where $\text{left}(v)$ is the sibling of $v$ with rank one less than the rank of $v$.
      \end{enumerate}
      One can show that the $\prec_T$ order as the unique order satisfying \precconditions. This order is illustrated by example in \figref{fig_ternary_prect}.
      \begin{figure}[t]
        \centering
        \includegraphics[width=.6\textwidth]{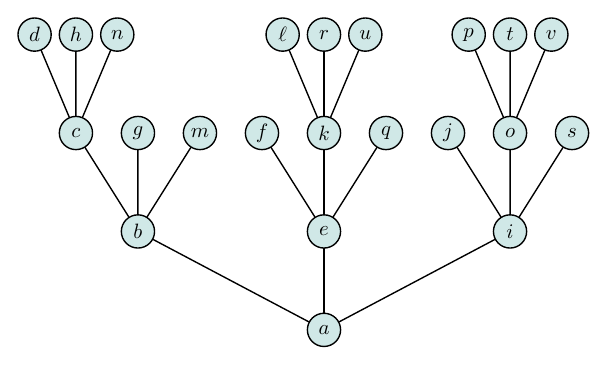}
        \caption[A 3-ary tree example of the $\prec_T$ order]{In this 3-ary tree, $a \prec_T b \prec_T \dots \prec_T v$.} 
        \label{fig_ternary_prect}
      \end{figure}
    \subsubsection{Bijection from \texorpdfstring{\mtrees}{decorated m-ary trees} to \texorpdfstring{$m$-Catalan}{m-Catalan} faces} \label{subsubsec_mcatalan_bijection}
      In this subsection, we will define the bijection between $m$-Catalan faces in $\mathbb{R}^n$ and \mtrees. To define the map, we require some further vocabulary and notation. 
      \begin{definition}
        Let $T$ be an \mtree. A \emphnew{captive node} of $T$ is a node that is connected to its parent by a dashed edge. A \emphnew{free node} is a node that is not a captive node. 
      \end{definition}
      \begin{definition} \label{def_dead_leaves}
        Let $T$ be an \mtree. A \emphnew{dead leaf} of $T$ is a leaf $\ell$ such that there is a captive node in $\text{lsib}(\ell)$. Any vertex that is not a dead leaf is said to be \emphnew{live}. 
      \end{definition}
      \begin{definition}
        Let $T$ be an \mtree.
        \begin{enumerate}
          \item For $i \in [n]$, let $\emphnewmath{v_T(i)}$ denote the unique node of $T$ whose label contains $i$.
          \item For $i \in [n]$ and $s \in [0,m]$, let $\emphnewmath{v^s_T(i)}$ denote the child of $v_T(i)$ of rank $s$.
          \item For $v$ a vertex in $T$, let $\emphnewmath{\nextlive(v)}$ denote the minimum \live vertex of $T$ in the $\prec_T$ order that is greater than or equal to $v$.
        \end{enumerate}
      \end{definition}
      A \emphnew{dashed path} is a path consisting of only dashed edges. For integers $a \leq b$, let $\emphnewmath{[a,b]} := \{a,a+1,\dots,b\}$. Our main result about the faces of the $m$-Catalan arrangement is the following:
      \begin{theorem} \label{thm_mcat_bijection}
        For any \mtree $T$, there exists a unique $m$-Catalan face $\canonicalmcatface$ in $\mathbb{R}^n$ consisting of all points $p = (p_1,\dots,p_n)$ for which the following three conditions hold:
        \begin{enumerate}
          \item[\mcatconditionone] For all $i,j \in [n]$, $p_i \leq p_j$ if and only if $v_T(i) \preceq_T v_T(j)$.
          \item[\mcatconditiontwo] For all $i,j \in [n]$, and $s \in [1,m]$, $p_i < p_j+s$  if and only if $v_T(i) \prec_{T} v^s_T(j)$.
          \item[\mcatconditionthree] For all $i,j \in [n]$, and $s \in [1,m]$, $p_i = p_j+s$ if and only if $v_T(i) = \nextlive(v^s_T(j))$ and there is a dashed path from $v_T(j)$ to $v_T(i)$.
        \end{enumerate}
        Let $\mcatmap$ denote the associated mapping $T \mapsto \canonicalmcatface$ \footnote{The map in \thmref{thm_cat_bijection} is a special case of this map, namely the case $m=1$. We use the same symbol for both, since the value of $m$ may be inferred from context.}. Then $\mcatmap$ is a bijection from \mtrees to $m$-Catalan faces in $\mathbb{R}^n$. Furthermore, an \mtree $T$ has $k$ free nodes if and only if $\mcatmap(T)$ is a face of dimension $k$. 
      \end{theorem}
      We will prove \thmref{thm_mcat_bijection} in \secref{sec_proof_mcatalan}. The map $\mcatmap$ can be made very explicit at the level of individual hyperplanes as follows. An $m$-Catalan face may be defined by a choice function 
      \begin{align} 
        \inefunc : \{(i,j,s) \given i,j \in [n], s \in [0,m]\} \to \{-1,0,1\},
      \end{align}
      assigning to each hyperplane a choice of $-1$ ($x_i-x_j<s$), $0$ ($x_i-x_j=s$), or $+1$ ($x_i-x_j>s$), such that the resulting system of equalities and inequalities is non-void. Given $T, i, j$, we define $\inefunc(i,j,0)$ and $\inefunc(i,j,s)$ for $s \in [1,m]$ as:
      \begin{align} \label{eq_m_inefunc_zero}
        \inefunc(i,j,0) &:= 
        \begin{cases} 
          -1 &\padtext{if} v_T(i) \prec_T v_T(j), \\ 
          0 &\padtext{if} v_T(i) = v_T(j), \\ 
          +1 &\lpadtext{otherwise}. 
        \end{cases}  \\
        \inefunc(i,j,s) &:= \label{eq_m_inefunc_one}
        \begin{cases} 
          -1 &\padtext{if} v_T(i) \prec_T v^{s}_T(j), \\ 
          0 & \; \parbox{6cm}{if $v_T(i) = \nextlive(v^{s}_T(j))$ and there is a dashed path from $v_T(j)$ to $v_T(i)$,} \\
          +1 &\lpadtext{otherwise}.
        \end{cases} 
      \end{align}
      The face $\mcatmap(T)$ is the face arising from this choice function $\inefunc$. 
      It is easy to see that this is an equivalent definition of $\mcatmap$ to the one given in \thmref{thm_mcat_bijection}. 
      We also note that $\inefunc$ is unchanged if the first case in \eref{eq_m_inefunc_one} is replaced by ``$-1 \;$ if $v_T(i) \prec_T \nextlive(v^s_T(j))$'', which bears a closer resemblance to the second case. 
      \begin{example}
        Suppose $T$ is the tree in \figref{fig_nine_decorated_tree}. Since $T$ is a $[9]$-decorated 3-ary tree with 6 free nodes, its corresponding $3$-Catalan face lives in $\mathbb{R}^9$ and has dimension $6$.
        We list a selection of the inequalities that define the face corresponding to $T$.
        \begin{itemize}
          \item Using the definition of $\inefunc(i,j,0)$, we obtain the relative order of all $x_i$: 
          \begin{align} 
            x_6 < x_4 = x_5 < x_2 < x_1 < x_7 < x_3 < x_9 < x_8.
          \end{align}
          \item Since $\nextlive(v^{2}_T(1)) = v_T(8)$, with a dashed path from $v_T(1)$ to $v_T(8)$, we have $x_{1}+2=x_8$. Similarly, we also have $x_3+1 = x_8$, and we must also have $x_1+1=x_3$.
          \item On the other hand $\nextlive(v^{1}_T(7)) = v_T(9)$ but with a solid path, so we have $x_7+1<x_9$.
        \end{itemize}
        One point in the face corresponding to $T$ is 
        \begin{align}
        (
          2.3, \;
          1.5, \;
          3.3, \;
          1.4, \;
          1.4, \;
          1.0, \;
          3.1, \;
          4.3, \;
          4.2
        ).
        \end{align}
        For this example, $\inefunc$ is given by the following three tables:
        \begin{center}
          \includegraphics[width=.45\textwidth]{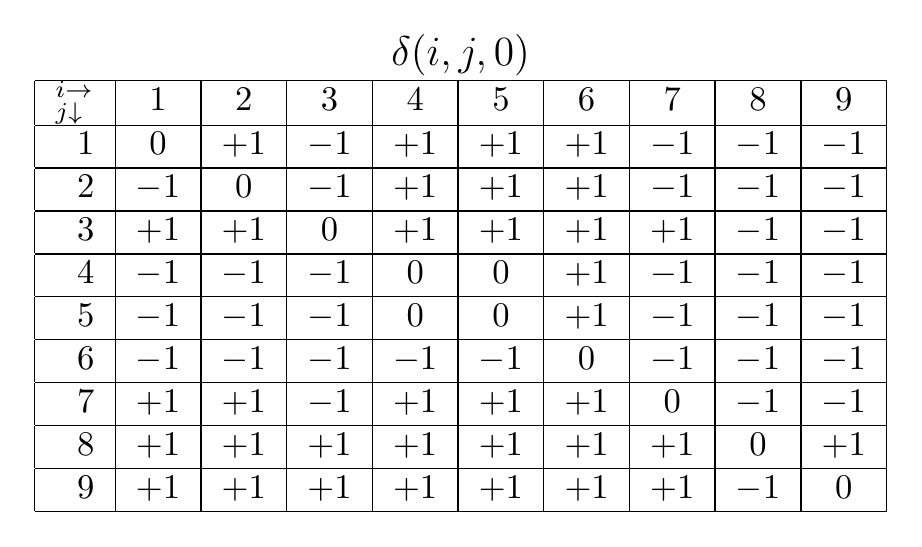}
          \hfill
          \includegraphics[width=.45\textwidth]{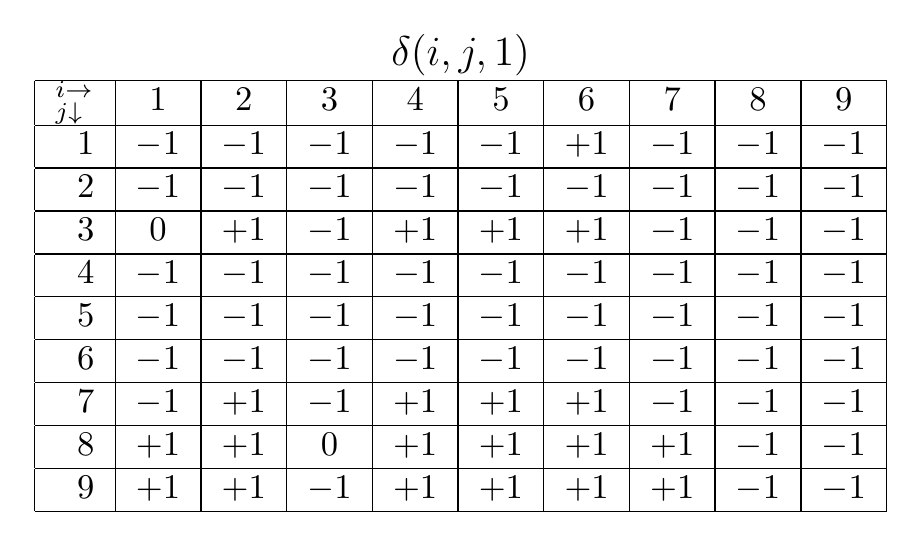}
  
          \includegraphics[width=.45\textwidth]{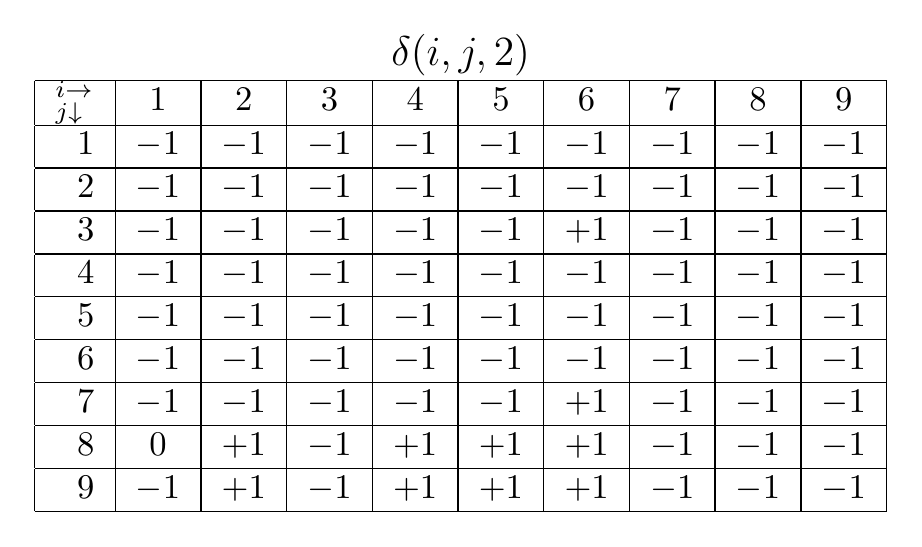}
        \end{center}
      \end{example}
      \begin{remark}
        It is possible to phrase \mcatconditionthree without reference to ``$\nextlive(v^s_T(j))$.'' The \emphnew{drift} of a path of vertices $(v_1,\dots,v_{t})$ is $\sum_{i=1}^{t} |\text{lsib}(v_i)|$. One can show that \mcatconditionthree can be replaced by 
        \begin{itemize}
          \item[\mcatconditionthreeprime] For all $i,j \in [n]$, and $s \in [1,m]$, $p_i = p_j+s$ if and only if there is a dashed path from $v_T(j)$ to $v_T(i)$ whose drift is equal to $s$. 
        \end{itemize}
        This condition will also hold for $s > m$, unlike \mcatconditionthree. The proof that \mcatconditionthree $\iff$ \mcatconditionthreeprime can be done through \lemref{lem_next_of_dead_leaf}. We have chosen to use \mcatconditionthree as stated because of its greater resemblance to the other conditions \mcatconditionone and \mcatconditiontwo. 
      \end{remark}
      \begin{remark} \label{rem_mcat_onedim}
        We highlight the lowest-dimensional case of \thmref{thm_mcat_bijection} ($k=1$). The theorem establishes a bijection between the one-dimensional $m$-Catalan faces in $\mathbb{R}^n$ and \mtrees with exactly one free node. For these trees, all nodes except the root are captive nodes. Such a tree consists of a single dashed path, where the siblings of each node are all leaves. The number of such trees is $$\sum_{i=1}^{n} S(n,i) i! m^{i-1},$$ since the node labels are described by an ordered set partition, and the number of left-sibling leaves for each of the captive nodes may be any number in $[1,m]$.
        This result was first obtained by the finite field method \cite[Thm. 8.3.1, coefficient on $t^{n-1}$]{cite_athanasiadis1996algebraiccombinatorics}. 
      \end{remark}
  
  \subsection{Result for the Shi arrangement} \label{subsec_mshi_bijection}
    We require two additional definitions to state the bijections.
    \begin{definition} \label{def_edge_descent}
      An internal edge of an \mtree is a \emphnew{descent} if $\max A > \min B$, where $A$ is the label of the parent and $B$ is that of the child.
    \end{definition}
    \begin{definition}
      A $[n]$-decorated binary tree is of \emphnew{Shi type} if all right internal edges are descents. 
    \end{definition}
    For example, all of the trees shown in \figref{fig_shi_example} are of Shi type.
    Let $\emphnewmath{\pairsilessthanj} := \{(i,j) \in [n] \times [n] \given i<j\}$.
    Our main result about the Shi faces is the following: 
    \begin{theorem} \label{thm_shi_bijection}
      For any $[n]$-decorated binary tree $T$ of Shi type, there exists a unique Shi face $\canonicalmshiface$ in $\mathbb{R}^n$ consisting of all points $p = (p_1,\dots,p_n)$ for which the following three conditions hold:
      \begin{enumerate}
        \item[\shiconditionone] For all $i,j \in [n], p_i \leq p_j$ if and only if $v_T(i) \preceq_{T} v_T(j)$.
        \item[\shiconditiontwo] For all $(i,j) \in \pairsilessthanj, p_i < p_j+1$ if and only if $v_T(i) \prec_{T} v^1_T(j)$. 
        \item[\shiconditionthree] For all $(i,j) \in \pairsilessthanj, p_i = p_j+1$ if and only if $v_T(i) = v^1_T(j)$ and $v^1_T(j)$ is a captive node. 
      \end{enumerate}
      Let $\mshimap$ denote the associated mapping $T \mapsto \canonicalmshiface$.\footnote{The map in \thmref{thm_shi_bijection} is a special case of this map, namely the case $m=1$. We use the same symbol for both, since the value of $m$ may be inferred from context.} Then $\mshimap$ is a bijection from $[n]$-decorated binary trees of Shi type to Shi faces in $\mathbb{R}^n$.  Furthermore, an $[n]$-decorated binary tree of Shi type $T$ has $k$ free nodes if and only if $\mshimap(T)$ is a face of dimension $k$.
    \end{theorem}

    Equivalently, one could simply define $\mshimap(T)$ to be the unique Shi face in $\mathbb{R}^n$ that contains all points in the Catalan face $\mcatmap(T)$. 
    \thmref{thm_shi_bijection} follows from setting $m=1$ in \thmref{thm_mshi_bijection}, our result for the $m$-Shi arrangement.
    \figref{fig_shi_example} shows the bijection $\mshimap$ for the 21 Shi faces of dimension $2$ in $\mathbb{R}^3$. 

    As we did for $\mcatmap$, the map $\mshimap$ can be made explicit at the level of individual hyperplanes. A Shi face may be defined by a choice function
    \begin{align} \label{eq_delta_shin}
      \delta_{\text{Shi}_n} : \{(i,j,s) \given 1 \leq i < j \leq n, s \in \{0,1\} \} \to \{-1,0,1\}.
    \end{align}
    Let $\delta_{\text{Cat}_n}$ be the choice function corresponding to $T$ given by \eref{eq_inefunc_zero} and \eref{eq_inefunc_one}. Then the choice function $\delta_{\text{Shi}_n}$ for the face $\mshimap(T)$ is simply the restriction of $\delta_{\text{Cat}_n}$ to the domain in \eref{eq_delta_shin}. 
    \begin{remark} \label{rem_shi_onedim}
      We highlight the lowest-dimensional case of \thmref{thm_shi_bijection} ($k=1$). 
      Specialized to this case, \thmref{thm_shi_bijection} establishes a bijection between the one-dimensional Shi faces in $\mathbb{R}^n$ and $[n]$-decorated binary trees of Shi type with exactly one free node. 
      As in \remref{rem_cat_onedim}, these trees must be right-paths with all nodes joined by dashed edges (with all left-children being leaves). 
      Furthermore, all right-edges must be descents. 
      Each of these paths is uniquely defined by the labels of its nodes, which form the ascending runs of any permutation of $[n]$. 
      It follows that the number of one-dimensional Shi faces in $\mathbb{R}^n$ is $n!$.
    \end{remark}
  \subsection{Result for the \texorpdfstring{$m$-Shi}{m-Shi} arrangement}
    \begin{definition}
      An \mtree is of \emphnew{Shi type} if all internal edges of rank $m$ are descents. 
    \end{definition}
    Our main result about the $m$-Shi faces is the following: 
    \begin{theorem} \label{thm_mshi_bijection}
      For any \mtree $T$ of Shi type, there exists a unique $m$-Shi face $\canonicalmshiface$ in $\mathbb{R}^n$ consisting of all points $p = (p_1,\dots,p_n)$ for which the following five conditions hold:
      \begin{enumerate}
        \item[\mshiconditionone] For all $i,j \in [n], p_i \leq p_j$ if and only if $v_T(i) \preceq_{T} v_T(j)$.
        \item[\mshiconditiontwo] For all $i,j \in [n]$ and $s \in [1,m-1]$, $p_i < p_j+s$ if and only if $v_T(i) \prec_{T} v^s_T(j)$. 
        \item[\mshiconditionfour] For all $(i,j) \in \pairsilessthanj, p_i < p_j+m$ if and only if $v_T(i) \prec_{T} v^m_T(j)$.
        \item[\mshiconditionthree] For all $i,j \in [n]$ and $s \in [1,m-1]$, $p_i = p_j+s$ if and only if $v_T(i) = \nextlive(v^s_T(j))$ and there is a dashed path from $v_T(j)$ to $v_T(i)$. 
        \item[\mshiconditionfive] For all $(i,j) \in \pairsilessthanj, p_i = p_j+m$ if and only if $v_T(i) = \nextlive(v^m_T(j))$ and there is a dashed path from $v_T(j)$ to $v_T(i)$. 
      \end{enumerate}
      Let $\mshimap$ denote the associated mapping $T \mapsto \canonicalmshiface$. Then $\mshimap$ is a bijection from \mtrees of Shi type to $m$-Shi faces in $\mathbb{R}^n$.  Furthermore, an \mtree of Shi type $T$ has $k$ free nodes if and only if $\mshimap(T)$ is a face of dimension $k$.
    \end{theorem}
    Equivalently, one could simply define $\mshimap(T)$ to be the unique $m$-Shi face in $\mathbb{R}^n$ that contains all points in the $m$-Catalan face $\mcatmap(T)$. 
    We will use this latter formulation in the proof of \thmref{thm_mshi_bijection} in \secref{sec_proof_mshi}. 
    The key fact underlying \thmref{thm_mshi_bijection} is that for every $m$-Shi face $\canonicalmshiface$, there is one and only one $m$-Catalan face contained by $\canonicalmshiface$ such that $\mcatmapinv(\canonicalmcatface)$ is of Shi type.  

    As we did for $\mcatmap$, the map $\mshimap$ can be made explicit at the level of individual hyperplanes. An $m$-Shi face may be defined by a choice function
    \begin{align}
      \delta_{\text{Shi}_n} : \{(i,j,s) \given 1 \leq i < j \leq n, s \in [-m+1,m] \} \to \{-1,0,1\}.
    \end{align}
    Let $\delta_{\text{Cat}_n}$ be the choice function corresponding to $T$ given by \eref{eq_m_inefunc_zero} and \eref{eq_m_inefunc_one}. Then the face $\mshimap(T)$ is given by the following $\delta_{\text{Shi}_n}$. For $i,j \in \pairsilessthanj$ and $s \in [-m+1,m]$,
    \begin{align}
      \delta_{\text{Shi}_n}(i,j,s) = \begin{cases} \delta_{\text{Cat}_n}(i,j,s) & \padtext{if} s\geq0 \\ -\delta_{\text{Cat}_n}(j,i,-s) & \padtext{if} s < 0. \end{cases}
    \end{align}
    Put another way, we can re-use the choice function for the $m$-Catalan face corresponding to $T$, and then ``forget'' the choices for hyperplanes that are not in the $m$-Shi arrangement.
    \begin{remark} \label{rem_mshi_onedim}
      We highlight the lowest-dimensional case of \thmref{thm_mshi_bijection} ($k=1$). Specialized to this case, \thmref{thm_mshi_bijection} establishes a bijection between the one-dimensional $m$-Shi faces in $\mathbb{R}^n$ and \mtrees of Shi type with exactly one free node. As in \remref{rem_mcat_onedim}, all nodes except the root are captive nodes, so the tree must be simply a path of captive nodes each with siblings that are all leaves. Furthermore, any internal edges of rank $m$ must be descents. One can show that the number of such trees is 
      \begin{align} \label{eq_mshi_onedim_formula}
        m^{n-1} n!,
      \end{align}
      which we leave as an exercise for the reader. It follows that the number of one-dimensional $m$-Shi faces in $\mathbb{R}^n$ is given by \eref{eq_mshi_onedim_formula}. This striking formula also follows easily from \cite[Thm. 8.2.1]{cite_athanasiadis1996algebraiccombinatorics}, which is obtained using the finite field method. This formula has been further refined in \cite[Cor. 4.3]{cite_ehrenborg2019countingfaces} via a recurrence relation. This result will reappear later as a consequence of our general bijection from trees of Shi type to certain functions $f: [n-1] \to [mn+1]$ in \subsecref{subsec_mshi_everything_results}.   
    \end{remark}

\section{Proof that \texorpdfstring{$\mcatmap$}{map for m-Catalan faces} is a bijection} \label{sec_proof_mcatalan}
  In this section we will prove \thmref{thm_mcat_bijection}, which implies \thmref{thm_cat_bijection}.
  Throughout this section, $T$ is a given \mtree.
  We separate this proof into five distinct elements: existence and uniqueness of the $m$-Catalan face, injectivity and surjectivity of $\mcatmap$, and the grading by dimension. 
  From our point of view, the hardest element to prove is surjectivity, which will require several additional definitions and lemmas. 

  Before moving to these proofs, we record a useful lemma, which will be used in the following sections, and again later in \subsecref{subsec_mshi_proof_surjective}.
  \begin{lemma} \label{lem_next_of_dead_leaf}
    Let $R$ be an \mtree and let $\ell$ be a dead leaf of $R$. Let $v,k$ be such that $\ell$ is the $k$th right-sibling of a captive node $v$, meaning the rank of $\ell$ is $k$ more than the rank of $v$. Then the vertex immediately following $\ell$ in the $\prec_R$ order is the child of $v$ of rank $k$. 
  \end{lemma}
  \begin{proof}
    We use induction on $k$, starting with $k=1$. Recall from \precconditionone that the vertex immediately following $v$ is its child of rank $0$. Letting $v^1$ be the child of $v$ of rank $1$, it then follows from \precconditionthree that $\ell \prec_R v^1$. Now suppose there exists a vertex $u$ such that $\ell \prec_R u \prec_R v^1$. Either $u$ has rank $>0$, in which case we may apply \precconditionthree to obtain a vertex strictly between $v$ and $v^0$, which is impossible by \precconditionone, or $u$ has rank $0$ and we may replace $u$ by its parent, repeating until $u$ has rank $>0$, all the while maintaining $\ell \prec_R u \prec_R v^1$ by \precconditiontwo. We must eventually replace $u$ by a node of rank $>0$, since $\ell \prec_R u$ and $\ell$ is a leaf (here we have used \precconditionone). Thus, in either case we eventually reach a contradiction. The inductive step for general $k$ follows by analogous reasoning.
  \end{proof}
  \lemref{lem_next_of_dead_leaf} can be used to give an explicit description of $\nextlive(v)$ for any vertex $v$. If $v$ is \live, then $\nextlive(v)=v$. If $v$ is dead, then we may apply \lemref{lem_next_of_dead_leaf} to advance the $\prec_R$ order by one vertex, and repeat as necessary. We omit the details, since this is not needed for our purposes. 

  \subsection{Existence, uniqueness, and grading by dimension} \label{subsec_mcat_proof_existence}
    In this subsection, we will prove the existence and uniqueness of the $m$-Catalan face corresponding to $T$, and that this face has dimension $k$, where $k$ is the number of free nodes of $T$. 

    We first point out that uniqueness is trivial, since the hyperplanes of the $m$-Catalan arrangement consist of all those of the form $x_i=x_j+s$ for $i,j \in [n]$ and $s \in [0,m]$. If two $m$-Catalan faces are on the same side of all $m$-Catalan hyperplanes then they must be the same face. 

    We will show existence and dimension grading at the same time.
    We say a point $p \in \mathbb{R}^n$ is \emphnew{valid} (for $T$) if it satisfies conditions \mcatconditions from \thmref{thm_mcat_bijection}. 
    To show existence, it is enough to construct a single valid point $p \in \mathbb{R}^n$, since whichever $m$-Catalan face contains $p$ will satisfy the requirements of \thmref{thm_mcat_bijection}. 
    Fix a vector of positive reals $\epsvector = (\epsilon_0, \epsilon_1, \dots, \epsilon_{k-1}) \in \mathbb{R}_{>0}^k$ such that $\epsilon_0 < 1/2$ and for all $q \in [k-1]$ we have $\epsilon_q < \frac{\epsilon_{q-1}}{2}$. 

    Let $E$ be the set of internal edges of $T$. 
    For $e \in E$, let $\text{child}(e)$ be the endpoint of $e$ that is further from the root. 
    For $e,e' \in E$ say $e \prec_T e'$ if $\text{child}(e) \prec_T \text{child}(e')$. 
    Since $T$ has $k$ free nodes, $T$ has $k-1$ solid internal edges (every free node except the root is the child endpoint of a solid internal edge). 
    Let $e_1 \prec_T \dots \prec_T e_{k-1}$ be the internal solid edges of $T$, and define $\edgeweight: E \to \mathbb{R}$ by 
    \begin{align} \label{eq_cat_existence_edgeweight}
      \edgeweight(e) := \begin{cases} \epsilon_q & \text{if} \; e = e_q, \\ 0 & \text{otherwise}. \end{cases}
    \end{align}
    For $i \in [n]$, let $\rho(i)$ be the tuple of edges in $E$ that appear on the path from the root of $T$ to $v_T(i)$. We define $\rhodot(i) := \rhodot(v_T(i))$, which was defined in \subsubsecref{subsubsec_decorated_mary_trees}. For the point $p$, we set 
    \begin{align} \label{eq_mcat_existence_point}
      p_i := \epsilon_0 + \rhodot(i) + \sum_{e \in \rho(i)} \edgeweight(e).
    \end{align}
    It remains to show that $p$ is valid. We show $p$ satisfies \mcatconditiontwo and \mcatconditionthree, and leave the similar verification of \mcatconditionone to the reader. 

    Let $i,j,s$ be given, and assume $v_T(i) \prec_{T} v^s_T(j)$. There are three cases arising from the definition of $\prec_T$:

    \caseif{A}{$\rhodot(i) < \rhodot(v^s_T(j))$.}
      Since $\rhodot(v^s_T(j)) = \rhodot(j)+s$, we have $\rhodot(i) - \rhodot(j) \leq  s-1$.
      Therefore from \eref{eq_mcat_existence_point} we have
      \begin{align}
        p_i - p_j &= \rhodot(i) - \rhodot(j) + \sum_{e \in \rho(i)} \edgeweight(e) - \sum_{e \in \rho(j)} \edgeweight(e) \\
        &\leq s-1 + \sum_{q=1}^{k-1} \epsilon_q < s,
      \end{align}
      because it is clear from the definition of $\epsilon$ that $\sum_{q=1}^{k-1} \epsilon_q < 1$.

    \caseif{B}{$\rhodot(i) = \rhodot(v^s_T(j))$ and $\rho(i)$ is a proper prefix of $\rho(v^s_T(j))$.}
      This case is not possible, since $s>0$. Explicitly, if $\rho(i)$ is a proper prefix of $\rho(v^s_T(j))$, then $\rhodot(i) \leq \rhodot(j)$. Since $\rhodot(v^s_T(j)) = \rhodot(j)+s$ and $s>0$, we have strictly $\rhodot(i) < \rhodot(v^s_T(j))$.

    \caseif{C}{$\rhodot(i) = \rhodot(v^s_T(j))$ and the first difference between $\rho(i)$ and $\rho(v^s_T(j))$ is $E_b$ in $\rho(i)$ and $E_a$ in $\rho(v^s_T(j))$ with $a<b$.}
      First, we claim that this edge $E_a$ in $\rho(v^s_T(j))$ must be an internal edge, say $e_{\ell}$. This is because if $E_a$ were not an internal edge, then since $a<b$, we would have strictly $\rhodot(i) > \rhodot(v^s_T(j))$. 

      Since $\rhodot(i) = \rhodot(v^s_T(j)) = \rhodot(j)+s$, we have
      \begin{align} \label{eq_mcatconditiontwo_check}
        p_i - p_j &= 
        s + \sum_{e \in \rho(i)} \edgeweight(e) - \sum_{e \in \rho(j)} \edgeweight(e) 
        \leq 
        s - \epsilon_{\ell} + \sum_{e \in \rho(i) \smallsetminus \rho(j)} \edgeweight(e),
      \end{align}
      where the inequality is obtained by cancellation of the common $\edgeweight$ terms, and then replacing the negative sum by only one of its terms (the notation $e \in \rho(i) \smallsetminus \rho(j)$ should be understood to mean $e$ is in $\rho(i)$ and $e$ is not in $\rho(j)$). Next, we claim that, all edges in $\rho(i) \smallsetminus \rho(j)$ are greater than $e_{\ell}$ in the $\prec_T$ order. This is immediately true for the first edge $E_b \in \rho(i) \smallsetminus \rho(j)$, since $a<b$, and so it will indeed be true for all edges that follow $E_b$. Therefore the $\sum$ in the last expression of \eref{eq_mcatconditiontwo_check} is less than or equal to $\sum_{q=\ell+1}^{k-1} \epsilon_q$. Thus,
      \begin{align}
        p_i-p_j &\leq s - \epsilon_{\ell} + \sum_{q=\ell+1}^{k-1} \epsilon_q \lt s,
      \end{align}
      where the last inequality follows from the geometric decay of $\epsvector$.

    We have shown the forward direction of \mcatconditiontwo, and the backwards direction follows by identical reasoning.
    We now show $p$ satisfies \mcatconditionthree. It is not hard to show from \lemref{lem_next_of_dead_leaf} and induction on $s$ that $\rhodot(\nextlive(v^s_T(j))) = \rhodot(v_T(j))+s$, and that there is a path from $v_T(j)$ to $\nextlive(v^s_T(j))$. It follows that, if $v_T(i) = \nextlive(v^s_T(j))$ then
    \begin{align} \label{eq_pi_pj_mcatthree}
      p_i - p_j = s + \sum_{e \in \rho(i) \smallsetminus \rho(j)} \edgeweight(e).
    \end{align}
    Furthermore, if the path from $v_T(j)$ to $v_T(i)$ is a dashed path, then the $\sum$ in \eref{eq_pi_pj_mcatthree} is zero. This shows $p_i=p_j+s$ as desired. 
    For the other direction of \mcatconditionthree, assume that $p_i-p_j=s$. It follows from \eref{eq_mcat_existence_point} that 
    \begin{align}
      s = \rhodot(i)-\rhodot(j) + \sum_{e \in \rho(i)} \edgeweight(e) - \sum_{e \in \rho(j)} \edgeweight(e).
    \end{align}
    Therefore the right-hand-side must be an integer, and since $0 < \sum \epsvector < 1$, we have
    \begin{align} \label{eq_edgeweight_eq_edgeweight}
      \sum_{e \in \rho(i)} \edgeweight(e) = \sum_{e \in \rho(j)} \edgeweight(e),
    \end{align}
    and
    \begin{align} \label{eq_rhodot_eq_rhodot}
      \rhodot(i) = \rhodot(j)+s.
    \end{align}
    It follows from the geometric decay of $\epsvector$ that no distinct subsets of $\epsvector$ can have equal sum. 
    Thus, the positive terms on each side of \eref{eq_edgeweight_eq_edgeweight} are identical. 
    It follows that the set of solid edges in $\rho(i)$ and the set of solid edges in $\rho(j)$ are equal.
    Since each node can have at most one dashed edge, either $\rho(j)$ is a prefix of $\rho(i)$ or vice-versa. 
    Since $\rhodot(i) = \rhodot(j)+s$, with $s>0$, it must be that $\rho(j)$ is a prefix of $\rho(i)$. 
    Thus, we have a dashed path from $v_T(j)$ to $v_T(i)$. 

    It remains to show $v_T(i) = \nextlive(v^{s}_T(j))$. 
    If $v_T(i) \prec_T \nextlive(v^{s}_T(j))$, then we could apply \mcatconditiontwo to show $p_i < p_j+s$, which is a contradiction. If $\nextlive(v^{s}_T(j)) \prec_T v_T(i)$, then by the definition of $\prec_T$ and \eref{eq_rhodot_eq_rhodot}, the first difference in the paths must be $E_a$ in the path to $v_T(i)$ and $E_b$ in the path to $\nextlive(v^{s}_T(j))$ with $a<b$. But since there is a path from $v_T(j)$ to $\nextlive(v^{s}_T(j))$, and there is a dashed path from $v_T(j)$ to $v_T(i)$, the edge $E_a$ in question must be dashed. Therefore, since $a < b$, the edge $E_b$ leads to a dead leaf. This contradicts the fact that $\nextlive(v^{s}_T(j))$ is live. Thus, we have $v_T(i) = \nextlive(v^{s}_T(j))$ as desired.

    We have shown that $p$ is valid for any $\epsvector$, and clearly by varying $\epsvector$ one obtains a $k$-dimensional neighborhood of distinct valid points $p \in \mathbb{R}^n$.
    It follows that the $\canonicalmcatface$ in \thmref{thm_mcat_bijection} exists for every $T$, and that the dimension of $\mcatmap(T)$ is at least $k$.

    It remains to show that the dimension of the face $\mcatmap(T)$ is no more than $k$. 
    For a node $v$ in $T$, let $\emphnewmath{\text{label}(v)}$ denote the set of numbers in the label of $v$. 
    Let $r_{1} \prec_T \dots \prec_T r_{k}$ be the free nodes of $T$, and for $i \in [k]$ let $J_{i} := \text{label}(r_i)$, and let $j_i := \min J_i$. 
    Assume that two valid points $p,p'$ match for all coordinates $j_i$, that is, $p_{j_i} = p'_{j_i}$ for all $i \in [k]$. 
    Under this assumption, we claim that the validity of $p$ and $p'$ implies $p=p'$. 
    First, it is easy to show that for each $J_i$, $p_j = p'_j$ for \textit{all} $j \in J_i$. This follows from condition \mcatconditionone in \thmref{thm_mcat_bijection}. The remaining coordinates have indices appearing in the labels of captive nodes, and those are easily seen to be determined by condition \mcatconditionthree in \thmref{thm_mcat_bijection} (every such node has an ancestor which is free). Hence $p = p'$. Since we only assumed $p$ and $p'$ matched on $k$ coordinates, it follows that the affine dimension of $\mcatmap(T)$ is at most $k$, as desired. 

  \subsection{Proof that \texorpdfstring{$\mcatmap$}{map for m-Catalan faces} is injective} \label{subsec_mcat_proof_injective}
    Suppose $T$ and $T'$ are \mtrees and that the following three conditions hold:
    \begin{enumerate}
      \item[\mcattreeconditionone] For all $i,j \in [n]$, $v_T(i) \preceq_T v_T(j)$ if and only if $v_{T'}(i) \preceq_{T'} v_{T'}(j)$.
      \item[\mcattreeconditiontwo] For all $i,j \in [n]$ and $s \in [1,m]$, $v_T(i) \prec_T v^s_T(j)$ if and only if $v_{T'}(i) \prec_{T'} v^s_T(j)$.
      \item[\mcattreeconditionthree] For all $i,j \in [n]$ and $s \in [1,m]$, $v_T(i) = \nextlive(v^s_T(j))$ and there is a dashed path from $v_T(j)$ to $v_T(i)$ if and only if $v_{T'}(i) = \nextlive(v^s_{T'}(j))$ and there is a dashed path from $v_{T'}(j)$ to $v_{T'}(i)$.
    \end{enumerate}
    In this subsection, we show that these conditions imply $T = T'$, which proves that $\mcatmap$ is injective.
    \begin{definition}
      Let $T,T'$ be \mtrees, and suppose $v$ is a vertex in $T$ and $w$ is a vertex in $T'$. We write $\emphnewmath{v \simeq w}$ if both vertices are leaves, or if they are both nodes. We write $\emphnewmath{v \cong w}$ if $v,w$ are both leaves, or if $v,w$ are both nodes and $\text{label}(v) = \text{label}(w)$.
    \end{definition}
    \begin{lemma} \label{lem_mcat_prect_inductive}
      Let $T,T'$ be \mtrees. Let $v_1 \prec_T \dots \prec_T v_{p}$ be the vertices of $T$ and $w_1 \prec_{T'} \dots \prec_{T'} w_{p'}$ be the vertices of $T'$. For any $q\geq1$, if $v_i \simeq w_i$ for all $i \in [q-1]$ then $\rho(v_i) = \rho(w_i)$ for all $i \in [q]$ (where $[0] = \varnothing$). 
    \end{lemma}
    \begin{proof}
      We proceed by induction on $q$. Since $\rho(v_1)$ and $\rho(w_1)$ are both the empty word, the claim is true for $q=1$. For $q>1$, we assume that $v_i \simeq w_i$ for all $i \in [q-1]$, and, by induction, it is enough to show only that $\rho(v_q) = \rho(w_q)$. We may deduce from $v_i \simeq w_i$ for $i \in [q-1]$ that the path $\rho(v_q)$ must also be a path to some vertex in $T'$. It is easy to show that the terminus of this path is none other than $w_q$, and thus $\rho(v_q) = \rho(w_q)$. 
    \end{proof}
    Let $v_1 \prec_T \dots \prec_T v_{p}$ be the vertices of $T$, and let $w_1 \prec_{T'} \dots \prec_{T'} w_{p}$ be the vertices of $T'$. To show $T = T'$, it is enough to show that for all $i \in [p]$,
    \begin{enumerate}[label=(\roman*)]
      \item $v_i \cong w_i$,
      \item $\rho(v_i) = \rho(w_i)$, and
      \item $v_i$ is a captive node if and only if $w_i$ is a captive node. 
    \end{enumerate}
    We prove (i) by induction on $i$. The base case $i=1$ clearly falls under Case A below.

    \caseif{A}{$v_i$ is a node and $w_i$ is a node.}
      Let $a$ be any element of $\text{label}(v_i)$. Let $b$ be any element of $\text{label}(w_i)$. Since $b \in [n]$, $b$ appears in $\text{label}(v_{j})$ for some $j \in [p]$, and by the inductive hypothesis $j \geq i$. This proves 
      \begin{align} \label{eq_mcatalan_vta_prec_vtb}
        v_{T}(a) \preceq_T v_{T}(b).
      \end{align}
      Similarly, $a$ appears in $T'$ in $\text{label}(w_{j'})$ for some $j' \geq i$, so we also have 
      \begin{align} \label{eq_mcatalan_vtprimeb_prec_vtprimea}
        v_{T'}(b) \preceq_{T'} v_{T'}(a).
      \end{align}
      On the other hand, by \mcattreeconditionone, \eref{eq_mcatalan_vta_prec_vtb} implies
      \begin{align} \label{eq_mcatalan_vtprimea_prec_vtprimeb}
        v_{T'}(a) \preceq_{T'} v_{T'}(b).
      \end{align}
      Combining \eref{eq_mcatalan_vtprimea_prec_vtprimeb} and \eref{eq_mcatalan_vtprimeb_prec_vtprimea}, we obtain $v_{T'}(a) = v_{T'}(b)$, which by \mcattreeconditionone also proves $v_{T}(a) = v_{T}(b)$. This shows that $\text{label}(v_i) = \text{label}(w_i)$, and so $v_i \cong w_i$. 

    \caseif{B}{$v_i$ is a leaf and $w_i$ is a leaf.}
      In this case, $v_i \cong w_i$ trivially.

    \caseif{C}{$v_i$ is a leaf and $w_i$ is a node.} 
      We will show that this case is not possible. 
      Let $b$ be any element of $\text{label}(w_i)$. 
      The number $b$ appears in $\text{label}(v_{j})$ for some $j$, and by the inductive hypothesis $j \gt i$. 
      Therefore, there is at least one node following $v_i$ in the $\prec_T$ order. 
      Let $\emphnewmath{\nextnode}(v)$ denote the minimum node of $T$ in the $\prec_T$ order that is greater than or equal to $v$.
      Let $u := \nextnode(v)$.
      It follows that $u \preceq_{T} v_j = v_T(b)$. 
      Further, it follows from \precconditiontwo and the fact that $v_i$ is a leaf that $u$ is the child of rank $s>0$ of some vertex, $v_{h}$.
      Let $a$ be any element of $\text{label}(v_h)$.
      Since $v_h$ is the parent of $u$, $v_{h} \prec_T u$ and since $u = \nextnode(v_i)$ we have $v_{h} \prec_T v_{i}$. 
      In summary we have
      \begin{align} \label{eq_mcat_v_i_sandwich} 
        v_T(a) = v_{h} \prec_{T} v_{i} \prec_{T} u = v^s_T(a) \preceq_{T} v_T(b).
      \end{align}
      The final relation $v^s_T(a) \preceq v_T(b)$ is the one that will create a contradiction.  
      We claim that the opposite is true in $T'$, namely
      \begin{align} \label{eq_vtpb_prec_kthelement}
        v_{T'}(b) \prec_{T'} v^s_{T'}(a).
      \end{align}
      Our proof of \eref{eq_vtpb_prec_kthelement} goes as follows. First, since $v_{h} \prec_T v_i$, we have $h<i$. From \lemref{lem_mcat_prect_inductive} and the inductive hypothesis, we obtain
      \begin{align} \label{eq_mcat_rhovh_eq_rhowh}
        \rho(v_{h}) = \rho(w_{h}),
      \end{align}
      and
      \begin{align} \label{eq_mcat_rhovi_eq_rhowi}
        \rho(v_i) = \rho(w_i).
      \end{align}
      Furthermore, since $h<i$ we have $w_{h} \cong v_{h}$. Since $v_{h} = v_T(a)$ this implies $w_{h} = v_{T'}(a)$. So we can rewrite \eref{eq_mcat_rhovh_eq_rhowh} as 
      \begin{align} \label{eq_rhota_eq_rhotpa}
        \rho(v_T(a)) = \rho(v_{T'}(a)),
      \end{align}
      which implies 
      \begin{align} \label{eq_rho_sthchild_eq_rho_sthchild}
        \rho(v^s_T(a)) = \rho(v^s_{T'}(a)).
      \end{align}
      Finally, combining the equations \eref{eq_mcat_rhovi_eq_rhowi}, \eref{eq_rho_sthchild_eq_rho_sthchild} and the fact that $v_i \prec_T v^s_T(a)$, we can deduce
      \begin{align} \label{eq_wi_prectp_rho_sthchild}
        w_i \prec_{T'} v^s_{T'}(a),
      \end{align}
      since the $\prec_{T}$ and $\prec_{T'}$ orders are defined completely in terms of the paths given by $\rho$. Finally, since $w_i = v_{T'}(b)$, we obtain \eref{eq_vtpb_prec_kthelement}. 
      It is easy to see that \eref{eq_vtpb_prec_kthelement} and the last relation of \eref{eq_mcat_v_i_sandwich} contradict \mcattreeconditiontwo. So indeed this case is impossible.

    \caseif{D}{$v_i$ is a node and $w_i$ is a leaf.} 
      This is the same as Case C by symmetry, so it is also impossible.

    This completes our proof of (i), that $v_i \cong w_i$ for all $i \in [p]$. \lemref{lem_mcat_prect_inductive} shows that (i) implies (ii). Moreover, (iii) clearly follows from \mcattreeconditionthree together with (i) and (ii). Thus we obtain $T = T'$ as desired. 

  \subsection{Proof that \texorpdfstring{$\mcatmap$}{map for m-Catalan faces} is surjective} \label{subsec_mcat_proof_surjective}
    In this subsection, we will show that for every $m$-Catalan face $\canonicalmcatface$ in $\mathbb{R}^n$, there exists an \mtree $T$ such that the three conditions in \thmref{thm_mcat_bijection} are satisfied. To do so, we will need to construct the tree $T$. Before coming to the construction, we introduce some additional structure and vocabulary. 

    We associate to any point $p \in \mathbb{R}^n$ a tuple consisting of $m+1$ functions and $m$ subsets of $[n]$. Let 
    \begin{align}
      W_p := \bigcup_{i=0}^{n} \bigcup_{s=0}^{m} \, \{\, (p_i+s,-s)\, \}.
    \end{align}
    For $x \in \mathbb{R}$ and $z \in \mathbb{Z}$ we write $(x,z) \lesslex (x',z')$ if $x<x'$ or $x=x'$ and $z<z'$. Given $k \in [n]$ and $s \in [0,m]$, we define 
    \begin{align} \label{eq_etap_definition} 
      \emphnewmath{\mnnsub{s}(k)} &:=  \# \{w \in W_p \given w \lesseqlex (p_k+s,-s)\}.
    \end{align}
    For $s \in [1,m]$ we define 
    \begin{align}
      \emphnewmath{\mnndashedsub{s}} := \{k \given \padtext{there exists $j\in[n]$ such that} p_j+s=p_k\}.
    \end{align}
    We let $\emphnewmath{\mnntriplep}$ denote $(\mnntriplefull, \mnndashedfull)$, and call $\mnntriplep$ the \emphnew{$m$-Catalan code} of $p$.

    Let $\emphnewmath{\image(\mnnsub{s})}$ denote the image of $\mnnsub{s}$. Note that $\cup_{s=0}^{n} \image(\mnnsub{s}) = [h]$ for some $h \leq (m+1)n$, and that for any $s,s' \in [1,m]$ with $s \neq s'$ we have $\image(\mnnsub{s}) \cap \image(\mnnsub{s'}) = \varnothing$. We call the $t \in \cup_{s=0}^{n} \image(\mnnsub{s})$ the \emphnew{sites} of $\mnntriplep$. For $s \geq 0$, we say $t$ is a \emphnew{type $s$ site} if $t \in \image(\mnnsub{s})$.

    A \emphnew{dash site} is a site $t$ of type $s>0$ such that $t+1$ is a type $s'$ site with $s'\lt s$ and $\mnnsubinv{s'}(t+1) \subset \mnndashedsub{s-s'}$, where 
    \begin{align}
      \emphnewmath{\mnnsubinv{s}(i)} &:= \{ k \in [n] \given \mnnsub{s}(k)=i\}.
    \end{align}
    \begin{example} \label{example_cat_mnn}
      Suppose $m=1$. Let $p = (1.0, \; 1.0, \; 1.1, \; 2.0, \; 2.2, \; 4.0) \in \mathbb{R}^6$. Then
      \begin{align}
        W_p = \{&(1.0, 0), (1.1, 0), (2.0, -1), (2.0,0), (2.1,-1), \\
        &(2.2,0), (3.0,-1), (3.2,-1), (4.0,0), (5.0,-1)\}.
      \end{align}
      The resulting values of $\mnnzero$ and $\mnnone$ are given by the following table.
        \begin{center}
          \begin{tabular}{r|cccccc}
            $k$ & 1 & 2 & 3 & 4 & 5 & 6 \\ \hline
            $\mnnzero(k)$ &        1 & 1 & 2 & 4 & 6 & 9 \\
            $\mnnone(k)$  &        3 & 3 & 5 & 7 & 8 & 10 
          \end{tabular}
        \end{center}
        We have $\mnndashedsub{1} = \{4\}$. 
        The type 0 sites are 1, 2, 4, 6, 9.
        The type 1 sites are 3, 5, 7, 8, 10.
        The only dash site is $3$.
    \end{example}
    \begin{example} \label{example_mcat_mnn}
      Suppose $m=2$. 
      Let $p = (2.3, \; 1.5, \; 3.3, \; 1.4, \; 1.4, \; 1.0, \; 3.1, \; 4.3, \; 4.2) \in \mathbb{R}^7$.
      The set $W_p$ has cardinality $24$.
      The resulting values of $\mnnzero, \mnnone,$ and $\mnnsub{2}$ are given by the following table.
        \begin{center}
          \begin{tabular}{r|ccccccccc}
                        $k$ & 1  & 2  & 3  & 4  & 5  & 6 & 7  & 8  & 9 \\ \hline
            $\mnnzero(k)$   & 5  & 3  & 11 & 2  & 2  & 1 & 9  & 18 & 15 \\
            $\mnnone(k)$    & 10 & 7  & 17 & 6  & 6  & 4 & 14 & 22 & 20 \\
            $\mnnsub{2}(k)$ & 16 & 13 & 21 & 12 & 12 & 8 & 19 & 24 & 23
          \end{tabular}
        \end{center}
        We have $\mnndashedsub{1} = \{3,8\}$, and $\mnndashedsub{2} = \{8\}$. 
        The type $0$ sites are $1, 2, 3, 5, 9, 11, 15,$ and $18$.
        The type $1$ sites are $4, 6, 7, 10, 14, 17, 20,$ and $22$.
        The type $2$ sites are $8, 12, 13, 16, 19, 21, 23,$ and $24$.
        Sites $10, 16, 17$ are dash sites.
    \end{example}
    We record a simple fact that will be important later. 
    \begin{lemma} \label{lem_mnndashed_after_one}
      Let $\mnntriple = (\mnntriplefull, \mnndashedfull)$ be the $m$-Catalan code of a point $p \in \mathbb{R}^n$, and suppose $\mnnsub{s}(i)$ is a dash site. Let $s',j$ be such that $\mnnsub{s}(i)+1=\mnnsub{s'}(j)$. Then for any $u \in [-s',m-s]$, we have $\mnnsub{s+u}(i)+1=\mnnsub{s'+u}(j)$ and $\mnnsub{s+u}(i)$ is a dash site.
    \end{lemma}
    \begin{proof}
      It is easy to show from the definition of dash site that the hypothesis implies $p_i+s = p_{j}+s'$, which trivially implies $p_i+(s+u) = p_{j}+(s'+u)$ for any $u$. For $u \in [-s',m-s]$ we have $0 \leq s'+u,s+u \leq m$, and so in $W_p$ we have
      \begin{align}
        (p_i+s+u,-(s+u)) \lesslex (p_j+s'+u,-(s'+u)),
      \end{align}
      and one easily sees that there cannot be an element between them, because $\mnnsub{s}(i)+1=\mnnsub{s'}(j)$. The conclusion follows.
    \end{proof}
    \begin{proposition} \label{prop_mnn_bijection}
      Let $\mnntriple = (\mnntriplefull, \mnndashedfull)$ be the $m$-Catalan code of a point $p \in \mathbb{R}^n$. 
      \begin{enumerate}
        \item[(1)] For all $i,j \in [n]$, $p_i \leq p_j$ if and only if $\mnnzero(i) \leq \mnnzero(j)$.
      \end{enumerate}
      And for any $s \in [1,m]$ we have:
      \begin{enumerate}
        \item[(2)] For all $i,j \in [n]$, $p_i < p_j+s$ if and only if $\mnnzero(i) < \mnnsub{s}(j)$.
        \item[(3)] For all $i,j \in [n]$, $p_i = p_j+s$ if and only if every $t \in [\mnnsub{s}(j),\mnnzero(i)-1]$ is a dash site.
      \end{enumerate}
    \end{proposition}
    \begin{proof}
      Items (1) and (2) are clear from the definitions. 

      For the forward direction of (3), we use induction on $s$. For the case $s=1$, it is easy to show that $p_i=p_j+1$ implies $\mnnsub{1}(j)+1=\mnnzero(i)$, and that $\mnnsub{1}(j)$ is a dash site, as desired. For $s>1$, let $s',j'$ be such that $\mnnsub{s}(j)+1=\mnnsub{s'}(j')$. 
      In $W_p$, we must have 
      \begin{align} \label{eq_pj_lesslex}
        (p_j+s,-s) \lesslex (p_{j'}+s',-s') \lesslex (p_i,0),
      \end{align}
      which implies $p_j+s=p_{j'}+s'$ and $s \gt s'$. Therefore $\mnnsub{s}(j)$ is a dash site. We also see from \eref{eq_pj_lesslex} that $p_{j'}+s'=p_i$ so, by the inductive hypothesis, every $t \in [\mnnsub{s'}(j),\mnnzero(i)-1]$ is a dash site. This proves the claim.

      For the backward direction of (3), we use (strong) induction on $s$. For the case $s=1$, it is clear that $\mnnsub{1}(j)+1=\mnnzero(i)$. Since $\mnnsub{1}(j)$ is a dash site, we know there exists $j'\in[n]$ with $p_{j'}+1 = p_i$. It is easy to see that this implies $\mnnsub{1}(j')+1=\mnnzero(i)$, whence $\mnnsub{1}(j')=\mnnsub{1}(j)$. It is obvious that $\mnnsub{1}(j')=\mnnsub{1}(j)$ implies $p_j=p_{j'}$, and the claim follows. 

      For $s>1$, let $j',s'$ be such that $\mnnsub{s}(j)+1=\mnnsub{s'}(j')$. First, since $\mnnsub{s}(j)$ is a dash site, we have $s' \lt s$, so by the inductive hypothesis, we have $p_{j'}+s'=p_i$. On the other hand, by \lemref{lem_mnndashed_after_one}, we must have $\mnnsub{s-s'}(j)+1=\mnnsub{0}(j')$ and $\mnnsub{s-s'}(j)$ is a dash site. Applying the inductive hypothesis again gives $p_{j}+s-s' = p_{j'}$. Adding $s'$ to both sides proves $p_{j}+s=p_i$ as desired.
    \end{proof}
    For $p \in \mathbb{R}^n$, let $\emphnewmath{\mcatface{p}}$ denote the unique $m$-Catalan face in $\mathbb{R}^n$ containing the point $p$.
    \begin{corollary} \label{cor_mcatface_eta}
      For two points $p,q \in \mathbb{R}^n$, 
      \begin{align} \label{} 
        \mnntriplesub{p} = \mnntriplesub{q} \iff \mcatface{p} = \mcatface{q}.
      \end{align}
    \end{corollary}
    \begin{proof}
      Follows from \propref{prop_mnn_bijection}.
    \end{proof}
    We are now ready to carry out the construction of the tree $T$, given a $m$-Catalan face $\canonicalmcatface$. We will denote this construction $\canonicalmcatface \mapsto T$ by $\mcattreeconstruct$. 

    For the remainder of this section, we let $\mnntriple_{\canonicalmcatface} = (\mnntriplefull, \mnndashedfull)$ denote the $m$-Catalan code of any point in $\canonicalmcatface$. By \corref{cor_mcatface_eta}, the choice of point is immaterial. 

    A \emphnew{budding $(m+1)$-ary tree} is an $(m+1)$-ary tree whose nodes are labeled by disjoint subsets of $[n]$ (not necessarily a partition of $[n]$), with a choice of solid/dashed for each edge, and with a marked subset of leaves, referred to as \emphnew{buds}. Leaves that are not buds are called \emphnew{true leaves}. We construct $T$ through a sequence of budding $(m+1)$-ary trees as illustrated in \figref{fig_matchmap_example}. We establish some vocabulary regarding these trees and operations on them: 

    \noindent \textbf{Vocabulary:} Let $R$ be a budding $(m+1)$-ary tree.
    \begin{itemize}
      \item \emphspecial{first bud}: Assuming $R$ contains at least one bud, its \specialword{first bud} is its minimum bud according to the $\prec_R$ order.
    \end{itemize}
    \noindent \textbf{Operations:} Let $R$ be a budding $(m+1)$-ary tree.
    \begin{itemize}
       \item \emphspecial{close}: Find the \specialword{first bud} of $R$ and replace it with a true leaf.
       \item \emphspecial{$\text{open}_S$}: With $S \subset [n]$, find the \specialword{first bud} of $R$ and replace it with a node labeled with $S$ that has $m+1$ children, all buds.
       \item \emphspecial{$\text{dashopen}_S$}: Apply \specialword{$\text{open}_S$} and make the new node labeled $S$ a captive node, that is, make dashed the edge connecting the node labeled $S$ to its parent.
    \end{itemize}
    Examples may be seen in \figref{fig_matchmap_example}.

    \begin{figure}[t]
      \centering
      \includegraphicswide{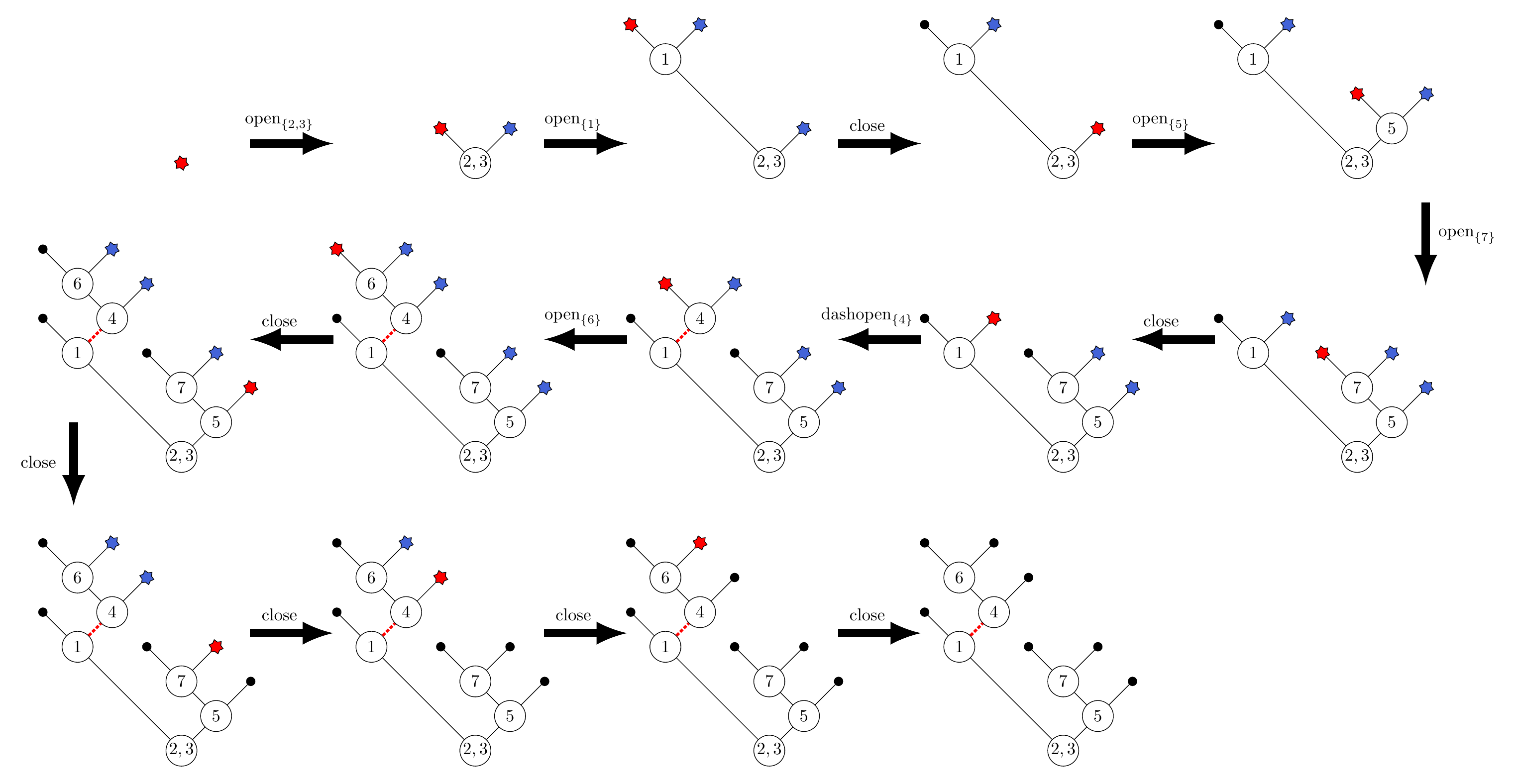}
      \caption[Example of $\mcattreeconstruct$]{The intermediary trees $R_0,\dots,R_{12}$ and final result of $\mcattreeconstruct(\canonicalmcatface)$ where $\canonicalmcatface$ has $\mnntriple_{\canonicalmcatface}$ given by \exampleref{example_cat_mnn}. The stars represent buds. The highlighted bud is the \specialword{first bud}.}
      \label{fig_matchmap_example}
    \end{figure}
    We start with the budding binary tree $R_0$, which consists of just a single bud. For each site $i$ of $\mnntriple$, we construct $R_i$ from $R_{i-1}$ according to the following three cases: 
    \begin{enumerate}
      \item If $i$ is a type $s$ site for $s>0$, then apply a \specialword{close} operation to $R_{i-1}$.
      \item If $i$ is a type 0 site and $\mnnzeroinv(i) \not\subset \mnndashedcup$, then let $S = \mnnzeroinv(i)$ and apply a \specialword{$\text{open}_S$} operation to $R_{i-1}$. 
      \item If $i$ is a type 0 site and $\mnnzeroinv(i) \subset \mnndashedcup$, then apply a \specialword{$\text{dashopen}_S$} operation to $R_{i-1}$, with $S = \mnnzeroinv(i)$. 
    \end{enumerate}
    Finally $R_h$, the last tree, will have exactly one bud remaining (proved below). Apply a \specialword{close} operation, and define $\mcattreeconstruct(\canonicalmcatface)$ to be the resulting tree. We will soon show that $T = \mcattreeconstruct(\canonicalmcatface)$ is a well-defined \mtree, and that $\mcatmap(T) = \canonicalmcatface$ (that is, $\canonicalmcatface$ satisfies the conditions \mcatconditions of \thmref{thm_mcat_bijection}).
    \begin{example} \label{example_mcat_matchmap}
      As illustrated in \figref{fig_matchmap_example} the $\mnntriple$ described in \exampleref{example_cat_mnn} is mapped by $\mcattreeconstruct$ to the $[7]$-decorated binary tree in \figref{fig_seven_decorated_tree}. The reader can verify that the $\mnntriple$ described in \exampleref{example_mcat_mnn} is mapped by $\mcattreeconstruct$ to the $[9]$-decorated 3-ary tree in \figref{fig_nine_decorated_tree}.
    \end{example}
    To prove that this construction is well-defined, there are two main issues to resolve. First, we need to prove that there are always enough buds so that $R_{i}$ is well-defined for all $i$. Second, we need to prove that the final tree $T$ satisfies the definition of \mtree, namely, that any dashed edges are cadet edges, and no buds remain. We will then show that this construction produces a tree $T$ such that $\mcatmap(T) = \canonicalmcatface$, the $m$-Catalan face with which we started. We will use the following definition and lemma. 
    \begin{definition}
      Let $R$ be a budding $(m+1)$-ary tree, let $k \in [n]$, and $s \in [0,m]$. We say $k$ is \emphnew{$s$-open} in $R$ if $R$ has a node containing $k$ and if the child of rank $s$ of that node is a bud. We say $k$ is \emphnew{$s$-closed} in $R$ if $R$ has a node containing $k$ and if the child of rank $s$ of that node is not a bud (either a node or true leaf).
    \end{definition}
    \begin{lemma} \label{lem_m_sharp_bud_ri}
      Let $R_0,\dots,R_h$ be the intermediate trees in the construction of $T = \mcattreeconstruct(\canonicalmcatface)$. For $k \in [n]$ and $s \in [0,m]$ and $i \in [h]$, we have:
      \begin{enumerate}
        \item $k$ is $s$-open in $R_i$   if and only if $\mnnzero(k) \leq i \leq \mnnsub{s}(k)$
        \item $k$ is $s$-closed in $R_i$ if and only if $\mnnsub{s}(k) < i$
      \end{enumerate}
    \end{lemma}
    \begin{proof}
      For $i \in [h]$, and $s \in [0,m]$, let 
      \begin{align}
        \mathcal{O}_s(R_i) &:= \{k \in [n] \given \text{$k$ is $s$-open in $R_i$}\},
      \end{align}
      and
      \begin{align}
        \mathcal{D}_s(R_i) &:= \{k \in [n] \given \text{$k$ is $s$-closed in $R_i$}\}.
      \end{align}
      We must show that 
      \begin{align} \label{eq_open_s_ri}
        \mathcal{O}_s(R_i) &= \{k \in [n] \given \mnnzero(k) \leq i \leq \mnnsub{s}(k)\}, 
      \end{align}
      and
      \begin{align}
        \mathcal{D}_s(R_i) &= \{k \in [n] \given \mnnsub{s}(k)+1 \leq i\}.
      \end{align}
      Since $k$ becomes $s$-closed as soon as $k$ becomes not $s$-open, it is enough to just show \eref{eq_open_s_ri}.
      By induction on $i \in [h]$, it suffices to show
      \begin{align} 
        \mathcal{O}_s(R_i) = \mathcal{O}_s(R_{i-1}) \smallsetminus \mnnsubinv{s}(i-1) \cup \mnnzeroinv(i), \label{eq_whosopen_induction_s} 
      \end{align}
      the claim \eref{eq_open_s_ri} being clearly true for $R_1$.
      There are two cases:

      \caseif{0}{$i-1$ is a type 0 site.} 
        Then $R_{i-1}$ was obtained from $R_{i-2}$ by replacing its \specialword{first bud} by a node with label $\mnnzeroinv(i-1)$ and $m+1$ buds. 
        Since $\prec_{R_{i-1}}$ obeys \precconditiontwo, the child of rank $0$ of the new node must be the \specialword{first bud} of $R_{i-1}$. 
        It follows that $R_i$ is obtained from $R_{i-1}$ by replacing the child of rank $0$ of the node with label $\mnnzeroinv(i-1)$, either by a node with label $\mnnzeroinv(i)$ if $i$ is a type 0 site, or simply a leaf if $i$ is a type $u>0$ site.  In either case, equation \eref{eq_whosopen_induction_s} is satisfied (in this case, $\mnnsubinv{s}(i-1)=\varnothing$).

      \caseif{1}{$i-1$ is a type $u$ site for $u>0$.} 
        Let $k$ be such that $\mnnsub{u}(k) = i-1$. We claim that for any $z$ and $u'$ such that $z \in \mathcal{O}_{u'}(i-1)$ we have 
        \begin{align} \label{eq_firstbud_rim1}
          v_{R_{i-1}}^{u}(k) \preceq_{R_{i-1}} v_{R_{i-1}}^{u'}(z).
        \end{align}
        It follows that the \specialword{first bud} of $R_{i-1}$ is the child of rank $u$ of the node with label $\mnnsubinv{u}(i-1)$. Therefore $R_i$ is obtained from $R_{i-1}$ by replacing this bud by either a node if $i$ is type 0 or simply a leaf for all other types. In either case \eref{eq_whosopen_induction_s} is satisfied. Thus it remains only to prove \eref{eq_firstbud_rim1}. 

        It follows by the inductive hypothesis that $\mathcal{O}_0(R_{i-1}) = \varnothing$, and so we have $u'>0$. 
        We may also assume that $z \notin \mathcal{O}_{u'-1}(i-1)$, otherwise replace $u'$ by $u'-1$. 
        It follows by the inductive hypothesis with $s=u'$ that $i-1 \leq \mnnsub{u'}(z)$. Since $i-1 = \mnnsub{u}(k)$ we have 
        \begin{align}
          \mnnsub{u}(k) \leq \mnnsub{u'}(z),
        \end{align}
        and therefore
        \begin{align} \label{eq_k_um1_z_upm1}
          \mnnsub{u-1}(k) \leq \mnnsub{u'-1}(z).
        \end{align}
        Since $z \notin \mathcal{O}_{u'-1}(i-1)$, we know $\mnnsub{u'-1}(z)+1 \leq i-1$. Therefore, by the inductive hypothesis (strong induction with $s=u'-1$) we know that $\mnnsub{u'-1}(z)+1$ is the first site such that $z$ is $(u'-1)$-closed. Similarly, $\mnnsub{u-1}(k)+1$ is the first site such that $k$ is $(u-1)$-closed. Combined with \eref{eq_k_um1_z_upm1}, this shows that $k$ became $(u-1)$-closed before $z$ was $(u'-1)$-closed, and therefore 
        \begin{align} \label{eq_firstbud_rim1_helper}
          v_{R_{i-1}}^{u-1}(k) \preceq_{R_{i-1}} v_{R_{i-1}}^{u'-1}(z).
        \end{align}
        Then \eref{eq_firstbud_rim1} follows from \eref{eq_firstbud_rim1_helper} and \precconditionthree.
    \end{proof}
    We now show that $\mcattreeconstruct$ well-defined:
    \begin{corollary}
      Let $R_0, \dots, R_h$ be the intermediate trees in the construction of $T = \mcattreeconstruct(\canonicalmcatface)$. 
      \begin{enumerate}
        \item Every $R_i$ contains a positive number of buds, and $R_h$ contains exactly one bud. 
        \item Any edges that are dashed in $R_i$ are cadet edges.
      \end{enumerate}
      Thus, $T = \mcattreeconstruct(\canonicalmcatface)$ is a well-defined \mtree.
    \end{corollary}
    \begin{proof}
      For (1): first $R_0$ by definition contains one bud. For $i \in [h]$, if $i$ is a type 0 site, then $\mnnzeroinv(i) \neq \varnothing$, and by \lemref{lem_m_sharp_bud_ri} any $k \in \mnnzeroinv(i)$ is $0$-open in $R_i$. Therefore $R_i$ contains at least one bud. If $i$ is a type $s$ site, then $\mnnsubinv{s}(i) \neq \varnothing$, and by \lemref{lem_m_sharp_bud_ri} any $k \in \mnnsubinv{s}(i)$ is $s$-open in $R_i$, and so again $R_i$ contains at least one bud. It is clear from \lemref{lem_m_sharp_bud_ri} that $R_{h}$ has exactly one bud.

      For (2), \specialword{dashopen} operations occur when creating trees $R_i$ where $i$ satisfies $\mnnzeroinv(i) \subset \mnndashedcup$. It is easy to show, from the definition of $\mnnsub{s}$ and $\mnndashedsub{s}$, that $i-1$ must be a dash site, and therefore is of type $s>0$.
      For any $j \in \mnnsubinv{s}(i-1)$, it follows from \lemref{lem_m_sharp_bud_ri} that $j$ is $s$-open in $R_{i-1}$ and is $s$-closed in $R_{i}$, which means the \specialword{first bud} of $R_{i-1}$ is the child of rank $s$ of the node containing $j$. 
      Therefore, the \specialword{dashopen} operation will create a captive node at the child of rank $s$ of $j$. 
      This proves that only edges of rank $>0$ will ever be dashed. 

      From \lemref{lem_mnndashed_after_one}, for every $u \in [s,m]$, the site immediately following $\mnnsub{u}(j)$ is a type $s'$ site for some $s'>0$. It follows from this and \lemref{lem_m_sharp_bud_ri} that the children of $v_T(j)$ of rank $>s$ will all be leaves. This shows that any dashed edges are indeed cadet edges. 
    \end{proof}
    Now that $\mcattreeconstruct$ is well-defined, we can show that $\mcatmap \circ \mcattreeconstruct$ is the identity.
    \begin{proposition} \label{prop_mcat_surjective}
      With $T = \mcattreeconstruct(\canonicalmcatface)$, for any point $p \in \canonicalmcatface$ we have
      \begin{enumerate}
        \item For all $i,j \in [n]$, $p_i \leq p_j$ if and only if $v_T(i) \preceq_T v_T(j)$.
        \item For all $i,j \in [n]$, and $s \in [1,m]$, $p_i < p_j+s$  if and only if $v_T(i) \prec_{T} v^s_T(j)$.
        \item For all $i,j \in [n]$, and $s \in [1,m]$, $p_i = p_j+s$ if and only if $v_T(i) = \nextlive(v^s_T(j))$ and there is a dashed path from $v_T(j)$ to $v_T(i)$.
      \end{enumerate}
      Whence, $\mcatmap(T) = \canonicalmcatface$.
    \end{proposition}
    \begin{proof}
      Part (1) is clear from the fact that the construction of $T$ takes place in the $\prec_T$ order. 

      For (2): By \lemref{lem_m_sharp_bud_ri}, the first appearance of a node containing $i$ is in the tree $R_{\mnnzero(i)}$, and the first tree in which $j$ is $s$-closed is $R_{\mnnsub{s}(j)+1}$.
      Thus, $v_T(i) \prec_{T} v^{s}_T(j)$ if and only if $\mnnzero(i) < \mnnsub{s}(j)+1$.
      Since a site cannot be both type $0$ and type $s>0$, $\mnnzero(i)$ cannot equal $\mnnsub{s}(j)$, so $\mnnzero(i) < \mnnsub{s}(j)+1$ is equivalent to $\mnnzero(i) < \mnnsub{s}(j)$.
      By \propref{prop_mnn_bijection}, $\mnnzero(i) < \mnnsub{s}(j)$ if and only if $p_i \lt p_j+s$, as desired.

      For (3): By \propref{prop_mnn_bijection}, we have $p_i=p_j+s$ if and only if every $t \in [\mnnsub{s}(j),\mnnzero(i)-1]$ is a dash site. 
      Let $L$ be the length of this interval; we will prove the claim by induction on $L$. 
      Assume first that $L=1$, that is, $\mnnsub{s}(j)$ is a dash site and $\mnnsub{s}(j)+1=\mnnzero(i)$. 
      It then follows from the definition of $\mcattreeconstruct$ and \lemref{lem_m_sharp_bud_ri} that $T$ has a dashed edge of rank $s$ with parent $v_T(j)$ and child $v_T(i)$.
      Thus, in this case $v_T(i) = v^s_T(j)$ and so the claim is certainly true for $L=1$.

      If $L>1$, let $s',j'$ be such that the interval of dash sites starts $[\mnnsub{s}(j),\mnnsub{s'}(j'),\dots]$. In particular, this means every $t \in [\mnnsub{s'}(j'),\dots, \mnnzero(i)-1]$ is a dash site. By the inductive hypothesis, and \propref{prop_mnn_bijection}, we have $v_T(i) = \nextlive(v^{s'}_T(j'))$ and there is a dashed path from $v_T(j')$ to $v_T(i)$. 

      On the other hand, by \lemref{lem_mnndashed_after_one}, $\mnnsub{s-s'}(j)+1=\mnnsub{0}(j')$, and $\mnnsub{s-s'}(j)$ is a dash site. Therefore from the $L=1$ case proved above we have 
      \begin{align} \label{eq_vtj2_eq_nextnode_vstarj}
        v_T(j') = v^{s-s'}_T(j),
      \end{align}
      and there is a dashed edge between $v_T(j)$ and $v_T(j')$. 

      Combining facts, we have a dashed path from $v_T(j)$ to $v_T(j')$ and a dashed path from $v_T(j')$ to $v_T(i)$, and therefore we have one from $v_T(j)$ to $v_T(i)$, as desired. 

      It remains to show $v_T(i) = \nextlive(v^s_T(j))$. So far, we have $v_T(i) = \nextlive(v^{s'}_T(j'))$ and \eref{eq_vtj2_eq_nextnode_vstarj}. So the claim may be rewritten as 
      \begin{align} \label{eq_next_child_eq_next_child}
        \nextlive (\text{the child of $v_T^{s-s'}(j)$ of rank $s'$}) = \nextlive ( v^{s}_T(j) ).
      \end{align}
      Note that, since $v_T(j')$ is a captive child of $v_T(j)$ of rank $s-s'$, the vertex $v^{s}_T(j)$ in question must be a dead leaf. With this in mind, equation \eref{eq_next_child_eq_next_child} is a consequence of \lemref{lem_next_of_dead_leaf} (in the notation of \lemref{lem_next_of_dead_leaf}, we have $v=v^{s-s'}_T(j)$ with $\ell=v^{s}_T(j)$ and $k=s'$).
    \end{proof}
    \propref{prop_mcat_surjective} proves that $\mcatmap$ is surjective, which completes the proof of \thmref{thm_mcat_bijection}.

\section{Proof that \texorpdfstring{$\mshimap$}{map for m-Shi faces} is a bijection} \label{sec_proof_mshi}
  In this section we will prove \thmref{thm_mshi_bijection}, which implies \thmref{thm_shi_bijection}.
  As in \secref{sec_proof_mcatalan}, we separate this proof into five distinct elements: existence and uniqueness of the $m$-Shi face, injectivity and surjectivity of $\mshimap$, and the grading by dimension. 
  The ``core claim'' of this section is that each $m$-Shi face contains exactly one $m$-Catalan face whose corresponding tree is of Shi type. That there exists at most one is equivalent to the claim that $\mshimap$ is injective. That there exists at least one is equivalent to the claim that $\mshimap$ is surjective.
  The other three elements: existence, uniqueness, and the grading by dimension, are somewhat more straightforward consequences of \thmref{thm_mcat_bijection}, and we start with these. Throughout this section, $T$ is a given \mtree.

  \subsection{Existence, uniqueness, and grading by dimension} \label{subsec_mshi_proof_existence}
    Uniqueness is immediate, since the hyperplanes of the $m$-Shi arrangement are of the form $x_i = x_j+s$ for $s \in [0,m-1]$ for all $i,j \in [n]$, and $x_i = x_j+m$ for $(i,j) \in \pairsilessthanj$. 

    For existence, first let $\canonicalmcatface := \mcatmap(T)$. It is straightforward to show from \thmref{thm_mcat_bijection} that any point $p \in \canonicalmcatface$ will satisfy all of \mshiconditions of \thmref{thm_mshi_bijection}. It follows that the unique $m$-Shi face that contains all points of $\canonicalmcatface$ will satisfy the requirements of \thmref{thm_mshi_bijection}, and so existence is proven.

    Next we will show that for any \mtree $T$ of Shi type,
    \begin{align}
      \text{dim}(\mshimap(T)) = \# \{\text{free nodes of $T$}\},
    \end{align}
    as claimed in \thmref{thm_mshi_bijection}. It is enough to show that 
    \begin{align}
      \text{dim}(\mshimap(T)) = \text{dim}(\mcatmap(T)).
    \end{align}
    We begin with an important remark, which we will refer to again in the following subsection.
    \begin{remark} \label{rem_mshi_face_union_mcat}
      Since the $m$-Shi arrangement is obtained by deleting hyperplanes from the $m$-Catalan arrangement, for every $m$-Shi face $\canonicalmshiface$, there exist $m$-Catalan faces $\canonicalmcatface_1, \dots, \canonicalmcatface_{\ell}$ such that $\canonicalmshiface$ is the disjoint union of the $\canonicalmcatface_{i}$.
      Furthermore, it is clear that 
      \begin{align} \label{eq_dim_mshiface_equals_max}
        \text{dim}(\canonicalmshiface) = \max_{i} \, \text{dim}(\canonicalmcatface_i).
      \end{align}
    \end{remark}
    Let the inverse of $\mcatmap$ be denoted $\emphnewmath{\mcatmapinv}$.
    The main result of this subsection is the following:
    \begin{proposition} \label{prop_mshi_top_dimension}
      Let $\canonicalmshiface$ be any Shi face. Let $\canonicalmcatface_1,\dots,\canonicalmcatface_\ell$ be as in \remref{rem_mshi_face_union_mcat}. If $\canonicalmcatface_{i^*}$ has the property that $\mcatmapinv(\canonicalmcatface_{i^*})$ is of Shi type, then for any $i \in [\ell]$, 
      \begin{align} \label{eq_prop_dim_mshi_representative}
        \text{dim}(\canonicalmcatface_i) \leq \text{dim}(\canonicalmcatface_{i^*}),
      \end{align}
      and therefore $\text{dim}(\canonicalmshiface) = \text{dim}(\canonicalmcatface_{i^*})$.
    \end{proposition}
    \begin{proof}
      Let $T_{i^*} := \mcatmapinv(\canonicalmcatface_{i^*})$ and $T_i := \mcatmapinv(\canonicalmcatface_{i})$. It is easy to see by the properties in \thmref{thm_mcat_bijection} and the fact that $\canonicalmcatface_i$ and $\canonicalmcatface_{i^*}$ belong to the same $m$-Shi face that:
      \begin{enumerate}
        \item Two numbers $a,b \in [n]$ appear in the same node of $T_{i^*}$ if and only if $a,b$ appear in the same node of $T_{i}$.
        \item For $A,B \subset [n]$ and $s \in [1,m-1]$, in $T_{i^*}$ there is a dashed edge of rank $s$ from a node labeled $A$ to a node labeled $B$ if and only if in $T_{i}$ there is a dashed edge of rank $s$ from a node labeled $A$ to a node labeled $B$.
        \item For $A,B \subset [n]$, in $T_{i^*}$ there is a dashed edge of rank $m$ that is a descent from a node labeled $A$ to a node labeled $B$ if and only if in $T_{i}$ there is a dashed edge of rank $m$ that is a descent from a node labeled $A$ to a node labeled $B$.
      \end{enumerate}
      It follows from (1) that $T_{i^*}$ and $T_{i}$ have the same number of nodes. 
      Since $T_{i^*}$ is assumed to be of Shi type, all rank $m$ edges are descents. Therefore, from (2) and (3) we see that any captive node in $T_{i^*}$ must also be a captive node in $T_{i}$.
      Thus the number of captive nodes in $T_{i^*}$ is less than or equal to the number of captive nodes in $T_{i}$. 
      Since the total number of nodes is the same, we have that the number of free nodes of $T_{i^*}$ is greater than or equal to the number of free nodes of $T_{i}$. 
      Applying the dimension grading of \thmref{thm_mcat_bijection}, we obtain \eref{eq_prop_dim_mshi_representative}. 

      The last statement of the proposition follows immediately from \eref{eq_prop_dim_mshi_representative} and \eref{eq_dim_mshiface_equals_max}.
    \end{proof}
  \subsection{Proof that \texorpdfstring{$\mshimap$}{map for m-Shi faces} is injective} \label{subsec_mshi_proof_injective}
    We will prove that $\mshimap$ is injective by showing that the conditions on $T$ in \thmref{thm_mshi_bijection} are enough to define $T$. Suppose $T$ and $T'$ are \mtrees of Shi type and that the following five conditions hold.
    \begin{enumerate}
      \item[\mshitreeconditionone] For all $i,j \in [n]$, $v_T(i) \preceq_T v_T(j)$ if and only if $v_{T'}(i) \preceq_{T'} v_{T'}(j)$.
      \item[\mshitreeconditiontwo] For all $i,j \in [n]$, and $s \in [1,m-1]$, $v_T(i) \prec_T v_{T}^{s}(j)$ if and only if $v_{T'}(i) \prec_{T'} v_{T'}^{s}(j)$.
      \item[\mshitreeconditionfour] For all $(i,j) \in \pairsilessthanj, v_T(i) \prec_T v_{T}^{m}(j)$ if and only if $v_{T'}(i) \prec_{T'} v_{T'}^{m}(j)$.
      \item[\mshitreeconditionthree] For all $i,j \in [n]$ and $s \in [1,m-1]$, $v_T(i) = \nextlive(v^s_T(j))$ and there is a dashed path from $v_T(j)$ to $v_T(i)$ if and only if $v_{T'}(i) = \nextlive(v^s_{T'}(j))$ and there is a dashed path from $v_{T'}(j)$ to $v_{T'}(i)$.
      \item[\mshitreeconditionfive] For all $(i,j) \in \pairsilessthanj, v_T(i) = \nextlive(v^m_T(j))$ and there is a dashed path from $v_T(j)$ to $v_T(i)$ if and only if $v_{T'}(i) = \nextlive(v^m_{T'}(j))$ and there is a dashed path from $v_{T'}(j)$ to $v_{T'}(i)$.
    \end{enumerate}
    We will show $T=T'$. First, it is clear from \mshitreeconditionone that two numbers $a,b \in [n]$ appear in the same node of $T$ if and only if they appear in the same node of $T'$. This proves that $T$ and $T'$ have the same number of nodes, and since they are $(m+1)$-ary trees, they also have the same number of total vertices, $p$. 

    Let $v_1 \prec_T \dots \prec_T v_{p}$ be the vertices of $T$, and let $w_1 \prec_{T'} \dots \prec_{T'} w_{p}$ be the vertices of $T'$. As in \subsecref{subsec_mcat_proof_injective}, to show $T = T'$, it is enough to show:
    \begin{enumerate}[label=(\roman*)]
      \item $v_i \cong w_i$ for all $i \in [p]$, 
      \item $\rho(v_i) = \rho(w_i)$ for all $i \in [p]$, and
      \item $v_i$ is a captive node if and only if $w_i$ is a captive node. 
    \end{enumerate}
    We prove (i) by induction on $i$. The base case $i=1$ is covered by Case A below.

    \caseif{A}{$v_i$ is a node and $w_i$ is a node.}
      Since \mshitreeconditionone and \mcattreeconditionone are the same, this case is identical to the Case A in \subsecref{subsec_mcat_proof_injective}. Using the same argument verbatim, we obtain $v_i \cong w_i$.

    \caseif{B}{$v_i$ is a leaf and $w_i$ is a leaf.} 
      In this case, $v_i \cong w_i$ trivially.

    \caseif{C}{$v_i$ is a leaf and $w_i$ is a node.} 
      We will show that this case is not possible. The basic setup is the same as Case C in \subsecref{subsec_mcat_proof_injective}. Namely, we define the following:
      \begin{itemize}
        \item Let $b$ be any element of $\text{label}(w_i)$.
        \item Let $u := \nextnode(v_i)$. 
        \item Let $s$ be the rank of $u$. We know $s>0$.
        \item Let $v_h$ be the parent of $u$.
        \item Let $a$ be any element of $\text{label}(v_h)$.
      \end{itemize}
      Moreover, we have the following facts:
      \begin{align} \label{eq_mshi_v_i_sandwich} 
        v_T(a) = v_{h} \prec_{T} v_{i} \prec_{T} u = v_T^s(a) \preceq_{T} v_T(b),
      \end{align}
      and
      \begin{align} \label{eq_mshi_vtpb_prec_kthelement}
        v_{T'}(b) \prec_{T'} v_{T'}^{s}(a).
      \end{align}
      If $s \in [1,m-1]$ then the last relation of \eref{eq_mshi_v_i_sandwich} and \eref{eq_mshi_vtpb_prec_kthelement} contradict \mshitreeconditionthree, and we are done. From now on we assume $s=m$. 
      Unlike in the Catalan case, or the $s<m$ case, the equations \eref{eq_mshi_vtpb_prec_kthelement} and the last relation of \eref{eq_mshi_v_i_sandwich} do not directly create a contradiction, since we would require $b<a$ to apply \mshitreeconditionfour. 
      In order to create the contradiction, we will now find $b^*$ in $\text{label}(v_T(b))$ and $a^*$ in $\text{label}(v_T(a))$ such that $b^* < a^*$, which then leads to a contradiction with \mshitreeconditionfour. 
      The key fact is the following: 
      \begin{align} \label{eq_mshi_u_eq_vtb} 
        u = v_T(b).
      \end{align}
      To prove \eref{eq_mshi_u_eq_vtb} it is enough to show $v_T(b) \preceq_{T} u$, since the last relation of \eref{eq_mshi_v_i_sandwich} already gives $u \preceq_{T} v_T(b)$. Let $z$ be any element of $u$. By the definition of $u$, $z$ does not appear in the labels of any of $v_1,\dots,v_i$. By the inductive hypothesis, this means $z$ does not appear in the labels of any of $w_1,\dots,w_{i-1}$. Since $z$ has to appear somewhere in $T'$, we have shown 
      \begin{align} \label{eq_mshi_wi_preceqtp_vtpz}
        w_i \preceq_{T'} v_{T'}(z).
      \end{align}
      Since $w_i = v_{T'}(b)$, \eref{eq_mshi_wi_preceqtp_vtpz} is equivalent to 
      \begin{align} \label{eq_mshi_vtpb_preceqtp_vtpz}
        v_{T'}(b) \preceq_{T'} v_{T'}(z).
      \end{align}
      By \mshitreeconditionone, \eref{eq_mshi_vtpb_preceqtp_vtpz} implies
      \begin{align} \label{eq_mshi_vtb_preceqt_vtz}
        v_{T}(b) \preceq_{T} v_{T}(z).
      \end{align}
      and since $v_{T}(z) = u$, we have shown $v_T(b) \preceq_{T} u$. Before we showed $u \preceq_{T} v_T(b)$, so we have proven \eref{eq_mshi_u_eq_vtb}. 

      The importance of \eref{eq_mshi_u_eq_vtb} is that it proves $v_T(b)$ is the child of rank $m$ of $v_T(a)$, and therefore the edge connecting them is a descent. Let $a^* := \max \left(\text{label}(v_T(a))\right)$ and $b^* := \min \left(\text{label}(v_T(b)) \right)$. Since $T$ is of Shi type, we have $b^* < a^*$. We have $v_T(a) = v_T(a^*)$ and $v_T(b) = v_T(b^*)$, so by \mshitreeconditionone we also have $v_{T'}(a) = v_{T'}(a^*)$ and $v_{T'}(b) = v_{T'}(b^*)$. Since $u = v_T(b)$, we can rewrite the penultimate relation of \eref{eq_mshi_v_i_sandwich} as
      \begin{align} \label{eq_mshi_vtbstar_eq_kthelement}
          v_{T}^{m}(a^*) =  v_{T}(b^*).
      \end{align}
      We can also rewrite \eref{eq_mshi_vtpb_prec_kthelement} as
      \begin{align} \label{eq_mshi_vtpbstar_prec_kthelement}
        v_{T'}(b^*) \prec_{T'} v_{T'}^{m}(a^*).
      \end{align}
      Finally we have a clear contradiction between \eref{eq_mshi_vtbstar_eq_kthelement}, \eref{eq_mshi_vtpbstar_prec_kthelement}, and \mshitreeconditionfour.

      \caseif{D}{$v_i$ is a node and $w_i$ is a leaf.} This is the same as Case C by symmetry.

      This completes our proof of (i), that $v_i \cong w_i$ for all $i \in [p]$. \lemref{lem_mcat_prect_inductive} shows that (i) implies (ii). 

      It remains to show (iii), which is straightforward: suppose $v_i$ is a captive node of $T$, and let $s$ be the rank of $v_i$. By \defref{def_mtrees}, $s \in [1,m]$. If $s \in [1,m-1]$, then since $v_i \cong w_i$, \mshitreeconditionthree implies that $w_i$ must be a captive node in $T'$, as desired. 

      If $s=m$, then since $T$ is of Shi type, the edge connecting $v_i$ to its parent is a descent. 
      Therefore there exist $i \in \text{label}(v_i)$ and $j \in \text{label}(\text{parent}(v_i))$ such that $(i,j) \in \pairsilessthanj$.
      By (i) and (ii), $i \in \text{label}(w_i)$ and $j \in \text{label}(\text{parent}(w_i))$, so \mshitreeconditionfive shows $w_i$ is captive, as desired.
      Thus we obtain $T = T'$.

  \subsection{Proof that \texorpdfstring{$\mshimap$}{map for m-Shi faces} is surjective} \label{subsec_mshi_proof_surjective}
    Let $\canonicalmshiface$ be an arbitrary $m$-Shi face. We will show that there exists a tree $T$ of Shi type such that $\mshimap(T) = \canonicalmshiface$.

    By \remref{rem_mshi_face_union_mcat} there exist $m$-Catalan faces $\canonicalmcatface_1, \dots, \canonicalmcatface_{\ell}$ such that $\canonicalmshiface$ is the disjoint union of the $\canonicalmcatface_{i}$. Recall that each $m$-Catalan face is defined by choices of inequality or equality for each hyperplane of the $m$-Catalan arrangement. For any $m$-Catalan face $\canonicalmcatface$, let $\shirank(\canonicalmcatface)$ be the number of inequalities in the definition of $\canonicalmcatface$ of the form $x_i +m > x_j$ for $(i,j) \in \pairsilessthanj$. Among the $\canonicalmcatface_{i}$, choose one $\canonicalmcatface^{*}$ such that $\shirank(\canonicalmcatface^{*})$ is maximal. 

    Let $T$ be $\mcatmapinv(\canonicalmcatface^{*})$. 
    We will show that $T$ is of Shi type. 
    Note that this implies $\mshimap(T) = \canonicalmshiface$, which proves that $\mshimap$ is surjective. 

    Suppose for contradiction that $T$ is not of Shi type, that is, there exists an edge of rank $m$ with parent node $v$ labeled $A$ and child node $w$ labeled $B$ such that $a < b$ for all $a \in A$ and $b \in B$. 
    It follows from \lemref{lem_next_of_dead_leaf}, \precconditiontwo, and the fact that $w$ has rank $m$ that the vertex immediately preceding $w$ in the $\prec_T$ order must be a \live leaf, $\ell$ (see \defref{def_dead_leaves}). 
    We construct a new tree $T'$ in the following manner. Let the sub-trees that are children of $w$ be $T_0,\dots,T_m$. Replace the child of rank $m$ of $v$ with $T_0$. Then replace $\ell$ by a node with label $B$, with a leaf for its child of rank $0$, and with the sub-trees $T_1,\dots,T_m$ for its other children. If there was a dashed edge from $w$ to $T_i$ for $i \in [m]$, then the edge from the new node labeled $B$ to $T_i$ is also dashed.
    Since $\ell$ is a \live leaf, we know that the resulting tree $T'$ satisfies the definition of \mtree.

    Therefore, $T'$ corresponds to another $m$-Catalan face $\canonicalmcatface' := \mcatmapinv(T')$, and we claim that $\canonicalmshiface$ contains this new $\canonicalmcatface'$ and that $\shirank(\canonicalmcatface') > \shirank(\canonicalmcatface^{*})$. 
    It is not too hard to show from \thmref{thm_mcat_bijection} that the only change between $\canonicalmcatface^{*}$ and $\canonicalmcatface'$ is that the inequalities $x_{a}+m \leq x_b$ for all $a \in A$ and $b \in B$ have flipped to $x_{a}+m>x_{b}$. 
    This follows from the fact that the $\prec_T$ order has not changed except for the fact that the node $B$ in $T'$ now precedes the child of rank $m$ of the node labeled $A$ in $T'$. 

    Furthermore, for all $a \in A$ and $b \in B$ we have $a<b$, so the modified inequalities are not hyperplanes in the $m$-Shi arrangement, and therefore $\canonicalmshiface$ contains both $\canonicalmcatface^{*}$ and $\canonicalmcatface'$. Since the inequalities $x_{a}+m>x_{b}$ count towards $\shirank(\canonicalmcatface')$, we have $\shirank(\canonicalmcatface') > \shirank(\canonicalmcatface^{*})$. This contradicts the choice of $\canonicalmcatface^{*}$, and therefore we have shown $T$ must be of Shi type. 

    Recall that $\mshimap(T)$ is the unique Shi face that contains $\mcatmap(T)$. Since we have shown that $\canonicalmshiface$ contains $\mcatmap(T)$, it follows that $\mshimap(T) = \canonicalmshiface$ as desired. This concludes the proof that $\mshimap$ is surjective. We have now proved the entirety of \thmref{thm_mshi_bijection}.

\section{Bijections from \texorpdfstring{$m$-Shi}{m-Shi} faces to marked functions} \label{sec_additional_bijections}
  The original counting formula of Athanasiadis suggests that Shi faces correspond bijectively to the set of marked functions defined in \eref{eq_athanasiadis_shi_faces}. In this section, we explain how to obtain these marked functions from our decorated trees. 
  \subsection{Result for the Shi arrangement} \label{subsec_shi_everything_results}
    Before considering a general $m$, we specialize to the case $m=1$.
    A \emphnew{Cayley tree} is a labeled graph that is connected and acyclic. An \emphnew{internal node} of a Cayley tree is a vertex of degree at least two. Let
    \begin{align}
      \emphnewmath{\mathscr{T}_{n,k}} := \bracketbox{8cm}{$[n]$-decorated binary trees such that all right internal edges are descents, and with $k$ free nodes.}, \\
      \emphnewmath{\mathscr{U}_{n,k}} := \bracketbox{8cm}{Binary trees with $n$ nodes, labeled by $[n]$, such that all right internal edges are descents, with $n-k$ marked nodes, each of which has non-leaf left-children.}, \\
      \emphnewmath{\mathscr{V}_{n,k}} := \bracketbox{8cm}{Cayley trees with $n+1$ vertices, with $n-k$ marked internal node among those labeled $\{1,\dots,n\}$.}, \\
      \emphnewmath{\mathscr{W}_{n,k}} := \bracketbox{8cm}{Pairs $(f,S)$ where $f : [n-1] \to [n+1]$ and $S \subset \text{Im}(F) \cap [n]$ with $|S|=n-k$.}.
    \end{align}
    \begin{theorem} \label{thm_shi_everything}
      For any $n \geq 1$ and $k \in [n]$, the four sets $\mathscr{T}_{n,k}, \mathscr{U}_{n,k}, \mathscr{V}_{n,k}, \mathscr{W}_{n,k}$ are all in bijection with one another. Consequently,
      $$\# \bracketbox{2.5cm}{Shi faces in $\mathbb{R}^n$ of dimension $k$} = |\mathscr{T}_{n,k}| = 
      |\mathscr{U}_{n,k}| = 
      |\mathscr{V}_{n,k}| = 
      |\mathscr{W}_{n,k}|.$$
    \end{theorem}
    The first equals sign in the above theorem was obtained in \thmref{thm_shi_bijection}. We prove \thmref{thm_shi_everything} by describing a sequence of bijections $$\mathscr{T}_{n,k} \isoto \mathscr{U}_{n,k} \isoto \mathscr{V}_{n,k} \isoto \mathscr{W}_{n,k}.$$ One example of the chain of correspondences is shown in \figref{fig_shieverything}. 
    The descriptions of each map will be given for a general $m$ in the following subsections. 
    We briefly describe the case $m=1$ here, omitting some of the details. 
    \begin{figure}[t]
      \includegraphicswide{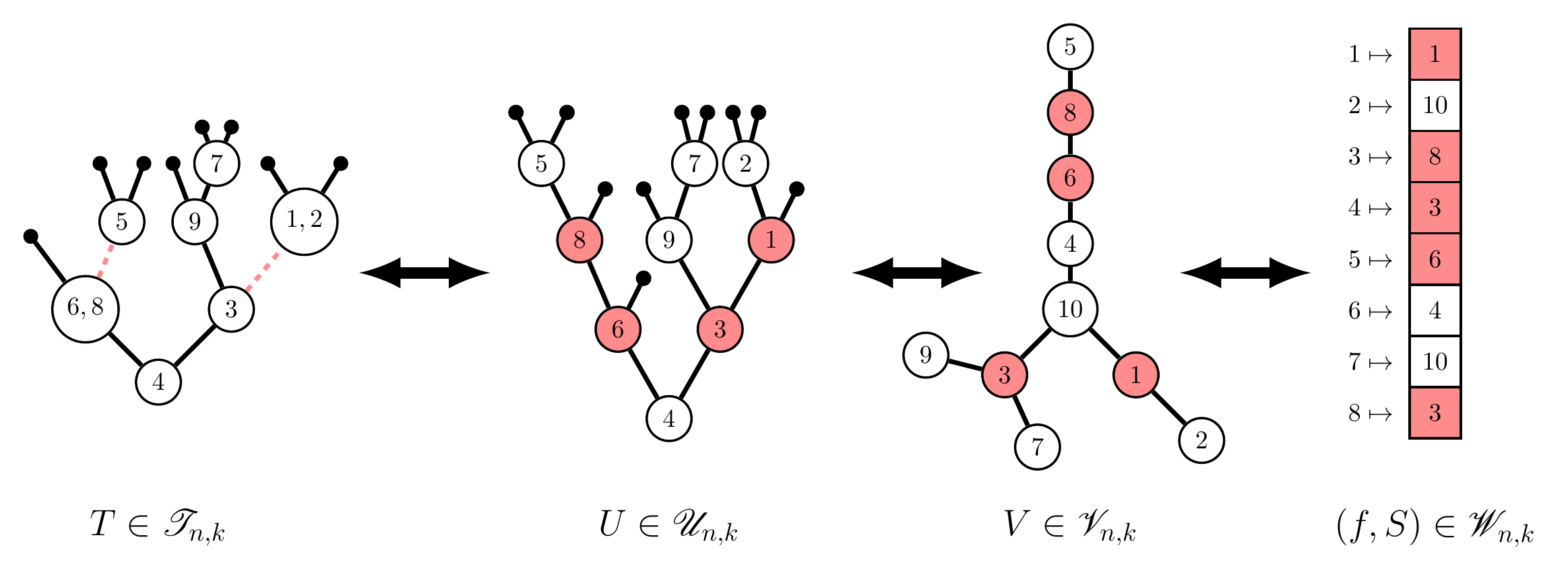}
        \caption[Example of the bijections in \thmref{thm_shi_everything}]{An example of the bijections in \thmref{thm_shi_everything}. In this example, $m=1$, $n=9$, and $k=5$. The element from $\mathscr{W}_{n,k}$ has $S = \{1,3,6,8\}$.}
        \label{fig_shieverything}
    \end{figure}
    The trickiest bijection is the first, $\mathscr{T}_{n,k} \isoto \mathscr{U}_{n,k}$. 
    For $T \in \mathscr{T}_{n,k}$, we apply two local operations (Step 1 and Step 2) to $T$ to obtain its corresponding $U \in \mathscr{U}_{n,k}$. 
    Both operations are summarized in \figref{fig_shi_bijections_one_steps}. 
    \begin{figure}[t]
      \includegraphicswide{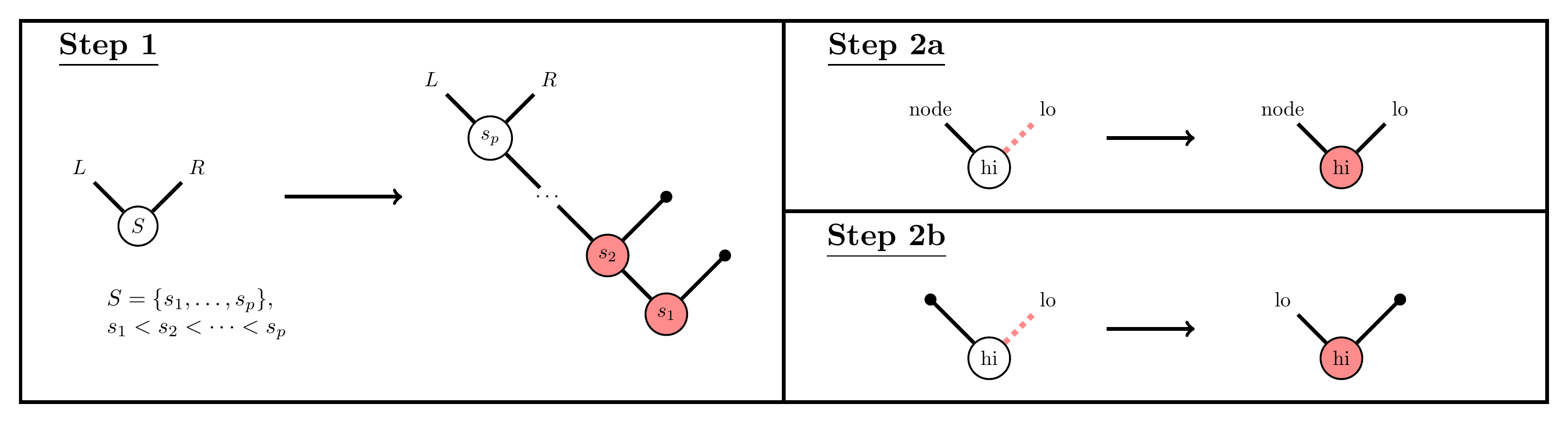}
        \caption[The bijection $\mathscr{T}_{n,k} \isoto \mathscr{U}_{n,k}$]{The two steps of the bijection $\mathscr{T}_{n,k} \isoto \mathscr{U}_{n,k}$. The notation of ``hi'' and ``lo'' should be understood to mean the edge connecting them is a descent, that is, the vertex labeled ``hi'' contains a number that is greater than the number in the vertex labeled ``lo''.}
        \label{fig_shi_bijections_one_steps}
    \end{figure}
    For Step 1, let $v$ be a node of $T$, with left-child $L$, right-child $R$, and with label $S = \{s_1,s_2,\dots,s_p\}$ with $s_1 < s_2 < \dots < s_p$. If $|S|>1$, then replace $v$ by an increasing left-path of singleton marked nodes labeled by $s_1,s_2,\dots,s_p$, each with right-leaves, except for the last singleton which is left unmarked and keeps the original left and right children of $v$. If the original edge connecting $v$ to $R$ was dashed, then the new edge from $s_p$ to $R$ remains dashed. This completes Step 1. Let $T'$ be the intermediary tree obtained by applying Step 1. 

    For Step 2, we apply a local operation to each node of $T'$ with a dashed right-edge. Let $v$ be such a node. Since dashed edges are internal, $v$ has a node as its right-child, and the corresponding right-edge must be a descent. First convert the right-edge of $v$ to a solid edge, and add $v$ to the set of marked nodes. Now there are two cases: if the left-child of $v$ is a node, we are done with $v$ (case (a)). If the left-child of $v$ is a leaf, then swap the two children of $v$ (case (b)). 

    \begin{figure}[t]
      \includegraphicswide{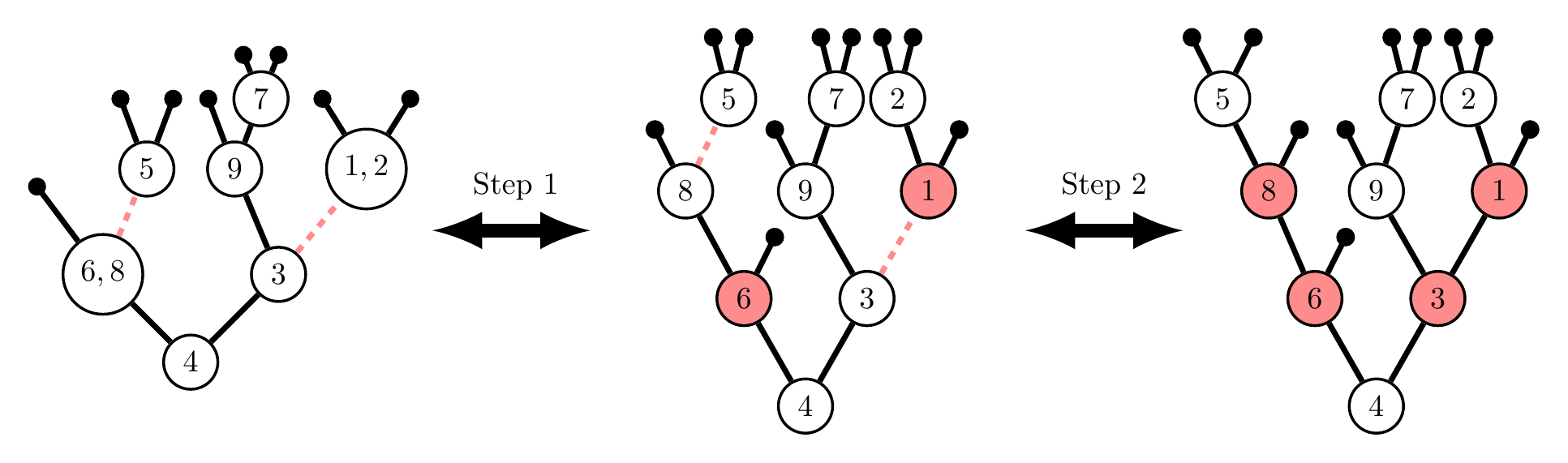}
        \caption[Example of the bijection $\mathscr{T}_{n,k} \isoto \mathscr{U}_{n,k}$]{A full example of steps 1 and 2 of the bijection $\mathscr{T}_{n,k} \isoto \mathscr{U}_{n,k}$.}
        \label{fig_shi_bijections_one_step_one}
    \end{figure}
    It is not too hard to see that the resulting map obtained by applying Steps 1 and 2 is a bijection, so we have established $\mathscr{T}_{n,k} \isoto \mathscr{U}_{n,k}$.

    The second bijection $\mathscr{U}_{n,k} \isoto \mathscr{V}_{n,k}$ is given by a standard bijection from binary trees to plane trees, which has been called the ``natural correspondence'' in \cite[\S2.3.2]{cite_knuth1997theart} (see also \cite[Thm. 1.5.1, (iii) to (ii)]{cite_stanley2015catalannumbers}). In our case, since right-paths are decreasing, we can forget the order of the children, obtaining Cayley trees for $\mathscr{V}_{n,k}$ instead of plane trees.

    The map can be described in three steps, illustrated by example in \figref{fig_shi_bijections_two}.
    For Step 1, delete all leaves. 
    For Step 2, add edges connecting each vertex to all of the vertices in the right path starting from its left-child, and add a node labeled $n+1$, connected to the vertices of the root's right-path (including the root).
    For Step 3, delete all of the original right-edges. 
    This concludes the bijection $\mathscr{U}_{n,k} \isoto \mathscr{V}_{n,k}$.
    \begin{figure}[t]
      \includegraphicswide{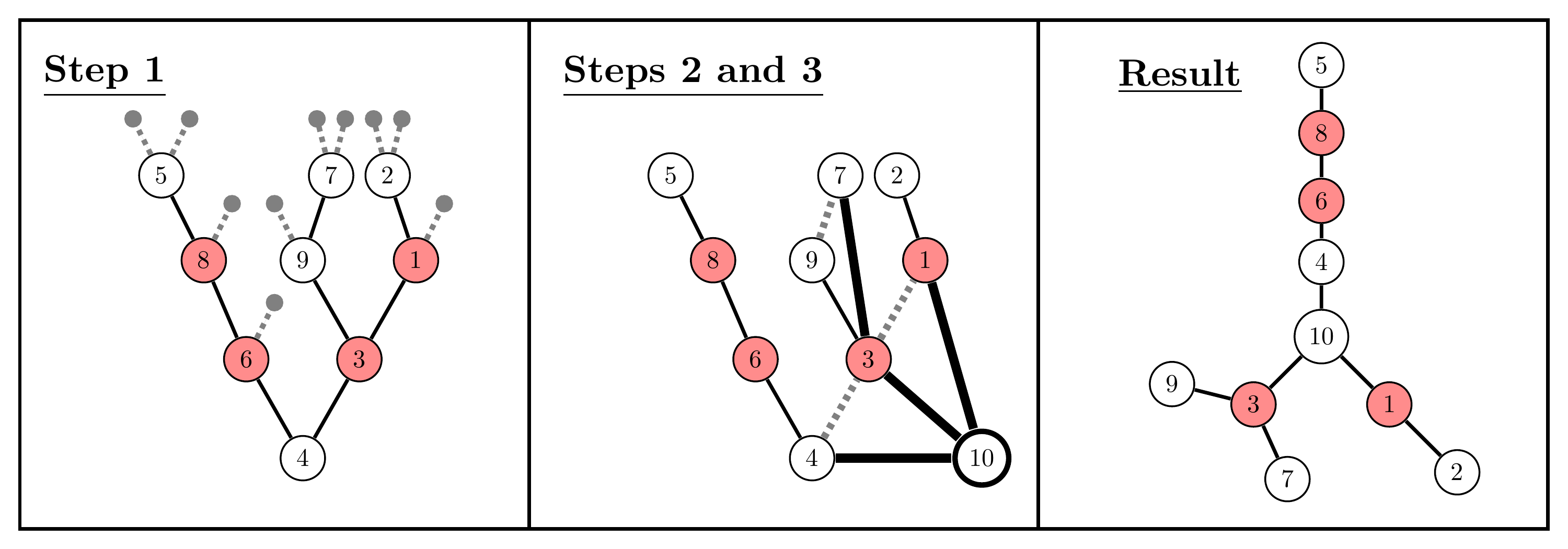}
        \caption[The bijection $\mathscr{U}_{n,k} \isoto \mathscr{V}_{n,k}$]{The three steps of the bijection $\mathscr{U}_{n,k} \isoto \mathscr{V}_{n,k}$. The dashed gray edges are deleted, the bold edges are added. }
        \label{fig_shi_bijections_two}
    \end{figure}

    The third bijection $\mathscr{U}_{n,k} \isoto \mathscr{W}_{n,k}$ is a decorated version of the standard Pr\"{u}fer code bijection. 
    A \emphnew{leaf} of a Cayley tree is a vertex of degree one.  
    Given $V \in \mathscr{V}_{n,k}$, we define the pair $(f,S) \in \mathscr{W}^{(m)}_{n,k}$ as follows. 
    Define $f(1)$ to be the label of the parent of the leaf vertex of $V$ of minimum label. 
    Then delete this leaf, and use the same rule to define $f(2)$, then $f(3)$, and so on to $f(n-1)$ (deleting leaves at each stage). 
    The resulting function $f$ is known as the \emphnew{Pr\"ufer sequence} of the tree $V$, and it can be shown that this encoding is a bijection from Cayley trees to functions $f : [n-1] \to [n+1]$ (see \cite{cite_prufer1918neuerbeweis} or \cite[Prop. 5.3.2]{cite_stanley1999enumerativecombinatorics}). 
    For the subset $S$ we take the set of labels of the marked internal nodes of $V$. 
    It is not hard to show that $\text{Im}(f)$ consists of all internal nodes of $V$, whence $S \subset \text{Im}(f)$. 
    The inverse of this map is given by inverting the Pr\"ufer sequence, and then marking the vertices recorded in $S$. 
    Thus we have the final bijection $\mathscr{V}_{n,k} \isoto \mathscr{W}_{n,k}$.
  \subsection{Result for the \texorpdfstring{$m$-Shi}{m-Shi} arrangement} \label{subsec_mshi_everything_results}
    We now state our results for a general $m$, and give detailed descriptions of the bijections.
    A \emphnew{Cayley $m$-foliage} is a Cayley tree (connected, acyclic, labeled graph) where all edges are assigned one of $m$ colors, represented by the numbers $[0,m-1]$, except for the edges incident to the vertex with maximum label always have color $0$. An \emphnew{internal node} of a Cayley $m$-foliage is a vertex of degree at least two. 
    Let 
    \begin{align}
      \emphnewmath{\mathscr{T}^{(m)}_{n,k}} := \bracketbox{8cm}{$[n]$-decorated $(m+1)$-ary trees such that all internal edges of rank $m$ are descents, and with $k$ free nodes.}, \\
      \emphnewmath{\mathscr{U}^{(m)}_{n,k}} := \bracketbox{8cm}{$(m+1)$-ary trees with $n$ nodes, labeled by $[n]$, such that all internal edges of rank $m$ are descents, with $n-k$ marked nodes, each of which has at least one node among its children of rank $\lt m$.}, \\
      \emphnewmath{\mathscr{V}^{(m)}_{n,k}} := \bracketbox{8cm}{Cayley $m$-foliages with $n+1$ total vertices, with $n-k$ marked internal nodes.}, \\
      \emphnewmath{\mathscr{W}^{(m)}_{n,k}} := \bracketbox{8cm}{Pairs $(f,S)$ where $f : [n-1] \to [mn+1]$ and $S \subset [n]$, with $|S|=n-k$ such that $\text{Im}(f) \cap [(i-1)m+1,im] \neq \varnothing$ for all $i \in S$.}.
    \end{align}
    \begin{theorem} \label{thm_mshi_everything}
      For any $n \geq 1$ and $k \in [n]$, the four sets $\mathscr{T}^{(m)}_{n,k}, \mathscr{U}^{(m)}_{n,k}, \mathscr{V}^{(m)}_{n,k}, \mathscr{W}^{(m)}_{n,k}$ are all in bijection with one another. Consequently,
      \begin{align}
        \# \bracketbox{2.5cm}{$m$-Shi faces in $\mathbb{R}^n$ of dimension $k$} = |\mathscr{T}^{(m)}_{n,k}| = 
      |\mathscr{U}^{(m)}_{n,k}| = 
      |\mathscr{V}^{(m)}_{n,k}| = 
      |\mathscr{W}^{(m)}_{n,k}|.
      \end{align}
    \end{theorem}
    The first equals sign in the above theorem is obtained from \thmref{thm_mshi_bijection}. We prove \thmref{thm_mshi_everything} by describing the sequence of bijections $$\mathscr{T}^{(m)}_{n,k} \isoto \mathscr{U}^{(m)}_{n,k} \isoto \mathscr{V}^{(m)}_{n,k} \isoto \mathscr{W}^{(m)}_{n,k}.$$ 
    One example of the chain of correspondences is shown in \figref{fig_mshieverything}. 
    \begin{figure}[t]
      \includegraphicswide{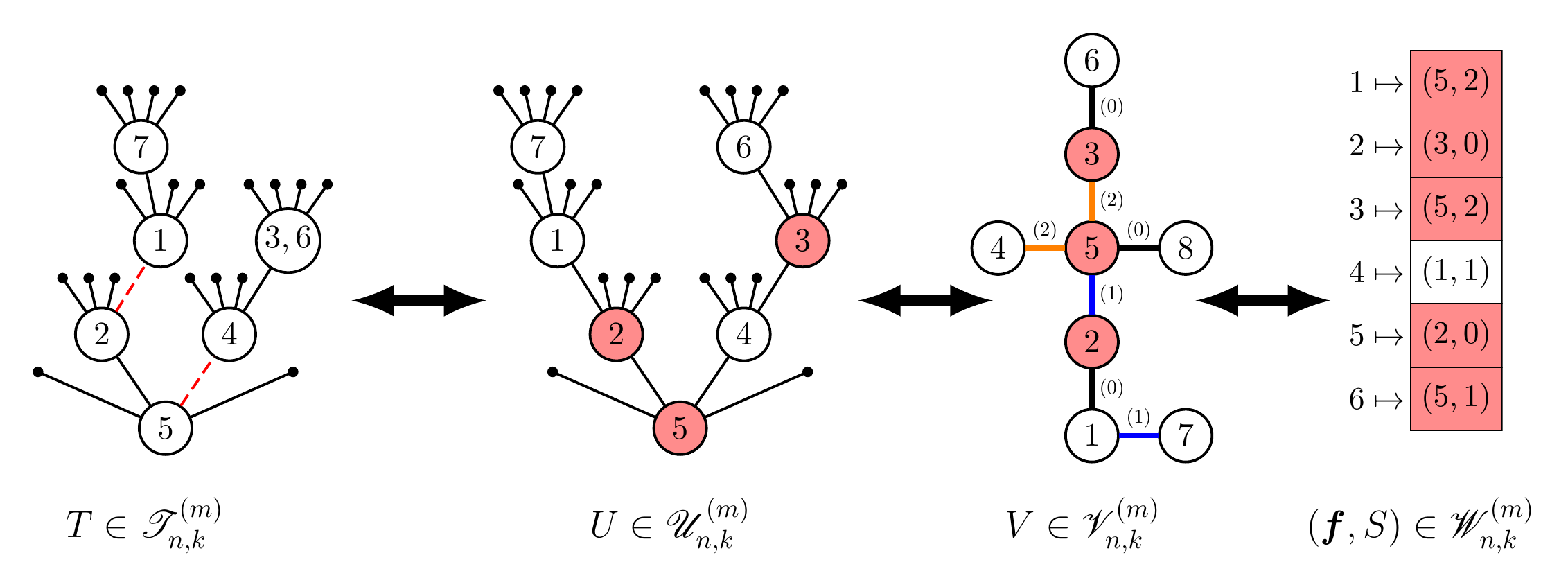}
        \caption[Example of the bijections in \thmref{thm_mshi_everything}]
        {
          An example of the bijections from \thmref{thm_mshi_everything}. 
          In this example, $m=3$, $n=7$, and $k=4$. 
          The edge labels for $V \in \mathscr{V}^{(m)}_{n,k}$ are the number for that edge's color (black is $0$, blue is $1$, orange is $2$).
          The element from $\mathscr{W}^{(m)}_{n,k}$ has $S = \{2,3,5\}$.
          The function $\specialf$ is shown using the convention described in \subsecref{subsec_mshi_to_functions}. 
        } 
        \label{fig_mshieverything}
    \end{figure}
  \subsection{Bijection \texorpdfstring{$\mathscr{T}^{(m)}_{n,k} \isoto \mathscr{U}^{(m)}_{n,k}$}{Script T power (m) sub n,k to Script U power (m) sub n,k}}
    Let $T \in \mathscr{T}^{(m)}_{n,k}$. We apply two local operations (Step 1 and Step 2) to $T$ to obtain its corresponding $U \in \mathscr{U}^{(m)}_{n,k}$. 

    For Step 1, let $v$ be a node of $T$ with children $(T_0,\dots,T_m)$, and with label $S = \{s_1,s_2,\dots,s_p\}$ with $s_1 < s_2< \dots < s_p$. If $|S| > 1$, then replace $v$ by an increasing left-path of singleton marked nodes labeled by $s_1,s_2,\dots,s_p$, each with leaves for their other (rank $\geq 1$) children, except for the last singleton (labeled $s_p$), which is left unmarked and keeps the original children of $v$. If there was a dashed edge in $T$ from $v$ to one of its children, say $T_s$, then the new vertex labeled $s_p$ has the same dashed edge to its $s$th child (which is $T_s$). Let $T'$ be this new tree obtained by applying Step 1. Trees obtained as $T'$ are easily seen to constitute the set of
    \begin{enumerate}
      \item $(m+1)$-ary trees with $n$ nodes labeled by $[n]$,
      \item Such that all internal edges of rank $m$ are descents,
      \item With a choice of solid/dashed for every cadet edge.
      \item With a subset of marked nodes such that every marked node has a node for its leftmost child, and the corresponding edge of rank $0$ is an ascent, meaning the label of its parent is less than the label of its child, and all other children (of rank $\geq 1$) are leaves.
    \end{enumerate}
    It is also clear that Step 1 may be inverted by merging any marked nodes with their sole child node. 

    For Step 2, we apply a local operation to each node of $T'$ with a dashed edge. Let $v$ be such a node, with dashed edge to its child of rank $s$ (we must have $s \in [1,m]$). First convert the dashed edge to a solid edge, and add $v$ to the set of marked nodes. Now there are two cases: if there is at least one node among the children of $v$ of rank $<m$, then we are done with $v$. Otherwise (the only non-leaf child of $v$ is its child of rank $m$), then swap the leftmost child of $v$, a leaf, with the $m$th child of $v$. 

    After applying Step 2 to $T'$ we obtain the resulting tree $U$. It is easy to verify that $U \in \mathscr{U}^{(m)}_{n,k}$. Furthermore, the map has a straightforward inverse: for any marked node $v$ of $U$, if there is a child node of rank $s$ with $s \in [1,m-1]$, simply dash the cadet edge of $v$. Otherwise, by the definition of $\mathscr{U}^{(m)}_{n,k}$, the leftmost child of $v$ must be a node, and all other children of $v$ must be leaves. In this case, if the leftmost edge is a descent, we swap the leftmost and rightmost children of $v$ and dash the resulting internal edge of rank $m$. If the leftmost edge is an ascent, then we undo the Step 1 as described above, that is, merge $v$ with its sole child. Thus we have established the bijection $\mathscr{T}^{(m)}_{n,k} \isoto \mathscr{U}^{(m)}_{n,k}$. 

  \subsection{Bijection \texorpdfstring{$\mathscr{U}^{(m)}_{n,k} \isoto \mathscr{V}^{(m)}_{n,k}$}{Script U power (m) sub n,k to Script V power (m) sub n,k}}
    Let $U \in \mathscr{U}^{(m)}_{n,k}$. To obtain $V \in \mathscr{V}^{(m)}_{n,k}$, we apply three steps: 

    Step 1: For a vertex $v$, let $\text{rightpath}(v)$ denote the unique tuple of nodes $(v_0,\dots,v_{\ell})$ where $v_0=v$, and for all $i \in [\ell]$, $v_{i}$ is the child of rank $m$ of vertex $v_{i-1}$, and $v_{\ell}$ has a leaf for its child of rank $m$.
    Now we apply an operation to each vertex $v$ of $U$. Let $(v^0,\dots,v^m)$ be the children of $v$ in rank order. For each $i \in [0,m-1]$, if $v^i$ is a node, then color the edge from $v$ to $v^i$ with the color $i$, and add more edges of color $i$ between $v$ and every element of $\text{rightpath}(v^i)$. Edges of rank $m$ will soon be deleted, and are not given a color.

    Step 2: Add a node with label $n+1$ and connect it with edges of color $0$ to every node in $\text{rightpath}(u)$ where $u$ is the root of $U$.

    Step 3: Delete all of the original edges of rank $m$, and delete all leaves.

    The set of marked vertices is unchanged throughout. In the end we obtain an element of $\mathscr{V}^{(m)}_{n,k}$. 

    We can describe the inverse of this map explicitly. Let $V \in \mathscr{V}^{(m)}_{n,k}$ be given, and view it as a rooted tree with root $n+1$.

    We modify the children of each vertex $v$ as follows. For each $s \in [0,m-1]$, let $(\kthChildOfColor{s}{1}, \kthChildOfColor{s}{2}, \dots, \kthChildOfColor{s}{\ell})$ be the tuple consisting of all children of $v$ that are connected to $v$ by an edge with color $s$, and with $\text{label}(\kthChildOfColor{s}{1}) > \text{label}(\kthChildOfColor{s}{2}) > \dots > \text{label}(\kthChildOfColor{s}{\ell})$. If $\ell \geq 2$, then for each $i \in [2,\ell]$, remove the edge connecting $\kthChildOfColor{s}{i}$ to $v$, and add an edge of color $m$ connecting $\kthChildOfColor{s}{i-1}$ to $\kthChildOfColor{s}{i}$.

    Once this has been done for each vertex, arrange the children of each vertex $v$ in increasing order by the color of the edge connecting them to $v$. Then add child leaves to $v$ such that the rank of each child node of $v$ is equal to the color of the edge connecting it to $v$ (it is easy to see that this is always possible). Finally delete the vertex $n+1$ and forget all of the edge colors. 

    It is not too hard to see that these maps are inverses of one another.
    Note also that the marked vertices of $U$ must have non-leaf left-children, and therefore the marked vertices of $V$ will have degree at least $2$, and vice versa. This establishes the bijection $\mathscr{U}_{n,k} \isoto \mathscr{V}_{n,k}$.

  \subsection{Bijection \texorpdfstring{$\mathscr{V}^{(m)}_{n,k} \isoto \mathscr{W}^{(m)}_{n,k}$}{Script V power (m) sub n,k to Script W power (m) sub n,k}} \label{subsec_mshi_to_functions}
    Given $V \in \mathscr{V}_{n,k}$, we define the pair $(f,S) \in \mathscr{W}_{n,k}$ as follows. For convenience we instead define a function $\specialf : [n-1] \to [n] \times [m] \cup \{(n+1,0)\}$ which is obviously equivalent to a function $f : [n-1] \to [mn+1]$, as in the definition of $\mathscr{W}^{(m)}_{n,k}$. The condition on $S$ is that for every $s \in S$ there exists $i \in [n-1]$ such that $s$ is the first element of $\specialf(i)$.

    A \emphnew{leaf} of a Cayley $m$-foliage is a vertex of degree $1$. Define $\specialf(1)$ to be the ordered pair $(i,c)$ where $i$ is the label of the leaf vertex of $V$ of minimum label, and $c$ is the color of the edge connecting this leaf to its parent. Then delete this leaf, and use the same rule to define $\specialf(2)$, then $\specialf(3)$, and so on to $\specialf(n-1)$ (deleting leaves at each stage). For the subset $S$ we take the set of labels of the marked internal vertices of $V$. It is easy to see that 
    \begin{align}
      \{a \given a \padtext{is the first coordinate of} \specialf(i) \lpadtext{for some $i \in [n-1]$}\} = \bracketbox{1.3in}{labels of non-leaf vertices of $V$}.
    \end{align}
    It follows that, since $S$ consists of only non-leaf vertices, the set $S$ satisfies its requirement. 

    This simple map is inverted by inverting the associated Pr\"{u}fer sequence and then filling in the colors and the marked internal vertices.
    Explicitly, given $(\specialf,S)$, let $\widehat{f}$ be defined as $\widehat{f}(i) = \widehat{a}$ where $\widehat{a}$ is the first coordinate in the pair given by $\specialf(i)$.
    Then $\widehat{f}$ is the Pr\"{u}fer sequence of a Cayley tree $T$.
    To turn $T$ into a Cayley $m$-foliage, we mark the vertices in the set $S$, and then color the edges of $T$ as follows.
    The color of the edge connecting the leaf of minimum label to its parent is given by the second coordinate in the pair $\specialf(1)$.
    Then we delete this leaf, and use the same rule to obtain another edge color from $\specialf(2)$.
    Continuing in this manner through $\specialf(n-1)$, we establish colors for all but one edge, the one exception being the edge connecting the two remaining vertices after these $n-1$ deletions.
    But it is easy to see that one of these remaining vertices is labeled $n+1$ (every tree has at least two leaves, and we delete the minimum leaf, so we will never delete $n+1$), and we already know that all edges incident to the vertex labeled $n+1$ have color $0$, by the definition of $\mathscr{W}^{(m)}_{n,k}$.
    Thus we have filled in all missing colors.
    It is clear that this map is the inverse of the first, and so we have the final bijection $\mathscr{V}^{(m)}_{n,k} \isoto \mathscr{W}^{(m)}_{n,k}$.

    \figref{fig_3shi_prufers} shows the full bijection from the faces of the $3$-Shi arrangement in $\mathbb{R}^3$ to the marked functions in $\mathscr{W}^{(3)}_{3,k}$ for $k=1,2,3$.

    \begin{figure}[tbp]
      \centering
      \includegraphicswide{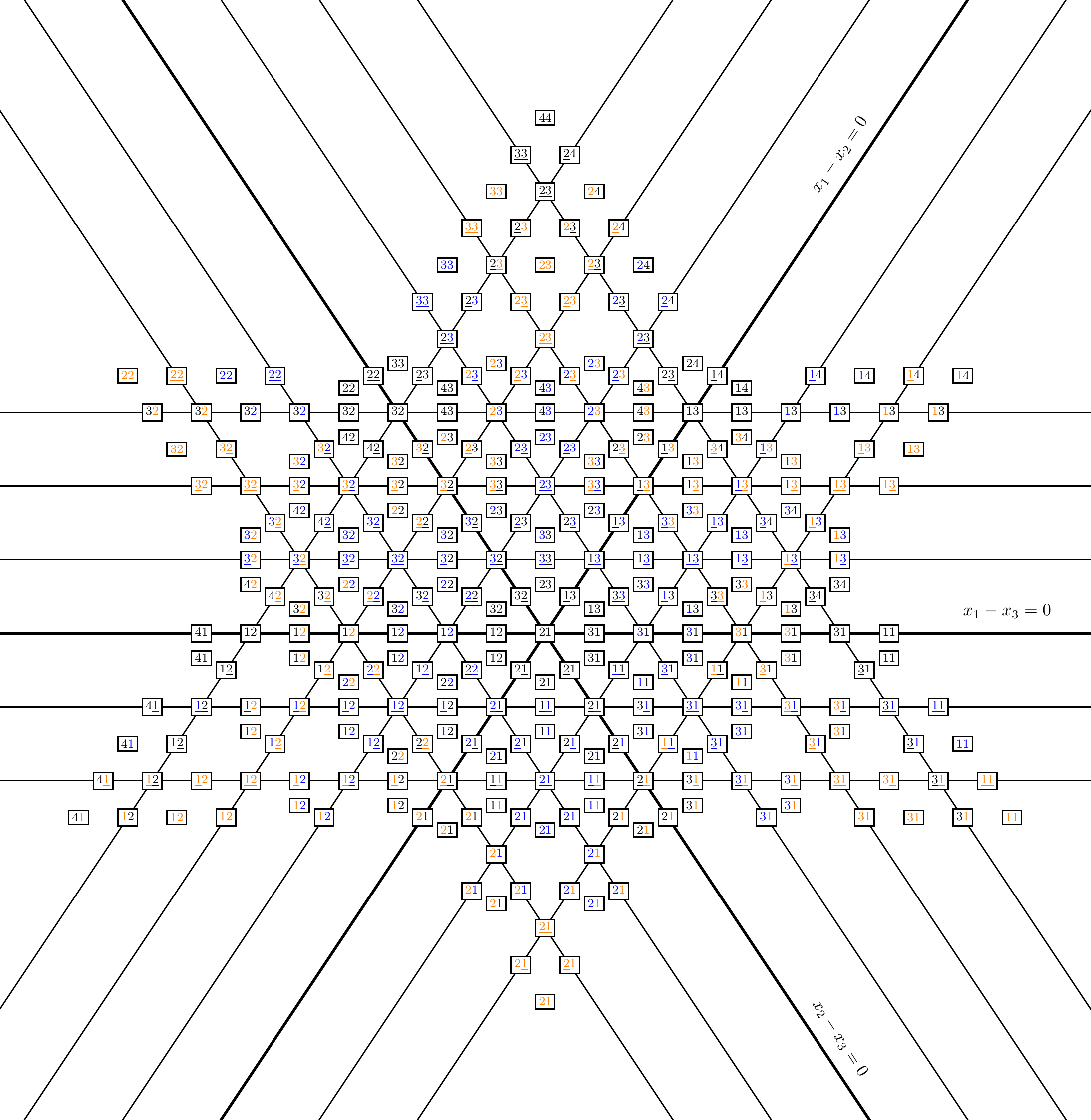}
      \caption[The bijection to marked functions for the $3$-Shi arrangement in $\mathbb{R}^3$]{The bijection between faces of the $3$-Shi arrangement in $\mathbb{R}^3$ and elements of $\displaystyle \bigcup_{k=1}^{3} \mathscr{W}^{(3)}_{3,k}$. We view the elements $(f,S) \in \mathscr{W}^{(3)}_{3,k}$ as the equivalent pair $(\specialf, S)$ as described in \subsecref{subsec_mshi_to_functions}. The function $\specialf$ is displayed as a number sequence, where the number in position $i$ is the first number of the pair $\specialf(i)$, with a color determined by the second number of $\specialf(i)$ (black is $0$, blue is $1$, orange is $2$). The set $S$ is a subset of the numbers appearing as the first element of a pair in $\image(\specialf)$, and is indicated by underlining these numbers in the number sequence. }
      \label{fig_3shi_prufers}
    \end{figure}

\section{Enumerative corollaries} \label{sec_enumerative_corollaries}
  In this section, we derive functional equations for the exponential generating functions of the $m$-Catalan and $m$-Shi faces counted according to their dimension. The results for the classical Catalan and Shi arrangements are obtained by substituting $m=1$. We also give explicit counting formulae for the faces of each dimension. These results are straightforward corollaries of the bijective correspondences obtained in the previous sections (\thmref{thm_mcat_bijection} and \thmref{thm_mshi_everything}).   
  \subsection{\texorpdfstring{$m$-Catalan}{m-Catalan} generating function and counting formula} \label{subsec_mcat_gf}
    We define the exponential generating function 
    \begin{align}
      C_m(x,y) := 1 + \sum_{n \geq 1} \frac{x^n}{n!} \sum_{k = 1}^{n} c_{n,k}^{(m)}  y^{k},
    \end{align}
    where $c_{n,k}^{(m)}$ is the number of $k$-dimensional $m$-Catalan faces in $\mathbb{R}^n$.
    \begin{corollary} \label{cor_cat_gf}
      Let $C$ stand for $C_m(x,y)$. We have the functional equation
      \begin{align}\label{eq_m_catalan_functional_simpler}
        C = 1 + \left(e^{x}-1 \right) \left( (1+y) C^{m+1} - C \right).
      \end{align}
      Furthermore, the number of $k$-dimensional $m$-Catalan faces in $\mathbb{R}^n$ is
      \begin{align} \label{eq_m_catalan_counting_formula}
        c_{n,k}^{(m)} = \sum_{i=k}^{n} S(n,i) (i-1)! \binom{i}{k} \sum_{j=0}^{i-k} (-1)^j \binom{i-k}{j} \binom{i(m+1)-jm}{i-1}.
      \end{align}
    \end{corollary}
    \begin{proof}
    Let 
    \begin{align}
      \mathscr{C}_m = \bigcup_{n \geq 0} \{\text{$[n]$-decorated $(m+1)$-ary trees}\}.
    \end{align}
    By \thmref{thm_mcat_bijection}, $C_m(x,y)$ is the generating function for elements of $\mathscr{C}_m$, with $x$ tracking $n$ (the size of the labelling set), and $y$ tracking the number of free nodes. We now give a recursive description of these trees, which translates into the equation 
      \begin{align} \label{eq_cat_gf_first} 
        C = 1 + y(e^x-1) C^{m+1} + (e^x-1)C(C^{m}-1).
      \end{align}
    Our approach follows the Symbolic Method (see \cite[Part A]{cite_flajolet2009analyticcombinatorics}). First, we claim that for any tree $T \in \mathscr{C}_m$, the number of free nodes of $T$ is equal to the number of nodes of $T$ ``without dashed children,'' that is, the number of nodes of $T$ for which none of their children are captive. This follows from the fact that every node has at most one captive child, and so the number of captive nodes of $T$ is equal to the number of nodes of $T$ that have a captive child (the bijection is $v \mapsto \text{parent}(v)$). Taking complements proves the claim.

    Therefore, we treat $C_m(x,y)$ as the generating function for elements of $\mathscr{C}_m$ with $x$ tracking $n$ and $y$ tracking the number of nodes without dashed children. With this in mind, a tree in $\mathscr{C}_m$ can be either:
      \begin{itemize}
        \item A leaf, contributing the term $1$ in \eref{eq_cat_gf_first}.
        \item A tree whose root is a node with no captive children. Such a tree is built from a nonempty set (the label of the root) and an arbitrary $m$-tuple of trees (the children). Thus, these trees contribute the term $y(e^x-1)C^2$. The $y$ counts the root as a node with no dashed children. 
        \item A tree whose root is a node with a captive child. Such a tree is built from a nonempty set (the label of the root), an arbitrary left-child, and a $m$-tuple of trees (the other children) not all of which are leaves. We need not record which child is captive, since it is always the cadet. These trees contribute the term $(e^x-1)C(C^m-1)$. 
      \end{itemize}
      Summing up the decomposition we obtain \eref{eq_cat_gf_first}. Then equation \eref{eq_m_catalan_functional_simpler} results from \eref{eq_cat_gf_first} by a simple algebraic simplification.

      We now proceed to the counting formula \eref{eq_m_catalan_counting_formula}. It is possible to obtain \eref{eq_m_catalan_counting_formula} directly from the function equation \eref{eq_m_catalan_functional_simpler}, but a significant amount of algebraic manipulation is required. Instead, we will obtain \eref{eq_m_catalan_counting_formula} by slightly more direct tree enumeration.

      First, we remove the labels from our trees and count them separately. A \emphnew{unlabeled $(m+1)$-ary dash tree} is an unlabeled $(m+1)$-ary tree that has two types of internal edges: solid/dashed, where leftmost edges are always solid, and if an edge is dashed then all sibling edges to its right must lead to leaves. We define the ordinary generating function
      \begin{align}
        H_m(x,y) := 1 + \sum_{n \geq 1} x^n \sum_{k=1}^{n} h^{(m)}_{n,k} y^k,
      \end{align}
      where $h^{(m)}_{n,k}$ is the number of unlabeled $(m+1)$-ary dash trees with $n$ nodes, of which $k$ have no dashed children (or, equivalently, with $k$ free nodes). Clearly, an $[n]$-decorated $(m+1)$-ary tree may be represented by a pair $(T,\pi)$ where $T$ is an unlabeled $(m+1)$-ary dash tree and $\pi$ is an ordered set partition of $[n]$ with the number of blocks of $\pi$ equal to the number of nodes of $T$. Since the number of ordered set partitions of $[n]$ with $i$ parts is equal to $S(n,i)i!$, we have 
      \begin{align} \label{eq_m_catalan_dressed_formula}
        c_{n,k}^{(m)} = \#\bracketbox{5cm}{$[n]$-decorated $(m+1)$-ary trees such that $k$ nodes do not have dashed children} = \sum_{i=1}^{n} S(n,i) i! h^{(m)}_{i,k}.
      \end{align}
      It remains to prove that for $i>0$,
      \begin{align} \label{eq_unlabeled_tree_claim} 
        h^{(m)}_{i,k} = \frac{1}{i} \sum_{j=0}^{i} (-1)^j \binom{i-k}{j} \binom{i(m+1)-jm}{i}.
      \end{align}
      We obtain equation \eref{eq_unlabeled_tree_claim} via the Lagrange Inversion Theorem \cite[pg. 66]{cite_flajolet2009analyticcombinatorics}. Let $H$ stand for $H_m(x,y)$. By the same decomposition as the one given above for $C$, we have
      \begin{align} \label{eq_unlabeled_trees_functional} 
        H = 1 + x \left( (y+1) H^{m+1} - H \right).
      \end{align}
      Let $H^{*} = H-1$. It follows that
      \begin{align} \label{eq_unlabeled_trees_transform_functional} 
        H^{*} = x \left( (y+1) (H^{*}+1)^{m+1} - (H^{*}+1) \right).
      \end{align}
      We recognize \eref{eq_unlabeled_trees_transform_functional} as a functional equation of the form $H^{*} = x \phi(H^*)$ where $\phi(t) := (y+1)(t+1)^{m+1}-(t+1)$. Therefore, the Lagrange Inversion Theorem gives
      \begin{align}  \label{eq_h_coefficients}
        h^{(m)}_{i,k}    & =  [x^i] [y^{k}] H^{*} \\
                         & = \frac{1}{i} [t^{i-1}][y^k] \phi(t)^i \\
                         & = \frac{1}{i} [t^{i-1}][y^k] \left((y+1)(t+1)^{m+1} - (t+1) \right)^{i}, \label{eq_h_coefficients_final}
      \end{align}
      where $\emphnewmath{[x^i] F}$ denotes the coefficient on $x^i$ in the formal power series $F$.
      Instead of pursuing an algebraic simplification, we prefer to give a combinatorial interpretation of 
      \begin{align}
        [t^{i-1}][y^k] \left((y+1)(t+1)^{m+1} - (t+1) \right)^{i}. \label{eq_h_coefficients_before_grid}
      \end{align}
      Consider the set $[m+1] \times [i]$, viewed as a grid with $m+1$ rows and $i$ columns. Let $Q_{i,m+1} := \{(X,N) \given X \subset [m+1] \times [i], N \subset [i]\}$. It is not hard to see that the generating function for $Q_{i,m+1}$ with $t$ tracking $|X|$ and $y$ tracking $|N|$ is $\left((y+1)(t+1)^{m+1} \right)^{i}$. 

      Now let $\overline{Q}_{i,m+1}$ consist of the elements $(X,N) \in Q_{i,m+1}$ such that for every $n' \in [i] \smallsetminus N$, the subset $X$ includes at least one cell in the last $m$ rows of column $n'$. We claim that the analogous generating function for $\overline{Q}_{i,m+1}$ is $\left( (y+1)(t+1)^{m+1} - (t+1) \right)^{i}$. 

      We reason as follows: the term $(y+1)(t+1)^{m+1}$ generates all subsets of $[m+1]$ twice, once with weight $y$ and once with weight $y^0$. By subtracting $t+1$, we exclude all subsets of $[m+1]$ counted with weight $y^0$ that do not include any elements except for possibly $1 \in [m+1]$. Taking to the $i$th power fills out a $[m+1] \times [i]$ grid such that each column given the weight $y^0$ contains at least one selected cell in its last $m$ rows. Thus we obtain exactly the elements of $\overline{Q}_{i,m+1}$.

      Hence the expression \eref{eq_h_coefficients_before_grid} counts elements $(X,N) \in \overline{Q}_{i,m+1}$ with $|X| = i-1$ and $|N| = k$. Since these elements are constrained by several ``at least one'' conditions, we will count them by inclusion-exclusion. 

      For $s=1,\dots,i-k$, let $A_s$ be the set of pairs $(X,N) \in Q_{i,m+1}$ with $|X| = i-1$, $|N|=k$, and such that the subset $X$ includes at least one cell in the last $m$ rows of column $n'_s$, where $n'_s$ is the $s$th smallest element of $[i] \smallsetminus N$. Note that the right-hand-side of \eref{eq_h_coefficients_before_grid} equals $|\bigcap_{s=1}^{i-k} A_s|$. 

      Let $A^{\compl}_s$ denote the complement of $A_s$ in $\{(X,N) \in Q_{i,m+1} \given |X|=i-1, |N|=k\}$, that is, the elements of $Q_{i,m+1}$ with $|X|=i-1$, $|N|=k$, and such that the subset $X$ does not include any cells in the last $m$ rows of column $n'_s$, where $n'_s$ is the $s$th smallest element of $[i] \smallsetminus N$. Observe that the intersection of any $j \geq 0$ of the $A^{\compl}_s$ has cardinality $\binom{i}{k} \binom{i(m+1)-jm}{i-1}$: the first coefficient is the choice of $N$, and the second is the choice of $X$.

      Thus starting from \eref{eq_h_coefficients_final}, and applying the above reasoning and inclusion-exclusion, we have
      \begin{align} \label{eq_grid_inclusion_exclusion} 
        [x^i] [y^{k}] H &= \frac{1}{i} \left| \bigcap_{s=1}^{i-k} A_s \right|  \\
        &= \frac{1}{i} \sum_{j=0}^{i-k} (-1)^j \binom{i-k}{j} \left| \bigcap_{s=1}^{j} A^{\compl}_s \right| \\
        &= \frac{1}{i} \sum_{j=0}^{i-k} (-1)^j \binom{i-k}{j} \binom{i}{k} \binom{i(m+1)-jm}{i-1}. \label{eq_grid_inclusion_exclusion_result}
      \end{align}
      Finally, combining \eref{eq_m_catalan_dressed_formula} and \eref{eq_grid_inclusion_exclusion_result} we obtain
      \begin{align} \label{} 
        c_{n,k}^{(m)} = \sum_{i=0}^{n} S(n,i) i! \frac{1}{i} \sum_{j=0}^{i-k} (-1)^j \binom{i-k}{j} \binom{i}{k} \binom{i(m+1)-jm}{i-1}.
      \end{align}
      The formula \eqref{eq_m_catalan_counting_formula} follows by canceling the $\frac{1}{i}$, and factoring the $\binom{i}{k}$ out of the sum on $j$.
    \end{proof}
    \begin{remark}
      In the case $m=1$ in \eref{eq_m_catalan_counting_formula}, one can show that the inner sum on $j$ collapses to $\binom{i+k}{k-1}$, that is,
      \begin{align} \label{} 
        \sum_{j=0}^{i-k} (-1)^j \binom{i-k}{j} \binom{2i-j}{i-1} = \binom{i+k}{k-1}.
      \end{align}
      This follows from the above inclusion-exclusion argument, because for $m=1$ the ``at least one'' conditions become ``exactly one'' conditions, removing the need for inclusion-exclusion entirely. We leave the details to the reader. It follows that the number of $k$-dimensional faces of the classical Catalan arrangement in $\mathbb{R}^n$ is simply
      \begin{align} 
        c_{n,k}^{(1)} = \sum_{i=k}^{n} S(n,i) (i-1)!  \binom{i}{k} \binom{i+k}{k-1}.
      \end{align}
      This formula was first obtained via a finite field method in \cite[Cor. 8.3.2]{cite_athanasiadis1996algebraiccombinatorics}. 
      We have a combinatorial explanation for each term: $S(n,i)(i-1)! \binom{i}{k} \binom{i+k}{k-1}$ is the number of $[n]$-decorated binary trees with $i$ nodes, of which $k$ do not have dashed children. 
      Summing on $i$ gives all the $[n]$-decorated binary trees that correspond to a face of dimension $k$. 
    \end{remark}
    \begin{remark}
      In proving \corref{cor_cat_gf} we have shown via Lagrange Inversion that these two sets are equinumerous:
      \begin{align} \label{eq_grids_and_trees} 
        \bracketbox{2.4in}{Ways of choosing $k$ columns and $i-1$ cells from an $(m+1) \times i$ grid such that the unchosen columns have at least one selected cell in their last $m$ slots.} \cong \bracketbox{1.8in}{Unlabeled $(m+1)$-ary dash trees with $i$ nodes, of which $k$ have no dashed children} \times [i].
      \end{align}
      It is not hard to construct a direct bijection. The columns of the grid play the role of nodes, and the $k$ chosen columns correspond to nodes without dashed children. The chosen cells in each column indicate which children are nodes. Reading the grid from left-to-right, we can construct the tree one node at a time. The extra factor of $i$ is there to allow a cyclic re-ordering of the columns (only one ordering is possible so we do not run out of nodes before columns).

      This is one example of the ubiquitous ``Cycle Lemma'', which arises frequently in the enumeration of trees (see \cite{cite_dershowtiz1990thecycle} for more examples). If one specializes to $k=i-1$ and $m=1$, then the bijection in \eref{eq_grids_and_trees} recovers the well-known formula $\text{Cat}_i = \frac{1}{i} \binom{2i}{i-1}$.

      By including the labels of the trees, we could obtain a bijection between the set of Catalan faces and certain placements of subsets of $[n]$ into grids of various sizes. We leave the details to the reader.

    \end{remark}
  \subsection{\texorpdfstring{$m$-Shi}{m-Shi} generating function and counting formula} \label{subsec_mshi_gf}
    We define the exponential generating function 
    \begin{align} \label{eq_S_def}
      S_m(x,y) := 1 + \sum_{n \geq 1} \frac{x^n}{n!} \sum_{k  = 1}^{n} s_{n,k}^{(m)}  y^{k} ,
    \end{align}
    where $s_{n,k}^{(m)}$ is the number of $k$-dimensional $m$-Shi faces in $\mathbb{R}^n$. 
    \begin{corollary} \label{cor_shi_gf}
      Let $S$ stand for $S_m(x,y)$. We have the functional equation
      \begin{align} \label{eq_mshi_functional}
        S = \exp \left( x(y+1)S^m - x \right). 
      \end{align}
      Furthermore, the number of $k$-dimensional $m$-Shi faces in $\mathbb{R}^n$ is
      \begin{align} \label{eq_mshi_counting}
        s_{n,k}^{(m)} = \binom{n}{k} \sum_{i=0}^{n-k} (-1)^i \binom{n-k}{i} (m(n-i)+1)^{n-1}.
      \end{align}
    \end{corollary}
    \begin{proof}
      We claim that $S$ is the generating function for ``rooted Cayley $m$-forests'', which we define now. A \emphnew{rooted Cayley $m$-tree} is a Cayley tree (see \subsecref{subsec_shi_everything_results}) with one distinguished vertex called the \emphnew{root}, and with every edge assigned one of $m$ colors, represented by the numbers $[0,m-1]$. An \emphnew{elder node} of a rooted Cayley $m$-tree is a vertex with at least one child (that is, $v$ is an elder node if at least one of its neighbors is further from the root than $v$ is). A \emphnew{rooted Cayley $m$-forest} is an unordered set of rooted Cayley $m$-trees. By \thmref{thm_mshi_everything}, $s_{n,k}^{(m)}$ is equal to the number of Cayley $m$-trees with $n+1$ vertices with a marked subset of the vertices $\{1,\dots,n\}$ that have degree $\geq 1$, and such that all edges incident to the vertex $n+1$ have color $0$. It is easy to see that, by deleting the vertex labeled $n+1$, $s^{(m)}_{n,k}$ is also equal to the number of rooted Cayley $m$-forests with $n$ vertices and a marked subset of $n-k$ elder nodes. Thus $S$ is the generating function for rooted Cayley $m$-forests with a marked subset of elder nodes, where $x$ tracks the number of vertices and $y$ tracks the number of unmarked vertices. Let $E = E(x,y)$ be the exponential generating function for rooted Cayley $m$-trees with a marked subset of elder nodes, where $x$ tracks the number of vertices and $y$ tracks the number of unmarked vertices. By the Exponential Formula (see \cite[Cor. 5.1.1]{cite_stanley1999enumerativecombinatorics}), $S = \exp E$. Furthermore, we claim that
      \begin{align} \label{eq_rooted_trees_equation}
        E = x(y+1)(S^m-1) + xy.
      \end{align}
      Indeed, a tree described by $E$ may be decomposed at the root vertex, leaving either:
      \begin{enumerate}
        \item A tree whose root has at least one child. In this case, the root is an elder node, so may or may not be marked. For all $s \in [0,m-1]$, the set of children connected to the root by an edge of color $s$ forms an element appearing as a term in the generating function $S$. Viewing the full set of children as a tuple 
        \begin{align}
          (\text{children of color $0$}, \text{children of color $1$},\dots,\text{children of color $m-1$}),
        \end{align}
        the full (nonempty) set of children forms an element appearing as a term in the generating function $S^m-1$. Thus, we have the term $x(y+1)(S^m-1)$ in \eref{eq_rooted_trees_equation}.
        \item A tree whose root has no children (just a single vertex). This vertex is not an elder node, so it can never be marked. Thus, these trees contribute the term $xy$.
      \end{enumerate}
      Summing up, we obtain \eref{eq_rooted_trees_equation}. By combining \eref{eq_rooted_trees_equation} with $S = \exp E$ we obtain
      \begin{align}
        S = \exp \left( x(y+1)(S^m-1) + xy \right),
      \end{align}
      whence, by slight algebraic simplification, we obtain \eref{eq_mshi_functional}.

      Now we proceed to the counting formula \eref{eq_mshi_counting}. It is possible to obtain this formula directly from the functional equation \eref{eq_mshi_functional}, but we prefer to make use of our bijective results. What follows is a routine application of inclusion-exclusion to the functions in $\mathscr{W}^{(m)}_{n,k}$ (see \subsecref{subsec_mshi_everything_results}). Let $A$ be the set of pairs $(f,S)$ where $f : [n-1] \to [mn+1]$ and $S \subset [n]$. For $j=1,\dots,n-k$, let $A_j$ be the set of pairs $(f,S) \in A$ such that at least one element of $[n-1]$ is mapped to $[(s_j-1)m+1,s_jm]$, where $s_j$ is the $j$th element of $S$. Let $A_j^{\compl}$ denote $A \smallsetminus A_j$, that is, the set of pairs $(f,S) \in A$ such that no elements of $[n-1]$ are mapped to $[(s_j-1)m+1,s_jm]$, where $s_j$ is the $j$th element of $S$. 
     Note that the intersection of any $i\geq0$ of the $A^{\compl}_j$ has cardinality $\binom{n}{n-k}(m(n-i)+1)^{n-1}$, since once we choose $S$ (the $\binom{n}{n-k}$ term) we have merely blocked out $mi$ possible values in the codomain of $f$. In total, 
      \begin{align} \label{} 
        s_{n,k}^{(m)} &= \left| \bigcap_{j=1}^{n-k} A_j \right| \\
        &= \sum_{i=0}^{n-k} 
        (-1)^i 
        \binom{n-k}{i} 
        \left| 
        \bigcap_{j=1}^{i} 
        A^{\compl}_{j}
        \right| \\
        &= \sum_{i=0}^{n-k} (-1)^i \binom{n-k}{i} \binom{n}{n-k} (m(n-i)+1)^{n-1}. \label{eq_mshi_counting_pre}
      \end{align}
      Of course, $\binom{n}{n-k} = \binom{n}{k}$, and the formula \eref{eq_mshi_counting} is obtained by factoring this binomial coefficient out of the sum. 
    \end{proof}
    
    Setting $m=1$ in \eref{eq_mshi_counting}, we recover the formula \eref{eq_athanasiadis_shi_counting} in \thmref{thm_ath_shi} first obtained via a finite field method in \cite[Thm. 8.2.1]{cite_athanasiadis1996algebraiccombinatorics}. 
    \begin{remark}
      For small $k$, the functions in $\mathscr{W}^{(m)}_{n,k}$ yield easy positive formulae. For example, the number of one-dimensional $m$-Shi faces is simply
      \begin{align} \label{eq_shi_onedim}
        s_{n,1}^{(m)} = n!m^{n-1},
      \end{align}
      which was also observed in \remref{rem_mshi_onedim}.
      For $k=2$, there are only a few cases to consider, so it is easily shown that 
      \begin{align} \label{eq_shi_twodim}
        s_{n,2}^{(m)} = \frac{ n! (n-1) (m(n+2)+2)m^{n-2}}{4}.
      \end{align}
      It is not too hard to derive the formula \eref{eq_shi_onedim} directly from the definition of Shi faces, but we do not know a direct proof of the formula \eref{eq_shi_twodim}.
    \end{remark}

\section{Conclusions and questions} \label{sec_conclusions}
  We are hopeful that the bijections introduced in this paper will be useful for future study of the faces of hyperplane arrangements. In this section, we highlight some possible directions for future research.
  \subsection{Structure of the faces}
  So far, we have only used the bijection to find the cardinality of the set of faces of each dimension. Going further, the faces have a natural partial order given by closure-inclusion, and it would be interesting to understand the corresponding partial order on $[n]$-decorated binary trees. Further still, the set of faces of any hyperplane arrangement can be given a semigroup structure, whose linearization defines the so-called \emph{Tits algebra} \cite[Chs. 1 \& 9]{cite_aguiar2017topicsin}. It would be interesting to understand this structure in terms of trees, or other combinatorial objects. For the braid arrangement, the Tits algebra structure is well-understood in terms of interleaving the blocks of ordered set partitions \cite[Sec. 3C]{cite_brown1998randomwalks}. Interesting applications are given in \cite{cite_bidigare1999acombinatorial}. No such combinatorial answer is known for the Catalan or Shi arrangements. The question of understanding this structure for the Shi arrangement was raised explicitly in \cite[Sec. 3F]{cite_brown1998randomwalks}, but remains open.

  \subsection{Other hyperplane arrangements}
  There are many other hyperplane arrangements closely related to the Catalan and Shi arrangements. Of particular interest is the Linial arrangement, consisting of hyperplanes 
  \begin{align}
    x_i - x_j = 1 \; \text{for} \; 1 \leq i < j \leq n.
  \end{align}
  It has been shown \cite[Thm. 8.1]{cite_postnikov2000deformationsof} that the number of regions of the Linial arrangement in $\mathbb{R}^n$ is 
  \begin{align} \label{eq_linial_regions}
    2^{-n} \sum_{k=0}^{n} \binom{n}{k} (k+1)^{n-1}.
  \end{align}
  It was further shown \cite[Thm. 8.2]{cite_postnikov2000deformationsof} that \eref{eq_linial_regions} is the number of ``alternating trees'' with $n$ vertices, and also the number of ``local binary search trees'' with $n$ vertices. It would be interesting to find a similar combinatorial interpretation for the rest of the faces. Currently, there is not even a known formula for the number of Linial faces.
  
  The braid, Catalan, and Shi arrangements belong to a large family known as \emph{truncated affine arrangements}, which are arrangements consisting of the hyperplanes
  \begin{align} \label{} 
    x_i-x_j=s \padtext{for} 1 \leq i < j \leq n,
  \end{align}
  where $s$ runs through an interval of integers $[-a,b]$ for $a,b\geq0$. The regions of truncated affine arrangements have been enumerated in full generality in \cite{cite_postnikov2000deformationsof} and bijections from the regions to certain families of trees have been given in \cite{cite_bernardi2018deformationsof}. The faces have not been enumerated in general, and merit future study.

  In another direction, the Shi arrangement bears a resemblance to the so-called Ish arrangement \cite{cite_armstrong2012shi}, consisting of hyperplanes 
  \begin{align}
    x_i - x_j &= 0 \padtext{for} 1 \leq i < j \leq n, \padtext{and} \\
    x_1 - x_j &= i \padtext{for} 1 \leq i < j \leq n.
  \end{align}
  Although the two arrangements look rather different, they have the same number of regions and the same characteristic polynomial, among many other striking similarities. However, it is evident as early as $n=3$ that they do not have the same number of faces (in $\mathbb{R}^3$ there are $43$ Shi faces and $45$ Ish faces). There is still much to be understood about the faces of these arrangements.

\printbibliography

\end{document}